\documentclass[11pt,reqno]{amsart}
\usepackage{hyperref}
\usepackage{amsmath,amsthm,amscd,amssymb,}
\usepackage{enumerate}

\newcommand*{\mailto}[1]{\href{mailto:#1}{\nolinkurl{#1}}}

\numberwithin{equation}{section}

\newtheorem{re}{Remark}[section]
\newtheorem{prop}{Proposition}[section]
\newtheorem{theo}{Theorem}[section]
\newtheorem{lem}{Lemma}[section]
\newtheorem{col}{Corollary}[section]
\newtheorem{defi}{Definition}
\newcommand{\be}{\begin{equation}}
\newcommand{\ee}{\end{equation}}
\newcommand\bes{\begin{eqnarray}} \newcommand\ees{\end{eqnarray}}
\newcommand{\bess}{\begin{eqnarray*}}
\newcommand{\eess}{\end{eqnarray*}}
\newcommand{\D}{\displaystyle}

\def\XXint#1#2#3{{\setbox0=\hbox{$#1{#2#3}{\int}$}
     \vcenter{\hbox{$#2#3$}}\kern-.5\wd0}}

\numberwithin{equation}{section}

\begin{document}

\title[global existence and incompressible limit for flow of liquid crystals]{Global existence and incompressible limit in critical spaces for compressible flow of liquid crystals}

\author[Q. Bie]{Qunyi Bie}
\address{College of Science, China Three Gorges University, Yichang 443002, PR China}
\email{\mailto{qybie@126.com}}

\author[H. Cui]{Haibo Cui}
\address{School of Mathematical Sciences, Huaqiao University, Quanzhou 362021,
 PR China}
\email{\mailto{cuihaibo2000@163.com}}
%\urladdr{\url{http://www.mat.univie.ac.at/~grunert/}}

\author[Q. Wang]{Qiru Wang$^{*}$}
\address{School of Mathematics and Computational Science, Sun Yat-Sen University, Guangzhou 510275, PR China}
\email{\mailto{mcswqr@mail.sysu.edu.cn}}
%\urladdr{\url{http://www.mat.univie.ac.at/~gerald/}}
\author[Z.-A. Yao]{Zheng-an Yao}
\address{School of Mathematics and Computational Science, Sun Yat-Sen University, Guangzhou 510275, PR China}
\email{\mailto{mcsyao@mail.sysu.edu.cn}}

%\thanks{Research supported by the Austrian Science Fund (FWF) under Grant
%No.\ Y330.}
%\thanks{Nonlinearity {\bf 22}, 1431--1457 (2009)}

\thanks{Research Supported by the NNSF of China (nos.\,11271379, 11271381), the National Basic Research Program of China (973 Program) (Grant No.\,2010CB808002) and the Science Foundation of China Three Gorges University (no. KJ2013B030).}
\thanks{$^*$ Corresponding author, email: mcswqr@mail.sysu.edu.cn, tel.: 86-20-84037100, fax: 86-20-84037978.
}

\begin{abstract}
The Cauchy problem for  the compressible flow of nematic liquid crystals in the framework of critical spaces is considered. We first establish the existence and uniqueness of global solutions provided that the initial data are close to some equilibrium states. This result improves the work by Hu and Wu [SIAM J. Math. Anal., 45 (2013), pp. 2678-2699] through relaxing the regularity requirement of the initial data in terms of the director field. We then  consider the incompressible limit problem for ill prepared initial data. We prove that as the Mach number tends to zero, the global solution to the compressible flow of liquid crystals converges to the solution to the corresponding incompressible model in some function spaces. Moreover, the accurate converge rates are obtained.

\end{abstract}

\maketitle
\tableofcontents
\section{Introduction and main results}\setcounter{equation}{0}\setcounter{section}{1}\indent
In this paper we consider the global well-posedness and incompressible limit to the following compressible  flow of nematic liquid crystals in critical spaces:
 \begin{equation}\label{1.1}
 \left\{\begin{array}{ll}\medskip\displaystyle\partial_t \rho+{\rm div}(\rho {\bf u})
=0,\\
 \medskip\displaystyle\partial_t(\rho{\bf u})+{\rm div}(\rho{\bf u}\otimes{\bf u})-\mu\Delta{\bf u}-(\mu+\lambda)\nabla{\rm div}{\bf u}+\nabla P(\rho)\\
 \medskip\quad=-\xi{\rm div}{\Big(}\nabla{\bf d}\odot\nabla{\bf d}-\frac{1}{2}|\nabla{\bf d}|^2{\bf I}{\Big)},\\
 \displaystyle\partial_t{\bf d}+{\bf u}\cdot\nabla{\bf d}=\theta(\Delta{\bf d}+|\nabla{\bf d}|^2{\bf d}),
 \end{array}
 \right.
 \end{equation}
where $\rho\in \mathbb{R}$ is the density function of the fluid, ${\bf u}\in \mathbb{R}^N (N\geq 2)$ is the velocity, and ${\bf d}\in \mathbb {S}^{N-1}$ represents the director field for the averaged macroscopic molecular orientations. The scalar function $P\in\mathbb{R}$ is the pressure, which is an increasing and convex function in $\rho$. We denote by $\lambda$ and $\mu$ the two Lam\'{e} coefficients of the fluid, which are constant and satisfy $\mu>0$ and $\nu:=\lambda+2\mu>0$. Such a condition ensures ellipticity for the operator $\mu\Delta+(\lambda+\mu)\nabla{\rm div}$ and is satisfied in the physical cases. The constants $\xi>0$, $\theta>0$ stand for the competition between the kinetic energy and the potential energy, and the microscopic elastic relaxation time (or the Debroah number) for the molecular orientation field, respectively. The symbol $\otimes$ denotes the Kronecker tensor product such that ${\bf u}\otimes{\bf u}=({\bf u}_i{\bf u}_j)_{1\leq i,j \leq N}$
and the term $\nabla{\bf d}\odot\nabla{\bf d}$ denotes a matrix whose $(i,j)-$th entry is $\partial_{x_i}{\bf d}\cdot\partial_{x_j}{\bf d}~(1\leq i,j\leq N)$. ${\bf I}$ is the $N\times N$ identity matrix. To complete the system \eqref{1.1}, the initial data are given by
\be\label{1.11}
\rho|_{t=0}=\rho_0(x), ~~{\bf u}|_{t=0}={\bf u}_0(x), ~~{\bf d}|_{t=0}={\bf d}_0(x),\,\,  {\rm with}\,\,\,  {\bf d}_0\in\mathbb{S}^{N-1}.
\ee

The hydrodynamic theory of liquid crystals was first proposed by Ericksen \cite{ericksen1961conservation, ericksen1962hydrostatic} and Leslie \cite{leslie1968some} in 1960s. In 1989, Lin \cite{lin1989nonlinear} first derived a simplified Ericksen-Leslie equation modeling
liquid crystal flows when the fluid is an incompressible and viscous fluid. Subsequently, Lin and Liu \cite{lin1995nonparabolic} showed the global existence of weak solutions and smooth solutions for the approximation system. Recently, Hong \cite{hong2011global} and Lin et al.\,\cite{lin2010liquid} showed independently the global existence of a weak solution of an incompressible model of system \eqref{1.1}
in two-dimensional space. Furthemore, in \cite{lin2010liquid}, the regularity of solutions except for a countable set of singularities whose projection on the time axis is a finite set had been obtained. Very recently, in dimension three, Lin and Wang \cite{lin2015global} have proved the existence of global weak solutions  under the assumption that ${\bf d}_0\in \mathbb{S}_+^2$ by developing some new compactness arguments, here $\mathbb{S}_+^2$ is the upper hemisphere.

 As for the compressible case, Huang et al.\,\cite{huang2012strong} proved the local existence of unique strong solution of \eqref{1.1} provided that the initial data $\rho_0, {\bf u}_0, {\bf d}_0$ are sufficiently regular and satisfy a natural compatibility condition. And a criterion for possible breakdown of such a local strong solution at finite time was given in terms of blow up of the quantities $\|\rho\|_{L_t^\infty L_x^\infty}$ and $\|\nabla{\bf d}\|_{L_t^3L_x^\infty}$. In \cite{huang2012blow}, an alternative blow-up criterion was derived in terms of the temporal integral of both the $L^\infty$-norm of the deformation
tensor $\mathcal{D}{\bf u}$ and the square of the $L^\infty$-norm of $\nabla{\bf d}$. In terms of the global well-posedness, results in one dimensional space have been obtained in \cite{ding2012compressible,ding2011weak}. In two dimensions, Jiang et al.\,\cite{jiang2014global} established the global existence of weak solutions under the small initial energy. In two or three dimensions, if some component of initial direction field is small, Jiang et al.\,\cite{jiang2013on} established the global existence of weak solutions to the initial-boundary problem with large initial energy and without any smallness condition on the initial density and velocity.
Recently, Lin et al.\,\cite{lin2014global}
 established the existence of global weak solutions in three-dimensional space, provided the initial orientational director field ${\bf d}_0$ lies in the hemisphere $\mathbb{S}_+^2$.

The low Mach number limit of the system \eqref{1.1}-\eqref{1.11} has also been studied recently. Hao and Liu\,\cite{hao2013incompressible} investigated  the so-called incompressible limit (i.e., the low Mach number limit) for
solutions in the whole space $\mathbb{R}^N (N = 2, 3)$ and a bounded domain of $\mathbb{R}^N (N = 2, 3)$ with Dirichlet boundary conditions.
Ding et al.\,\cite{ding2013incompressible} studied the incompressible limit with periodic boundary conditions in $\mathbb{R}^N (N = 2, 3)$.
Wang and Yu\,\cite{wang2014incompressible} proved the incompressible limit for weak solutions in a bounded domain. For more about the incompressible limit problem, one can refer to \cite{jiang2010incompressible, lions1998incompressible,ou2011low} and the references therein.

Let us mention that all of the above results were performed in the framework of Sobolev spaces. Inspired by\,\cite{danchin2000global} for the compressible Navier-Stokes equations, it is natural to study the system \eqref{1.1}-\eqref{1.11} in  critical Besov spaces. We observe  that the system \eqref{1.1} is invariant by the transformation
\be\label{1.9}
\tilde{\rho}=\rho(l^2t, lx),~~\tilde{\bf u}=l{\bf u}(l^2t, lx),~~\tilde{\bf d}={\bf d}(l^2t, lx)
\ee
 up to a change of the pressure law $\tilde{P}=l^2P$. A critical space is a space in which the norm is invariant under the scaling
$$
(\tilde{e},\tilde{\bf f}, \tilde{\bf g})(x)=(e(lx), l{\bf f}(lx), {\bf g}(lx)).
$$
 Very recently, in the case $N=3$, Hu and Wu \cite{hu2013global} studied the global strong solution to \eqref{1.1}-\eqref{1.11} in critical Besov spaces provided that the initial datum is close to an equilibrium state $(1,0,\hat{\bf d})$ with a constant vector $\hat{\bf d}\in \mathbb{S}^2$. More precisely, there exist two positive constants $\eta_0$ and $\Gamma_0$ such that if $\rho_0-1\in\tilde{B}_\nu^{\frac{1}{2},\infty}(\mathbb{R}^3)$, ${\bf u}_0\in \dot{B}_{2,1}^{\frac{1}{2}}(\mathbb{R}^3)$, and ${\bf d}_0-\hat{\bf d}\in\tilde{B}_\nu^{\frac{1}{2},\infty}(\mathbb{R}^3)$ satisfy
\be\label{1.6}
\|\rho_0-1\|_{\tilde{B}_\nu^{\frac{1}{2},\infty}}+\|{\bf u}_0\|_{\dot{B}_{2,1}^{\frac{1}{2}}}+\|{\bf d}_0-\hat{\bf d}\|_{\tilde{B}_\nu^{\frac{1}{2},\infty}}\leq \eta_0,
\ee
then system \eqref{1.1}-\eqref{1.11} has a unique global strong solution $(\rho, {\bf u}, {\bf d})$ with
\be\label{1.7}
\left.\begin{array}{ll}\medskip\D
\rho-1\in L^1(\mathbb{R}^+; \tilde{B}_\nu^{\frac{3}{2},1})\cap C(\mathbb{R}^+; \tilde{B}_\nu^{\frac{3}{2},\infty}),\\
\medskip\D  {\bf u}\in \left( L^1(\mathbb{R}^+; \dot{B}_{2,1}^{\frac{5}{2}})\cap C(\mathbb{R}^+; \dot{B}_{2,1}^{\frac{1}{2}})\right)^{3}, \\
\D {\bf d}-\hat{\bf d}\in\left(L^1(\mathbb{R}^+; \tilde{B}_\nu^{\frac{7}{2},\infty})\cap C(\mathbb{R}^+; \tilde{B}_\nu^{\frac{3}{2},\infty})\right)^3
\end{array}
\right.
\ee
satifying
\be\label{1.8}
\begin{split}
&\|\rho-1\|_{L^\infty(\mathbb{R}^+; \tilde{B}_\nu^{\frac{3}{2},\infty})}
+\|{\bf u}\|_{L^\infty(\mathbb{R}^+; \dot{B}_{2,1}^{\frac{1}{2}})}
+\|{\bf d}-\hat{\bf d}\|_{L^\infty(\mathbb{R}^+; \tilde{B}_\nu^{\frac{3}{2},\infty})}\\[2mm]
&\quad+\|\rho-1\|_{L^1(\mathbb{R}^+; \tilde{B}_\nu^{\frac{3}{2},1})}
+\|{\bf u}\|_{L^\infty(\mathbb{R}^+; \dot{B}_{2,1}^{\frac{5}{2}})}
+\|{\bf d}-\hat{\bf d}\|_{L^\infty(\mathbb{R}^+; \tilde{B}_\nu^{\frac{7}{2},\infty})}\leq \Gamma_0\eta_0.
\end{split}
\ee
Here $\dot{B}_{p,r}^s$ and $\tilde{B}_\nu^{s_1, r}$ denote  the homogeneous Besov space and  hybrid Besov space, respectively. We are going to explain these notations in Section 2.

The purpose of this paper includes the following two aspects:

 On one hand, we establish global strong solutions to the Cauchy problem of \eqref{1.1}-\eqref{1.11} in critical Besov spaces with initial data close to a stable equilibrium.  From \cite{hu2013global}, when $N=3$ and \eqref{1.6} holds true, the system \eqref{1.1}-\eqref{1.11} has a unique global strong solution $(\rho-1, {\bf u}, {\bf d}-\hat{\bf d})$ satisfying \eqref{1.8}. Concerning the global well-posedness with respect to ${\bf d}$, we carry out in the framework of critical Besov space $\dot{B}_{2,1}^{\frac{3}{2}}$ (if $N=3$) but not the hybrid Besov space $\tilde{B}_\nu^{\frac{3}{2},\infty}$ in \cite{hu2013global}.  Since $\tilde{B}_\nu^{\frac{3}{2},\infty}
\approx\dot{B}_{2,1}^{\frac{1}{2}}\cap\dot{B}_{2,1}^{\frac{3}{2}}
\subset\dot{B}_{2,1}^{\frac{3}{2}}$, the regularity requirement of the initial data in terms of the director field ${\bf d}_0-\hat{\bf d}$ is relaxed. The key point is that, different from the estimate of ${\bf d}$ in hybrid Besov space in \cite[Proposition 4.1]{hu2013global}, we make use of  the estimate of ${\bf d}$ in the homogeneous Besov space $\dot{B}_{2,1}^{\frac{3}{2}}$ (see Proposition \ref{pr3.2} below when $N=3$). In addition,  the global estimates of a linear hyperbolic-parabolic system given by Danchin \cite{danchin2000global} (see also \cite[Chapter 10]{bahouri2011fourier}) play an important role.

On the other hand, we give the rigorous justification of the convergence of the incompressible limit for global strong solutions to the compressible equations of liquid crystals when the initial data are ill prepared and small in a critical space. Meanwhile, the accurate converge rates are obtained. Our proof follows the ideas of Danchin \cite{danchin2002zero} and the key point is to use some dispersive inequalities for the wave equation: the so-called Strichartz estimates (see e.g., \cite{ginibre1995generalized, keel1998endpoint, strichartz1977restrictions} and the references therein).

We would like to point out that \cite{desjardins1999low} is the first paper devoted to the incompressible limit problem where Strichartz estimates have been used. In the spirit of \cite{desjardins1999low}, Danchin \cite{danchin2002zero} studied the zero Mach number limit in critical spaces for barotropic compressible Navier-Stokes equations. Fang and Zi \cite{fang2014incompressible} investigated the incompressible limit of Oldroyd-B fluids in the whole space.

Before presenting the main statements of this paper, we introduce the following function space:
\be\nonumber
\begin{split}
\mathfrak{B}_\nu^s(T)={\Big\{}(e, {\bf f}, {\bf g})&\in {\Big(}L^1(0, T; \tilde{B}_\nu^{s,1})\cap C([0,T];\tilde{B}_\nu^{s,\infty}){\Big)}\\
&\quad\times{\Big(}L^1(0, T; \dot{B}_{2,1}^{s+1})\cap C([0, T];\dot{B}_{2,1}^{s-1}){\Big)}^N\\
&\quad\times{\Big(}L^1(0, T; \dot{B}_{2,1}^{s+2})\cap C([0, T];\dot{B}_{2,1}^{s}){\Big)}^N{\Big\}}
\end{split}
\ee
and
\be\nonumber
\begin{split}
\|(e, {\bf f}, {\bf g})\|_{\mathfrak{B}_\nu^s(T)}&=\|e\|_{L_T^\infty(\tilde{B}_\nu^{s,\infty})}+\|{\bf f}\|_{L_T^\infty(\dot{B}_{2,1}^{s-1})}+\|{\bf g}\|_{L_T^\infty(\dot{B}_{2,1}^{s})}\\[2mm]
&\quad+\nu\|e\|_{L_T^1(\tilde{B}_\nu^{s,1})}+\underline{\nu}\|{\bf f}\|_{L_T^1(\dot{B}_{2,1}^{s+1})}+\theta\|{\bf g}\|_{L_T^1(\dot{B}_{2,1}^{s+2})}.
\end{split}
\ee
Here $T>0$, $s\in \mathbb{R}$, $\nu:=\lambda+2\mu$\,\,and\,\,$\underline{\nu}:=\min(\mu, \lambda+2\mu)$. We use the notation $\mathfrak{B}_\nu^s$ if $T=+\infty$ by changing the interval $[0, T]$ into $[0, \infty)$ in the definition above.

Our first result of this paper reads as follows.
 \begin{theo}\label{th1.1}
 Let $\hat{{\bf d}}\in \mathbb{R}^N$ be an arbitrary constant unit vector, and assume that $P^\prime(1)=1$.  There exist two positive constants $\eta$ and $\Gamma$ such that if $\rho_0-1\in\tilde{B}_\nu^{\frac{N}{2},\infty}, ~{\bf u}_0\in\dot{B}_{2,1}^{\frac{N}{2}-1}$ and ~${\bf d}_0-\hat{\bf d}\in\dot{B}_{2,1}^{\frac{N}{2}}$ satisfy
 \be
 \label{1.10}
 \|\rho_0-1\|_{\tilde{B}_\nu^{\frac{N}{2},\infty}}+\|{\bf u}_0\|_{\dot{B}_{2,1}^{\frac{N}{2}-1}}+\|{\bf d}_0-\hat{{\bf d}}\|_{\dot{B}_{2,1}^{\frac{N}{2}}}\leq \eta,
 \ee
then the following results hold true:

{\rm (i)}\, System \eqref{1.1}-\eqref{1.11} has a global strong solution $(\rho, {\bf u}, {\bf d})$ with $(\rho-1, {\bf u}, {\bf d}-\hat{{\bf d}})$ in $\mathfrak{B}_\nu^{\frac{N}{2}}$ satisfying
\be\label{1.12}
\|(\rho-1, {\bf u}, {\bf d}-\hat{{\bf d}})\|_{\mathfrak{B}_\nu^{\frac{N}{2}}}\leq
\Gamma{\Big(} \|\rho_0-1\|_{\tilde{B}_\nu^{\frac{N}{2},\infty}}+\|{\bf u}_0\|_{\dot{B}_{2,1}^{\frac{N}{2}-1}}+\|{\bf d}_0-\hat{{\bf d}}\|_{\dot{B}_{2,1}^{\frac{N}{2}}}{\Big)}.
\ee

{\rm (ii)}~Uniqueness holds in $\mathfrak{B}_\nu^{\frac{N}{2}}$ if $N\geq 3$ and in
$\mathfrak{B}_\nu^{1}\cap\mathfrak{B}_\nu^{s}\,(1<s<2)$ if $N=2$.
 \end{theo}

 \begin{re}\label{re1.1} {\rm Since $P(\rho)$ is an increasing convex function of $\rho$, we assume $P^\prime(1)=1$ for simplicity. The general barotropic case $P^\prime(1)>0$ can be verified by a slight modification of the argument below.}
\end{re}

Recall that the Mach number for the compressible flow \eqref{1.1} is defined as:
$$
M=\frac{|{\bf u}|}{\sqrt{P^\prime(\rho)}}.
$$
Thus, letting $M$ approach zero, we hope that $\rho, {\bf d}$ keep a typical size $1$, and ${\bf u}$ is of order $\epsilon$, where $\epsilon\in(0,1)$ is a small parameter. As in \cite{lions1998mathematical}, we scale $\rho, {\bf u}$ and ${\bf d}$ in the following way:
$$
\rho=\rho^\epsilon(\epsilon t, x), \,\,\,\,  {\bf u}=\epsilon {\bf u}^\epsilon(\epsilon t, x),\,\,\,\,
{\bf d}={\bf d}^\epsilon(\epsilon t, x),
$$
and we take the viscosity coefficients as:
$$
\mu=\epsilon\mu^\epsilon, \,\,\,\, \lambda=\epsilon \lambda^\epsilon, \,\,\,\,
\xi=\epsilon^2\xi^\epsilon, \,\,\,\, \theta=\epsilon\theta^\epsilon.
$$
Under this scaling, system \eqref{1.1}-\eqref{1.11} becomes
\begin{equation}\label{1.2}
 \left\{\begin{array}{ll}\medskip\displaystyle\partial_t \rho^\epsilon+{\rm div}(\rho^\epsilon {\bf u}^\epsilon)
=0,\\
 \medskip\displaystyle\partial_t(\rho^\epsilon{\bf u}^\epsilon)+{\rm div}(\rho^\epsilon{\bf u}^\epsilon\otimes{\bf u}^\epsilon)-\mu^\epsilon\Delta{\bf u}^\epsilon-(\mu^\epsilon+\lambda^\epsilon)\nabla{\rm div}{\bf u}^\epsilon+\frac{\nabla P(\rho^\epsilon)}{\epsilon^2}\\
 \medskip\quad=-\xi^\epsilon{\rm div}{\Big(}\nabla{\bf d}^\epsilon\odot\nabla{\bf d}^\epsilon-\frac{1}{2}|\nabla{\bf d}^\epsilon|^2{\bf I}{\Big)},\\[2mm]
 \medskip\displaystyle\partial_t{\bf d}^\epsilon+{\bf u}^\epsilon\cdot\nabla{\bf d}^\epsilon=\theta^\epsilon(\Delta{\bf d}^\epsilon+|\nabla{\bf d}^\epsilon|^2{\bf d}^\epsilon),\\[2mm]
 \D (\rho^\epsilon, {\bf u}^\epsilon, {\bf d}^\epsilon)|_{t=0}
 =(\rho_0^\epsilon, {\bf u}_0^\epsilon, {\bf d}_0^\epsilon).
 \end{array}
 \right.
 \end{equation}
For the simplicity of notations and presentation, we shall assume that $\mu^\epsilon, \lambda^\epsilon, \xi^\epsilon$ and $\theta^\epsilon$ are constants, independent of $\epsilon$, and still denote them as $\mu, \lambda, \xi$ and $\theta$ with an abuse of notations.

Formally, we get by letting $\epsilon\rightarrow 0$ the following incompressible model
\begin{equation}\label{1.3}
 \left\{\begin{array}{ll}
 \medskip\partial_t{\bf u}+{\bf u}\cdot\nabla{\bf u}+\nabla\pi=\mu\Delta{\bf u}-\xi{\rm div}(\nabla{\bf d}\odot\nabla{\bf d}),\\
 \medskip\displaystyle\partial_t{\bf d}+{\bf u}\cdot\nabla{\bf d}=\theta(\Delta{\bf d}+|\nabla{\bf d}|^2{\bf d}),\\
 \medskip\displaystyle {\rm div}{\bf u}=0,\\
 ({\bf u}, {\bf d})|_{t=0}=({\bf u}_0, {\bf d}_0).
 \end{array}
 \right.
 \end{equation}
Thus, roughly speaking, it is also reasonable to expect from the mathematical point of view that the global strong solutions to \eqref{1.2} converge in suitable functional spaces to the global strong solutions of \eqref{1.3} as $\epsilon\rightarrow 0$, and the hydrostatic pressure $\pi$ in the first equation of \eqref{1.3} is the limit of $\frac{ P(\rho^\epsilon)}{\epsilon^2}-\frac{\xi^\epsilon}{2}|\nabla{\bf d}^\epsilon|^2$. Our second goal is devoted to the rigorous justification of the convergence of the above incompressible
limit in the whole space. We remark that the existence of global strong solutions to the incompressible flow of liquid crystals \eqref{1.3} in critical Besov space was established in Xu et al.\,\cite{Xu2014well-posedness}.

As in \cite{danchin2002zero}, we want to consider so-called {\it ill prepared data} of the form $\rho_0^\epsilon=1+\epsilon b_0^\epsilon$, ${\bf u}_0^\epsilon$ and ${\bf d}_0^\epsilon$,  where $(b_0^\epsilon, {\bf u}_0^\epsilon, {\bf d}_0^\epsilon)$ are bounded in a sense that will be specified later on.  Setting $\rho^\epsilon=1+\epsilon b^\epsilon$, it is easy to check that $(b^\epsilon, {\bf u}^\epsilon, {\bf d}^\epsilon)$ satisfies
\begin{equation}\label{1.4}
 \left\{\begin{array}{ll}\medskip\displaystyle\partial_t b^\epsilon+\frac{{\rm div}{\bf u}^\epsilon}{\epsilon}=-{\rm div}(b^\epsilon{\bf u}^\epsilon),\\[2mm]
 \medskip\displaystyle\partial_t{\bf u}^\epsilon+{\bf u}^\epsilon\cdot\nabla{\bf u}^\epsilon-\frac{\mu\Delta{\bf u}^\epsilon+(\mu+\lambda)\nabla{\rm div}{\bf u}^\epsilon}{1+\epsilon b^\epsilon}+\frac{P^\prime(1+\epsilon b^\epsilon)}{1+\epsilon b^\epsilon}\frac{\nabla b^\epsilon}{\epsilon}\\[2mm]
 \medskip\quad\D=\frac{-\xi}{1+\epsilon b^\epsilon}{\rm div}{\Big(}\nabla{\bf d}^\epsilon\odot\nabla{\bf d}^\epsilon-\frac{1}{2}|\nabla{\bf d}^\epsilon|^2{\bf I}{\Big)},\\[2mm]
 \displaystyle\partial_t{\bf d}^\epsilon+{\bf u}^\epsilon\cdot\nabla{\bf d}^\epsilon=\theta(\Delta{\bf d}^\epsilon+|\nabla{\bf d}^\epsilon|^2{\bf d}^\epsilon),\\[2mm]
 \D (b^\epsilon, {\bf u}^\epsilon, {\bf d}^\epsilon)|_{t=0}
 =(b_0^\epsilon, {\bf u}_0^\epsilon, {\bf d}_0^\epsilon).
 \end{array}
 \right.
 \end{equation}

Our second result of the paper can be stated as follows.
\begin{theo}\label{th1.2}
Denote $\mathcal{Q}:=\nabla\Delta^{-1}{\rm div}$ and $\mathcal{P}:=I-\mathcal{Q}$. Let $\hat{{\bf d}}\in \mathbb{R}^N$ be an arbitrary constant unit vector. There exist two positive constants $c$ and $M$ such that if\, $b_0^\epsilon\in\tilde{B}_\nu^{\frac{N}{2},\infty}, ~{\bf u}_0^\epsilon\in\dot{B}_{2,1}^{\frac{N}{2}-1}$ and ~${\bf d}_0^\epsilon-\hat{\bf d}\in\dot{B}_{2,1}^{\frac{N}{2}}$ satisfy {\rm(}for all $0<\epsilon\leq \epsilon_0${\rm)}
\be\label{1.5}
C_0^{\epsilon\nu}:=\|b_0^\epsilon\|_{\dot{B}_{2,1}^{\frac{N}{2}-1}}+\epsilon\nu
\|b_0^\epsilon\|_{\dot{B}_{2,1}^{\frac{N}{2}}}+\|{\bf u}_0^\epsilon\|_{\dot{B}_{2,1}^{\frac{N}{2}-1}}
+\|{\bf d}_0^\epsilon-\hat{\bf d}\|_{\dot{B}_{2,1}^{\frac{N}{2}}}\leq c,
\ee
then the following results hold:

\medskip
1.  Existence:
\begin{itemize}
  \item  System \eqref{1.4} has a solution in $\mathfrak{B}_{\epsilon\nu}^{\frac{N}{2}}$ such that for all $0<\epsilon\leq\epsilon_0$,
$$
\|(b^\epsilon, {\bf u}^\epsilon, {\bf d}^\epsilon-\hat{\bf d})\|_{\mathfrak{B}_{\epsilon\nu}^{\frac{N}{2}}}\leq M C_0^{\epsilon\nu}.
$$
 \item System \eqref{1.3} has a unique solution such that
 \be\nonumber
 \begin{split}
 &\|{\bf u}\|_{\tilde{L}^\infty(\dot{B}_{2,1}^{\frac{N}{2}-1})}
+\|{\bf u}\|_{{L}^1(\dot{B}_{2,1}^{\frac{N}{2}+1})}
+\|{\bf d}-\hat{\bf d}\|_{\tilde{L}^\infty(\dot{B}_{2,1}^{\frac{N}{2}})}
+\|{\bf d}-\hat{\bf d}\|_{{L}^1(\dot{B}_{2,1}^{\frac{N}{2}+2})}\\[2mm]
&\quad \leq M{\Big(}\|{\bf u}_0\|_{\dot{B}_{2,1}^{\frac{N}{2}-1}}
+\|{\bf d}_0-\hat{\bf d}\|_{\dot{B}_{2,1}^{\frac{N}{2}}}{\Big)}.
 \end{split}
 \ee
 \end{itemize}

2. Convergence:
 \begin{itemize}
 \item If $N\geq 4$: For all $p\in [p_N,\infty]$ with $p_N:=\frac{2(N-1)}{N-3}$, we have
 \be\nonumber
 \|b^\epsilon\|_{\tilde{L}^2(\dot{B}_{p,1}^{\frac{N}{p}-\frac{1}{2}})}
+\|\mathcal{Q}{\bf u}^\epsilon\|_{\tilde{L}^2(\dot{B}_{p,1}^{\frac{N}{p}-\frac{1}{2}})}\leq M C_0^{\epsilon\nu}\epsilon^{\frac{1}{2}},
 \ee
 and
\be\nonumber
\begin{split}
&\|\mathcal{P}{\bf u}^\epsilon-{\bf u}\|_{L^\infty(\dot{B}_{p,1}^{\frac{N}{p}-\frac{3}{2}})}+\|\mathcal{P}{\bf u}^\epsilon-{\bf u}\|_{L^1(\dot{B}_{p,1}^{\frac{N}{p}+\frac{1}{2}})}
+\|{\bf d}^\epsilon-{\bf d}\|_{L^\infty(\dot{B}_{p,1}^{\frac{N}{p}-\frac{1}{2}})}\\[2mm]
&\quad+\|{\bf d}^\epsilon-{\bf d}\|_{L^1(\dot{B}_{p,1}^{\frac{N}{p}+\frac{3}{2}})}
\leq M{\Big(}\|\mathcal{P}{\bf u}_0^\epsilon-{\bf u}_0\|_{\dot{B}_{p,1}^{\frac{N}{p}-\frac{3}{2}}}
+\|{\bf d}_0^\epsilon-{\bf d}_0\|_{\dot{B}_{p,1}^{\frac{N}{p}-\frac{1}{2}}}+ C_0^{\epsilon\nu}\epsilon^{\frac{1}{2}}{\Big)}.
\end{split}
\ee
 \item If $N=3$: For all $p\in [2,\infty)$, we have
 \be\nonumber
 \|b^\epsilon\|_{\tilde{L}^{\frac{2p}{p-2}}(\dot{B}_{p,1}^{\frac{2}{p}-\frac{1}{2}})}
+\|\mathcal{Q}{\bf u}^\epsilon\|_{\tilde{L}^2(\dot{B}_{p,1}^{\frac{4}{p}-\frac{1}{2}})}\leq M C_0^{\epsilon\nu}\epsilon^{\frac{1}{2}-\frac{1}{p}},
 \ee
 and
\be\nonumber
\begin{split}
&\|\mathcal{P}{\bf u}^\epsilon-{\bf u}\|_{L^\infty(\dot{B}_{p,1}^{\frac{4}{p}-\frac{3}{2}})}+\|\mathcal{P}{\bf u}^\epsilon-{\bf u}\|_{L^1(\dot{B}_{p,1}^{\frac{4}{p}+\frac{1}{2}})}
+\|{\bf d}^\epsilon-{\bf d}\|_{L^\infty(\dot{B}_{p,1}^{\frac{4}{p}-\frac{1}{2}})}\\[2mm]
&\quad+\|{\bf d}^\epsilon-{\bf d}\|_{L^1(\dot{B}_{p,1}^{\frac{4}{p}+\frac{3}{2}})}
\leq M{\Big(}\|\mathcal{P}{\bf u}_0^\epsilon-{\bf u}_0\|_{\dot{B}_{p,1}^{\frac{4}{p}-\frac{3}{2}}}
+\|{\bf d}_0^\epsilon-{\bf d}_0\|_{\dot{B}_{p,1}^{\frac{4}{p}-\frac{1}{2}}}+  C_0^{\epsilon\nu}\epsilon^{\frac{1}{2}-\frac{1}{p}}{\Big)}.
\end{split}
\ee
 \item If $N=2$: For all $p\in [2, 6]$, we have
 \be\nonumber
 \|b^\epsilon\|_{\tilde{L}^{\frac{4p}{p-2}}(\dot{B}_{p,1}^{\frac{3}{2p}-\frac{3}{4}})}
+\|\mathcal{Q}{\bf u}^\epsilon\|_{\tilde{L}^2(\dot{B}_{p,1}^{\frac{5}{2p}-\frac{1}{4}})}\leq M C_0^{\epsilon\nu}\epsilon^{\frac{1}{4}-\frac{1}{2p}},
 \ee
 and
\be\nonumber
\begin{split}
&\|\mathcal{P}{\bf u}^\epsilon-{\bf u}\|_{L^\infty(\dot{B}_{p,1}^{\frac{5}{2p}-\frac{5}{4}})}+\|\mathcal{P}{\bf u}^\epsilon-{\bf u}\|_{L^1(\dot{B}_{p,1}^{\frac{5}{2p}+\frac{3}{4}})}
+\|{\bf d}^\epsilon-{\bf d}\|_{L^\infty(\dot{B}_{p,1}^{\frac{5}{2p}-\frac{1}{4}})}\\[2mm]
&\quad+\|{\bf d}^\epsilon-{\bf d}\|_{L^1(\dot{B}_{p,1}^{\frac{5}{2p}+\frac{7}{4}})}
\leq M{\Big(}\|\mathcal{P}{\bf u}_0^\epsilon-{\bf u}_0\|_{\dot{B}_{p,1}^{\frac{5}{2p}-\frac{5}{4}}}
+\|{\bf d}_0^\epsilon-{\bf d}_0\|_{\dot{B}_{p,1}^{\frac{5}{2p}-\frac{1}{4}}}+ C_0^{\epsilon\nu}\epsilon^{\frac{1}{4}-\frac{1}{2p}}{\Big)}.
\end{split}
\ee
 \end{itemize}
\end{theo}
\begin{re}
{\rm When ${\bf d}\equiv\hat{\bf d}$,  the compressible flow of nematic liquid crystals is reduced to the well-known compressible Navier-Stokes system. Our results coincide with the ones in \cite{danchin2002zero} concerning the zero Mach number limit in critical spaces for compressible Navier-Stokes equations.}
\end{re}

Our paper is organized as follows. In the next section, we recall some basic facts about Littlewood-Paley decomposition and  the homogeneous Besov spaces. In Section 3, we investigate the existence of global solutions for system \eqref{1.1}-\eqref{1.11}. Section 4 is devoted to the proof of uniqueness. In Section 5,  we will prove the convergence of the incompressible limit in the whole space $\mathbb{R}^N$.

We end this section by introducing the notations used throughout this paper. $C$ stands for a harmless constant which never depends on $\epsilon$, and we sometimes use the notation $A\lesssim B$ as an equivalent to $A\leq CB$. The notation $A\approx B$ means that $A\lesssim B$ and $B\lesssim A$.

\section{Homogeneous and hybrid Besov spaces}
\setcounter{equation}{0}\setcounter{section}{2}\indent
 We first recall the definition and some basic properties of homogeneous Besov spaces. They could be defined through the use of a dyadic partition of unity in Fourier variables called homogeneous Littlewood-Paley decomposition. To this end, choose a radial function $\varphi\in \mathcal{S}(\mathbb{R}^N)$ supported in $\mathcal{C}=\{\xi\in\mathbb{R}^N, \frac{3}{4}\leq |\xi|\leq \frac{8}{3}\}$ such that
$$
\sum_{j\in\mathbb{Z}}\varphi(2^{-j}\xi)=1~~~{\rm if}~~~ \xi\neq 0.
$$
The homogeneous frequency localization operator $\dot{\Delta}_j$ and $\dot{S}_j$ are defined by
$$
\dot{\Delta}_j u=\varphi (2^{-j}D)u, ~~~~\dot{S}_j u=\sum_{k\leq j-1}\dot{\Delta}_k u~~~{\rm for}~~j\in\mathbb{Z}.
$$
With our choice of $\varphi$, one can easily verify that
\be\label{2.1}
\dot{\Delta}_p\dot{\Delta}_q u\equiv 0 ~~{\rm if}~~|p-q|\geq 2~~~~{\rm and}~~~
\dot{\Delta}_p(\dot{S}_{q-1}u\dot{\Delta}_q u)\equiv 0~~~{\rm if}~~|p-q|\geq 5.
\ee

Let us denote the space $\mathcal{Y}^\prime(\mathbb{R}^N)$ by the quotient space of $\mathcal{S}^\prime(\mathbb{R}^N)/\mathcal{P}$ with the polynomials space $\mathcal{P}$. The formal equality
$$
u=\sum_{k\in\mathbb{Z}}\dot{\Delta}_k u
$$
holds true for $u\in \mathcal{Y}^\prime(\mathbb{R}^N)$ and is called the homogeneous Littlewood-Paley decomposition.

We will repeatedly use the following Bernstein's inequality:
\begin{lem}\label{le2.1}
{\rm(}see {\rm \cite{chemin1998perfect}}{\rm )} Let $\mathcal{C}$ be an annulus and $\mathcal{B}$ a ball, $1\leq p\leq q\leq +\infty$. Assume that $f\in L^p(\mathbb{R}^N)$, then for any nonnegative integer $k$, there exists constant $C$ independent of $f$, $k$ such that
$$
{\rm supp} \hat{f}\subset\lambda \mathcal{B}\Rightarrow\|D^k f\|_{L^q(\mathbb{R}^N)}:=\sup_{|\alpha|=k}\|\partial^\alpha f\|_{L^q(\mathbb{R}^N)}\leq C^{k+1}\lambda^{k+N(\frac{1}{p}-\frac{1}{q})}\|f\|_{L^p(\mathbb{R}^N)},
$$
\be\nonumber
{\rm supp} \hat{f}\subset\lambda\mathcal{C}\Rightarrow C^{-k-1}\lambda^k\|f\|_{L^p(\mathbb{R}^N)}\leq \|D^k f\|_{L^p(\mathbb{R}^N)}\leq
C^{k+1}\lambda^k\|f\|_{L^p(\mathbb{R}^N)}.
\ee
\end{lem}

Next, let us recall the definitions of the Besov spaces.
\begin{defi}\nonumber
Let $s\in\mathbb{R}$, $1\leq p, r\leq +\infty$. The homogeneous Besov space $\dot{B}_{p,r}^s$ is defined by
$$
\dot{B}_{p,r}^s=\{f\in\mathcal{Y}^\prime(\mathbb{R}^N): \|f\|_{\dot{B}_{p,r}^s}<+\infty\},
$$
where
$$
\|f\|_{\dot{B}_{p,r}^s}:=\|2^{ks}\|\dot{\Delta}_k f\|_{L^p}\|_{\ell^r}.
$$
\end{defi}

We next introduce the Besov-Chemin-Lerner space $\tilde{L}_T^\rho(\dot{B}_{p,r}^s)$, which is initiated in \cite{chemin1995flot}.
\begin{defi}
Let $s\in\mathbb{R}, 1\leq p, \rho, r\leq +\infty$, $0<T\leq +\infty$. The space $\tilde{L}_T^\rho(\dot{B}_{p,r}^s)$ is defined by
$$
\tilde{L}_T^\rho(\dot{B}_{p,r}^s)=\{f\in (0,+\infty)\times\mathcal{Y}^\prime(\mathbb{R}^N):
\|f\|_{\tilde{L}_T^\rho(\dot{B}_{p,r}^s)}<+\infty\},
$$
where
$$
\|f\|_{\tilde{L}_T^\rho(\dot{B}_{p,r}^s)}:=\|2^{ks}\|\dot{\Delta}_k f(t)\|_{L^\rho(0,T;L^p)}\|_{\ell^r}.
$$
\end{defi}
A direct application of Minkowski's inequality implies that
$$
L_T^\rho(\dot{B}_{p,r}^s)\hookrightarrow \tilde{L}_T^\rho(\dot{B}_{p,r}^s),~~~{\rm if}~~r\geq \rho,
$$
$$
\tilde{L}_T^\rho(\dot{B}_{p,r}^s)\hookrightarrow {L}_T^\rho(\dot{B}_{p,r}^s),~~~{\rm if}~~\rho\geq r.
$$

We also need the following hybrid Besov space introduced by Danchin in \cite{danchin2000global}:
\begin{defi}
For $\nu>0$, $r\in [1, +\infty], s\in\mathbb{R}$, we define
$$
\|f\|_{\tilde{B}_\nu^{s, r}}:=\sum_{j\in\mathbb{Z}}2^{js}\max\{\nu, 2^{-j}\}^{1-\frac{2}{r}}\|\dot{\Delta}_j f\|_{L^2}.
$$
\end{defi}
By the definition, it is easy to verify that
$$
\|f\|_{\tilde{B}_\nu^{s,\infty}}\approx \|f\|_{\dot{B}_{2,1}^{s-1}}+\nu\|f\|_{\dot{B}_{2,1}^{s}},
$$
which means that $\tilde{B}_\nu^{s,\infty}\approx\dot{B}_{2,1}^{s-1}\cap\dot{B}_{2,1}^{s}$.

Let us now state some classical properties for the Besov spaces.
\begin{prop}\label{pr2.1} The following properties hold true:

\medskip
{\rm 1)} Derivation: There exists a universal constant $C$ such that
$$
C^{-1}\|f\|_{\dot{B}_{p,r}^s}\leq \|\nabla f\|_{\dot{B}_{p,r}^{s-1}}\leq C\|f\|_{\dot{B}_{p,r}^s}.
$$

{\rm 2)} Sobolev embedding: If $1\leq p_1\leq p_2\leq\infty$ and $1\leq r_1\leq r_2\leq\infty$, then $\dot{B}_{p_1, r_1}^s\hookrightarrow \dot{B}_{p_2, r_2}^{s-\frac{N}{p_1}+\frac{N}{p_2}}$.

\medskip
{\rm 3)} Real interpolation: $\|f\|_{\dot{B}_{p,r}^{\theta s_1+(1-\theta)s_2}}\leq \|f\|_{\dot{B}_{p,r}^{s_1}}^{\theta}\|f\|_{\dot{B}_{p,r}^{s_2}}^{1-\theta}$.

\medskip
{\rm 4)} Algebraic properties: for $s>0$, $\dot{B}_{p,1}^s\cap L^\infty$ is an algebra.

\medskip
{\rm 5)} Scaling properties:

\medskip
\quad\quad {\rm(a)} for all $\lambda>0$ and $f\in\dot{B}_{p,1}^s$, we have
$$
\|f(\lambda\cdot)\|_{\dot{B}_{p,1}^s}\approx \lambda^{s-\frac{N}{p}}\|f\|_{\dot{B}_{p,1}^s},
$$

 \medskip
\quad\quad {\rm(b)} for $f=f(t, x)$ in $L^r(0, T; \dot{B}_{p,1}^s)$, we have
$$
\|f(\lambda^a\cdot, \lambda^b\cdot)\|_{L_T^r(\dot{B}_{p,1}^s)}\approx
\lambda^{b(s-\frac{N}{p})-\frac{a}{r}}\|f\|_{L_{\lambda^aT}^r(\dot{B}_{p,1}^s)}.
$$
\end{prop}
Next we recall a few nonlinear estimates in Besov spaces which may be obtained
by means of paradifferential calculus. Firstly introduced by J.-M. Bony in \cite{bony1981calcul}, the paraproduct between $f$
and $g$ is defined by
$$
\dot{T}_f g=\sum_{q\in\mathbb{Z}}\dot{S}_{q-1}f\dot{\Delta}_q g,
$$
and the remainder is given by
$$
\dot{R}(f,g)=\sum_{q\in\mathbb{Z}}\dot{\Delta}_q f\tilde{\dot{\Delta}}_q g
$$
with
$$
\tilde{\dot{\Delta}}_q g:=(\dot{\Delta}_{q-1}+\dot{\Delta}_{q}+\dot{\Delta}_{q+1})g.
$$
We have the following so-called Bony's decomposition:
\be\label{2.3}
fg=\dot{T}_g f+\underbrace{\dot{T}_f g+\dot{R}(f,g)}_{\dot{T}_f^\prime g}.
\ee

The paraproduct $\dot{T}$ and the remainder $\dot{R}$ operators satisfy the following continuous properties.
 \begin{prop}\label{pr2.2}
Suppose that $s\in\mathbb{R}, \sigma>0,$ and $1\leq p, p_1, p_2, r, r_1, r_2\leq \infty$.  Then we have
 \medskip

{\rm 1)} The paraproduct $\dot{T}$ is a bilinear, continuous operator from $L^\infty\times\dot{B}_{p,r}^s$ to $\dot{B}_{p,r}^s$, and from $\dot{B}_{\infty, r_1}^{-\sigma}\times\dot{B}_{p,r_2}^s$ to $\dot{B}_{p,r}^{s-\sigma}$ with
$\frac{1}{r}=\min\{1, \frac{1}{r_1}+\frac{1}{r_2}\}$.

\medskip
{\rm 2)} The remainder $\dot{R}$ is bilinear continuous from $\dot{B}_{p_1,r_1}^{s_1}\times\dot{B}_{p_2,r_2}^{s_2}$ to $\dot{B}_{p,r}^{s_1+s_2}$ with $s_1+s_2>0$, $\frac{1}{p}=\frac{1}{p_1}+\frac{1}{p_2}\leq 1$, and $\frac{1}{r}=\frac{1}{r_1}+\frac{1}{r_2}\leq 1$.

 \end{prop}

From \eqref{2.3} and Proposition \ref{pr2.2}, we have the following more accurate product estimate:
\begin{col}\label{co2.1}
If $u\in\dot{B}_{p_1,1}^{s_1}$ and $v\in\dot{B}_{p_2,1}^{s_2}$ with $1\leq p_1\leq p_2\leq \infty,~s_1\leq \frac{N}{p_1},~s_2\leq \frac{N}{p_2}$ and $s_1+s_2>0$, then
$uv\in\dot{B}_{p_2,1}^{s_1+s_2-\frac{N}{p_1}}$ and there exists a constant $C$, depending only on $N, s_1, s_2, p_1$ and $p_2$, such that
$$
\|uv\|_{\dot{B}_{p_2,1}^{s_1+s_2-\frac{N}{p_1}}}\leq C\|u\|_{\dot{B}_{p_1,1}^{s_1}}
\|v\|_{\dot{B}_{p_2,1}^{s_2}}.
$$
\end{col}

We finally need the following two composition lemmas (see \cite{danchin2000global, runst1996sobolev}).
\begin{lem}\label{le2.2}

Let $s>0$, $p\in[1,\infty]$ and $u\in\dot{B}_{p,1}^s\cap L^\infty$. Let $F\in W_{\rm loc}^{[s]+2,\infty}(\mathbb{R}^N)$ such that $F(0)=0$. Then $F(u)\in\dot{B}_{p,1}^s$ and there exists a constant $C=C(s,p,N,F,\|u\|_{L^\infty})$ such that
$$
\|F(u)\|_{\dot{B}_{p,1}^s}\leq C\|u\|_{\dot{B}_{p,1}^s}.
$$
\end{lem}

\begin{lem}\label{le2.3}
If $u$ and $v$ belong to $\dot{B}_{2,1}^{\frac{N}{2}}$, $(v-u)\in \dot{B}_{2,1}^s$ for  $s\in (-\frac{N}{2}, \frac{N}{2}]$
and $G\in W_{loc}^{[\frac{N}{2}]+3, \infty}(\mathbb{R}^N)$ satisfies $G^\prime(0)=0$, then $G(v)-G(u)$ belongs to $\dot{B}_{2,1}^s$ and there exists a function of two variables $C$ depending only on $s$, $N$ and $G$, and such that
$$
\|G(v)-G(u)\|_{\dot{B}_{2,1}^s}\leq C(\|u\|_{L^\infty}, \|v\|_{L^\infty})
{\Big(}\|u\|_{\dot{B}_{2,1}^{\frac{N}{2}}}+\|v\|_{\dot{B}_{2,1}^{\frac{N}{2}}}{\Big)}
\|v-u\|_{\dot{B}_{2,1}^{s}}.
$$
\end{lem}

\section{Global existence for initial data near equilibrium}
\setcounter{equation}{0}\setcounter{section}{3}\indent
In this section, we will prove the part (i) of Theorem \ref{th1.1}.
Throughout this paper, we consider only viscous fluids, those for which $\mu>0$ and $\nu>0$. The following proposition plays an important role in obtaining the  estimates of $({\bf u}, b)$.

\begin{prop}\label{pr3.1}{\rm(}see {\rm \cite{bahouri2011fourier}}{\rm )}
Let $\nabla {\bf v}\in L_T^1(L^\infty(\mathbb{R}^N))$, $s\in\mathbb{R}$, and $(a, {\bf u})$ be a solution of the following system
\begin{equation}\label{3.1}
  \left\{\begin{array}{ll}\medskip\D\partial_t a+{\rm div}(\dot{T}_{\bf v}a)+{\rm div}{\bf u}=F,\\
  \partial_t{\bf u}+\dot{T}_{\bf v}\cdot\nabla {\bf u}-\mu\Delta{\bf u}-(\mu+\lambda)\nabla{\rm div}{\bf u}+\nabla a=G,
  \end{array}
  \right.
\end{equation}
on $[0, T)$. Then there exists a constant $C$ depending only on $N$ and $s$, such that the following estimate holds on $[0, T)$:
\be\label{3.2}
\begin{split}
  \|a(t)\|_{\tilde{B}_\nu^{s,\infty}}+&\|{\bf u}(t)\|_{\dot{B}_{2,1}^{s-1}}
  +\int_0^t(\nu\|a\|_{\tilde{B}_\nu^{s,1}}
  +\underline{\nu}\|{\bf u}\|_{\dot{B}_{2,1}^{s+1}})dt^\prime
  \leq C e^{C\|\nabla {\bf v}\|_{L_t^1(L^\infty(\mathbb{R}^N))}}\\
  &\times{\Big(}\|a_0\|_{\tilde{B}_\nu^{s,\infty}}+\|{\bf u}_0\|_{\dot{B}_{2,1}^{s-1}}
  +\int_0^t e^{-C\|\nabla {\bf v}\|_{L_{t^\prime}^1(L^\infty(\mathbb{R}^N))}}
  (\|F\|_{\tilde{B}_\nu^{s,\infty}}+\|G\|_{\dot{B}_{2,1}^{s-1}})dt^\prime{\Big)}.
\end{split}
\ee
\end{prop}

Next, we establish the estimate of the director field ${\bf d}$ in critical Besov sapce.
\begin{prop}\label{pr3.2}
Let $s\in (-\frac{N}{2}, 1+\frac{N}{2})$, and ${\bf d}$ be a solution of the following equation
\bes\label{3.3}
\partial_t{\bf d}+{\bf u}\cdot \nabla{\bf d}-\theta\Delta {\bf d}=\mathcal{M}
\ees
on $[0, T)$, where ${\bf u}\in L_T^1(\dot{B}_{2,1}^{\frac{N}{2}+1})$ and $\mathcal{M}\in L_T^1(\dot{B}_{2,1}^s)$. Then there exists a constant $C$ depending on $N$ and $s$, such that the following estimate holds on $[0, T)$:
 \bes
 \label{3.4}
 \|{\bf d}\|_{\tilde{L}_t^\infty(\dot{B}_{2,1}^s)}+\theta\|{\bf d}\|_{{L}_t^1(\dot{B}_{2,1}^{s+2})}\leq C\exp{\Big(}{C\|{\bf u}\|_{{L}_t^1(\dot{B}_{2,1}^{\frac{N}{2}+1})}}{\Big)}{\Big(}\|{\bf d}_0\|_{\dot{B}_{2,1}^s}+\|\mathcal{M}\|_{L_t^1(\dot{B}_{2,1}^s)}{\Big)}.
 \ees
 \end{prop}

 \begin{proof}~Applying $\dot{\Delta}_q$ to \eqref{3.3} and taking the $L^2$ inner product of the resulting equation with $\dot{\Delta}_q{\bf d}$, integrating by part, we deduce that
 \be\nonumber
 \begin{split}
   \D &\frac{1}{2}\frac{d}{dt}\|\dot{\Delta}_q{\bf d}\|_{L^2}^2+\theta\|\nabla\dot{\Delta}_q{\bf d}\|_{L^2}^2\\
   &\quad\leq {\Big(}\frac{1}{2}\|{\rm div}{\bf u}\|_{L^\infty}\|\dot{\Delta}_q{\bf d}\|_{L^2}+\|[{\bf u},\dot{\Delta}_q]\cdot\nabla{\bf d}\|_{L^2}+\|\dot{\Delta}_q\mathcal{M}\|_{L^2}{\Big)}\|\dot{\Delta}_q{\bf d}\|_{L^2}.
 \end{split}
 \ee
Hence, according to Bernstein's inequality, we get, for some universal constant $\kappa$,
\be\nonumber
\begin{split}
&\|\dot{\Delta}_q{\bf d}\|_{L_t^\infty(L^2)}+\theta\kappa2^{2q}\|\dot{\Delta}_q{\bf d}\|_{L_t^1(L^2)}\\
&\quad\leq\|\dot{\Delta}_q{\bf d}_0\|_{L^2}+\int_0^t{\Big(}\frac{1}{2}\|{\rm div}{\bf u}\|_{L^\infty}\|\dot{\Delta}_q{\bf d}\|_{L^2}+\|[{\bf u},\dot{\Delta}_q]\cdot\nabla{\bf d}\|_{L^2}+\|\dot{\Delta}_q\mathcal{M}\|_{L^2}{\Big)}dt^\prime.
\end{split}
\ee
Now, multiplying both sides by $2^{qs}$ and summing over $q$, we end up with
\bes\label{3.5}
\begin{split}
   &\|{\bf d}\|_{\tilde{L}_t^\infty(\dot{B}_{2,1}^s)}+\kappa\theta\|{\bf d}\|_{L_t^1(\dot{B}_{2,1}^{s+2})}\\
   &\quad \leq \|{\bf d}_0\|_{\dot{B}_{2,1}^s}+\int_0^t{\Big(}\frac{1}{2}\|{\rm div}{\bf u}\|_{L^\infty}\|{\bf d}\|_{\dot{B}_{2,1}^s}+\sum_{q\in\mathbb{Z}}2^{qs}\|[{\bf u},\dot{\Delta}_q]\cdot\nabla{\bf d}\|_{L^2}+\|\mathcal{M}\|_{\dot{B}_{2,1}^s}{\Big)}dt^\prime.
\end{split}
\ees
According to Lemma 2.100 in \cite{bahouri2011fourier}, we get, for $-\frac{N}{2}<s<\frac{N}{2}+1$,
\be\label{3.6}
\sum_{q\in\mathbb{Z}}2^{qs}\|[{\bf u},\dot{\Delta}_q]\cdot\nabla{\bf d}\|_{L^2}\leq
C\|{\bf u}\|_{\dot{B}_{2,1}^{\frac{N}{2}+1}}\|{\bf d}\|_{\dot{B}_{2,1}^s}.
\ee
Substituting \eqref{3.6} into \eqref{3.5} and using the embedding $\dot{B}_{2,1}^{\frac{N}{2}}\hookrightarrow L^\infty$, we get
\bes\nonumber
\begin{split}
   &\|{\bf d}\|_{\tilde{L}_t^\infty(\dot{B}_{2,1}^s)}+\kappa\theta\|{\bf d}\|_{L_t^1(\dot{B}_{2,1}^{s+2})}\\
   &\quad \leq \|{\bf d}_0\|_{\dot{B}_{2,1}^s}+\|\mathcal{M}\|_{L_t^1(\dot{B}_{2,1}^s)}
   +C\int_0^t\|{\bf u}\|_{\dot{B}_{2,1}^{\frac{N}{2}+1}}\|{\bf d}\|_{\dot{B}_{2,1}^s}dt^\prime,
\end{split}
\ees
taking advantage of Gronwall's inequality, we obtain \eqref{3.4} immediately. This competes the proof of Proposition \ref{pr3.2}.
\end{proof}
{\bf Proof of the part (i) of Theorem \ref{th1.1}}.  We are going to prove that if the initial data
\bes\label{3.7}
\|\rho_0-1\|_{\tilde{B}_\nu^{\frac{N}{2},\infty}}+\|{\bf u}_0\|_{\dot{B}_{2,1}^{\frac{N}{2}-1}}+\|{\bf d}_0-\hat{\bf d}\|_{\dot{B}_{2,1}^{\frac{N}{2}}}\leq \eta,
\ees
for some sufficiently small $\eta$, there exists a positive constant $\Gamma$ such that
\be\label{3.8}
\|(\rho-1, {\bf u}, {\bf d}-\hat{\bf d})\|_{\mathfrak{B}_\nu^{\frac{N}{2}}}\leq \Gamma\eta.
\ee
This uniform estimates will enable us to extend the local solution $(\rho-1, {\bf u}, {\bf d})$ obtained by using a Friedrichs method as in \cite{bahouri2011fourier} to be a global one. To this end, we use a contradiction argument. Define
\be\label{3.9}
T_0=\sup{\Big\{}T\in [0,\infty): \|(\rho-1, {\bf u}, {\bf d}-\hat{\bf d})\|_{\mathfrak{B}_\nu^{\frac{N}{2}}(T)}\leq \Gamma\eta {\Big\}}
\ee
with $\Gamma$ to be determined later.

We note that $\nabla{\bf d}=\nabla({\bf d}-\hat{\bf d})$ because $\hat{\bf d}\in \mathbb{S}^{N-1}$ is a constant vector. Letting $b=\rho-1$, we rewrite the system \eqref{1.1} as
\be\label{3.10}
\left\{\begin{array}{ll}\medskip\D\partial_t b+{\rm div}{\bf u}+{\rm div}(\dot{T}_{\bf u}b)=-{\rm div}(\dot{T}_b^\prime{\bf u}),\\
\medskip\D\partial_t{\bf u}+\dot{T}_{\bf u}\cdot\nabla{\bf u}-\mathcal{A}{\bf u}+\nabla b=K(b)\nabla b-I(b)\mathcal{A}{\bf u}-\dot{T}_{\partial_i{\bf u}}^\prime u^i+H,\\
\D \partial_t({\bf d}-\hat{\bf d})+{\bf u}\cdot\nabla({\bf d}-\hat{\bf d})-\theta\Delta({\bf d}-\hat{\bf d})=J,
\end{array}
\right.
\ee
where
\be\nonumber
H:=-\frac{\xi}{1+b}{\rm div}{\Big(}\nabla({\bf d-\hat{\bf d}})\odot\nabla({\bf d}-\hat{\bf d})-\frac{1}{2}|\nabla({\bf d}-\hat{\bf d})|^2{\bf I}{\Big)},
\ee

\be\nonumber
J:=\theta|\nabla({\bf d}-\hat{\bf d})|^2({\bf d}-\hat{\bf d})+\theta|\nabla({\bf d}-\hat{\bf d})|^2\hat{\bf d},
\ee
and $\mathcal{A}:=\mu\Delta+(\mu+\lambda)\nabla{\rm div}$,~~$I(b):=\frac{b}{1+b}$,~~
$K(b):=1-\frac{P^\prime(1+b)}{1+b}$.

Suppose that $T_0<\infty$. We apply the linear estimates in Propositions \ref{pr3.1} and \ref{pr3.2} to the solutions of reformulated system \eqref{3.10} such that for all $t\in [0, T_0]$, the following estimates hold:
\be\label{3.11}
\begin{split}
&\medskip\D\|b(t)\|_{\tilde{B}_\nu^{\frac{N}{2},\infty}}+\|{\bf u}(t)\|_{\dot{B}_{2,1}^{\frac{N}{2}-1}}
+\int_0^{T_0}{\Big(}\nu\|b\|_{\tilde{B}_\nu^{\frac{N}{2},1}}
+\underline{\nu}\|{\bf u}\|_{\dot{B}_{2,1}^{\frac{N}{2}+1}}{\Big)}dt^\prime\\[2mm]
&\medskip\D\quad\leq Ce^{\|\nabla{\bf u}\|_{L_{T_0}^1(L^\infty(\mathbb{R}^N))}}{\Big\{}
\|b_0\|_{\tilde{B}_\nu^{\frac{N}{2},\infty}}+\|{\bf u}_0\|_{\dot{B}_{2,1}^{\frac{N}{2}-1}}+\|{\rm div}(\dot{T}_b^\prime{\bf u})\|_{L_{T_0}^1(\tilde{B}_\nu^{\frac{N}{2},\infty})}\\[2mm]
&\medskip\D\quad\quad+\|K(b)\nabla b\|_{L_{T_0}^1(\dot{B}_{2,1}^{\frac{N}{2}-1})}
+\|I(b)\mathcal{A{\bf u}}\|_{L_{T_0}^1(\dot{B}_{2,1}^{\frac{N}{2}-1})}
+\|\dot{T}_{\partial_i{\bf u}}^\prime u^i\|_{L_{T_0}^1(\dot{B}_{2,1}^{\frac{N}{2}-1})}
+\|H\|_{L_{T_0}^1(\dot{B}_{2,1}^{\frac{N}{2}-1})}{\Big\}},
\end{split}
\ee
and
\be\label{3.12}
\begin{split}
&\medskip\D\|{\bf d}(t)-\hat{\bf d}\|_{\dot{B}_{2,1}^{\frac{N}{2}}}+\theta\int_0^{T_0}\|{\bf d}-\hat{\bf d}\|_{\dot{B}_{2,1}^{\frac{N}{2}+2}}dt^\prime\\[2mm]
&\medskip\D\quad\leq C\exp{\Big(}C\|{\bf u}\|_{L_{T_0}^1({\dot{B}_{2,1}^{\frac{N}{2}+1}})}{\Big)}{\Big(}\|{\bf d}_0-\hat{\bf d}\|_{\dot{B}_{2,1}^{\frac{N}{2}}}+\|J\|_{L_{T_0}^1(\dot{B}_{2,1}^{\frac{N}{2}})}{\Big)}.
\end{split}
\ee

In what follows, we derive estimates for the nonlinear terms one by one. Similar to the case of isentropic Navier-Stokes equations \cite{bahouri2011fourier}, by Proposition \ref{pr2.2}, Corollary \ref{co2.1} and Lemma \ref{le2.2}, we have the following inequalities:
\be\label{3.13}
\|{\rm div}(\dot{T}_b^\prime{\bf u})\|_{L_{T_0}^1(\tilde{B}_\nu^{\frac{N}{2},\infty})}\lesssim \|b\|_{L_{T_0}^\infty(\tilde{B}_\nu^{\frac{N}{2},\infty})}\|{\bf u}\|_{L_{T_0}^1(\dot{B}_{2,1}^{\frac{N}{2}+1})},
\ee
\medskip
\be\label{3.14}
\|K(b)\nabla b\|_{L_{T_0}^1(\dot{B}_{2,1}^{\frac{N}{2}-1})}
\lesssim\|b\|_{L_{T_0}^2(\dot{B}_{2,1}^{\frac{N}{2}})}^2
\lesssim\|b\|_{L_{T_0}^\infty(\tilde{B}_\nu^{\frac{N}{2},\infty})}
\|b\|_{L_{T_0}^1(\tilde{B}_\nu^{\frac{N}{2},1})},
\ee
\medskip
\be\label{3.15}
\|I(b)\mathcal{A{\bf u}}\|_{L_{T_0}^1(\dot{B}_{2,1}^{\frac{N}{2}-1})}
\lesssim\|b\|_{L_{T_0}^\infty(\tilde{B}_\nu^{\frac{N}{2},\infty})}
\|{\bf u}\|_{L_{T_0}^1(\dot{B}_{2,1}^{\frac{N}{2}+1})},
\ee
\medskip
\be\label{3.16}
\|\dot{T}_{\partial_i{\bf u}}^\prime u^i\|_{L_{T_0}^1(\dot{B}_{2,1}^{\frac{N}{2}-1})}\lesssim\|{\bf u}\|_{L_{T_0}^\infty(\dot{B}_{2,1}^{\frac{N}{2}-1})}\|{\bf u}\|_{L_{T_0}^1(\dot{B}_{2,1}^{\frac{N}{2}+1})}.
\ee
Compared with the isentropic Navier-Stokes equations, the new terms  can be estimated as follows:
\be\label{3.17}
\begin{split}
\|H\|_{L_{T_0}^1(\dot{B}_{2,1}^{\frac{N}{2}-1})}
&={\Big\|}\frac{\xi}{1+b}{\rm div}{\Big(}\nabla({\bf d}-\hat{\bf d})\odot\nabla({\bf d}-\hat{\bf d})-\frac{1}{2}|\nabla({\bf d}-\hat{\bf d})|^2{\bf I}{\Big)}{\Big\|}_{L_{T_0}^1(\dot{B}_{2,1}^{\frac{N}{2}-1})}\\[2mm]
&\leq \xi{\Big\|}{{\rm div}{\Big(}\nabla({\bf d-\hat{\bf d}})\odot\nabla({\bf d}-\hat{\bf d})-\frac{1}{2}|\nabla({\bf d}-\hat{\bf d})|^2{\bf I}{\Big)}}{\Big\|}_{L_{T_0}^1(\dot{B}_{2,1}^{\frac{N}{2}-1})}\\[2mm]
&\quad+\xi{\Big\|}{I(b){\rm div}{\Big(}\nabla({\bf d-\hat{\bf d}})\odot\nabla({\bf d}-\hat{\bf d})-\frac{1}{2}|\nabla({\bf d}-\hat{\bf d})|^2{\bf I}{\Big)}}{\Big\|}_{L_{T_0}^1(\dot{B}_{2,1}^{\frac{N}{2}-1})}\\[2mm]
&\lesssim {\Big(}1+\|b\|_{L_{T_0}^\infty(\dot{B}_{2,1}^{\frac{N}{2}})}{\Big)}{\Big\|}{\nabla({\bf d-\hat{\bf d}})\odot\nabla({\bf d}-\hat{\bf d})-\frac{1}{2}|\nabla({\bf d}-\hat{\bf d})|^2{\bf I}}{\Big\|}_{L_{T_0}^1(\dot{B}_{2,1}^{\frac{N}{2}})}\\[2mm]
&\lesssim {\Big(}1+\|b\|_{L_{T_0}^\infty(\tilde{B}_\nu^{\frac{N}{2},\infty})}{\Big)}
\|\nabla({\bf d}-\hat{\bf d})\|_{L_{T_0}^2(\dot{B}_{2,1}^{\frac{N}{2}})}^2\\[2mm]
&\lesssim {\Big(}1+\|b\|_{L_{T_0}^\infty(\tilde{B}_\nu^{\frac{N}{2},\infty})}{\Big)}
\|{\bf d}-\hat{\bf d}\|_{L_{T_0}^\infty(\dot{B}_{2,1}^{\frac{N}{2}})}\|{\bf d}-\hat{\bf d}\|_{L_{T_0}^1(\dot{B}_{2,1}^{\frac{N}{2}+2})},
\end{split}
\ee

\be\label{3.18}
\begin{split}
\|J\|_{L_{T_0}^1(\dot{B}_{2,1}^{\frac{N}{2}})}&=\theta\||\nabla({\bf d}-\hat{\bf d})|^2({\bf d}-\hat{\bf d})\|_{L_{T_0}^1(\dot{B}_{2,1}^{\frac{N}{2}})}+\theta\||\nabla({\bf d}-\hat{\bf d})|^2\hat{\bf d}\|_{L_{T_0}^1(\dot{B}_{2,1}^{\frac{N}{2}})}:=J_1+J_2.
\end{split}
\ee
For the estimate of $J_1$, it follows that
\be\label{3.19}
\begin{split}
J_1&=\theta\||\nabla({\bf d}-\hat{\bf d})|^2({\bf d}-\hat{\bf d})\|_{L_{T_0}^1(\dot{B}_{2,1}^{\frac{N}{2}})}\\[2mm]
&\lesssim \||\nabla({\bf d}-\hat{\bf d})|^2\|_{L_{T_0}^1(\dot{B}_{2,1}^{\frac{N}{2}})}
\|{\bf d}-\hat{\bf d}\|_{L_{T_0}^\infty(\dot{B}_{2,1}^{\frac{N}{2}})}\\[2mm]
&\lesssim \|{\bf d}-\hat{\bf d}\|_{L_{T_0}^2(\dot{B}_{2,1}^{\frac{N}{2}+1})}^2
\|{\bf d}-\hat{\bf d}\|_{L_{T_0}^\infty(\dot{B}_{2,1}^{\frac{N}{2}})}\\[2mm]
&\lesssim \|{\bf d}-\hat{\bf d}\|_{L_{T_0}^1(\dot{B}_{2,1}^{\frac{N}{2}+2})}
\|{\bf d}-\hat{\bf d}\|_{L_{T_0}^\infty(\dot{B}_{2,1}^{\frac{N}{2}})}^2.
\end{split}
\ee
For $J_2$, we have by the definition of Besov's spaces
\be\label{3.20}
\begin{split}
J_2&=\theta\||\nabla({\bf d}-\hat{\bf d})|^2\hat{\bf d}\|_{L_{T_0}^1(\dot{B}_{2,1}^{\frac{N}{2}})}
\lesssim \|{\bf d}-\hat{\bf d}\|_{L_{T_0}^2(\dot{B}_{2,1}^{\frac{N}{2}+1})}^2\\[2mm]
&\lesssim \|{\bf d}-\hat{\bf d}\|_{L_{T_0}^\infty(\dot{B}_{2,1}^{\frac{N}{2}})} \|{\bf d}-\hat{\bf d}\|_{L_{T_0}^1(\dot{B}_{2,1}^{\frac{N}{2}+2})}.
\end{split}
\ee
Plugging inequalities  \eqref{3.13}-\eqref{3.20} in \eqref{3.11} and \eqref{3.12}, we thus get
\be\label{3.21}
\begin{split}
&\medskip\D\|b\|_{L_{T_0}^\infty(\tilde{B}_\nu^{\frac{N}{2},\infty})}+\|{\bf u}(t)\|_{L_{T_0}^\infty(\dot{B}_{2,1}^{\frac{N}{2}-1})}
+\int_0^{T_0}{\Big(}\nu\|b\|_{\tilde{B}_\nu^{\frac{N}{2},1}}
+\underline{\nu}\|{\bf u}\|_{\dot{B}_{2,1}^{\frac{N}{2}+1}}{\Big)}dt^\prime\\[2mm]
&\medskip\D\quad\lesssim e^{\|\nabla{\bf u}\|_{L_{T_0}^1(L^\infty(\mathbb{R}^N))}}{\Big\{}
\|b_0\|_{\tilde{B}_\nu^{\frac{N}{2},\infty}}+\|{\bf u}_0\|_{\dot{B}_{2,1}^{\frac{N}{2}-1}}+\|b\|_{L_{T_0}^\infty(\tilde{B}_\nu^{\frac{N}{2},\infty})}\|{\bf u}\|_{L_{T_0}^1(\dot{B}_{2,1}^{\frac{N}{2}+1})}\\[2mm]
&\medskip\D\quad\quad+\|b\|_{L_{T_0}^\infty(\tilde{B}_\nu^{\frac{N}{2},\infty})}
\|b\|_{L_{T_0}^1(\tilde{B}_\nu^{\frac{N}{2},1})}
+\|{\bf u}\|_{L_{T_0}^\infty(\dot{B}_{2,1}^{\frac{N}{2}-1})}\|{\bf u}\|_{L_{T_0}^1(\dot{B}_{2,1}^{\frac{N}{2}+1})}\\[2mm]
&\quad\quad+{\Big(}1+\|b\|_{L_{T_0}^\infty(\tilde{B}_\nu^{\frac{N}{2},\infty})}{\Big)}
\|{\bf d}-\hat{\bf d}\|_{L_{T_0}^\infty(\dot{B}_{2,1}^{\frac{N}{2}})}\|{\bf d}-\hat{\bf d}\|_{L_{T_0}^1(\dot{B}_{2,1}^{\frac{N}{2}+2})}{\Big\}},
\end{split}
\ee
and
\be\label{3.22}
\begin{split}
&\medskip\D\|{\bf d}-\hat{\bf d}\|_{L_{T_0}^\infty\dot{B}_{2,1}^{\frac{N}{2}}}+\theta\int_0^{T_0}\|{\bf d}-\hat{\bf d}\|_{\dot{B}_{2,1}^{\frac{N}{2}+2}}dt^\prime\\[2mm]
&\medskip\D\quad\lesssim \exp{\Big(}C\|{\bf u}\|_{L_{T_0}^1({\dot{B}_{2,1}^{\frac{N}{2}+1}})}{\Big)}\Big\{\|{\bf d}_0-\hat{\bf d}\|_{\dot{B}_{2,1}^{\frac{N}{2}}}+\|{\bf d}-\hat{\bf d}\|_{L_{T_0}^1(\dot{B}_{2,1}^{\frac{N}{2}+2})}
\|{\bf d}-\hat{\bf d}\|_{L_{T_0}^\infty(\dot{B}_{2,1}^{\frac{N}{2}})}^2\\[2mm]
&\quad\quad+\|{\bf d}-\hat{\bf d}\|_{L_{T_0}^1(\dot{B}_{2,1}^{\frac{N}{2}+2})}
\|{\bf d}-\hat{\bf d}\|_{L_{T_0}^\infty(\dot{B}_{2,1}^{\frac{N}{2}})}\Big\}.
\end{split}
\ee
The above two inequalities combined with the definition of $T_0$ yield that
\be\label{3.23}
\|(b, {\bf u}, {\bf d}-\hat{\bf d})\|_{\mathfrak{B}_\nu^{\frac{N}{2}}({T_0})}\leq C_1e^{C_1\Gamma\eta}(\eta+\Gamma^2\eta^2+\Gamma^3\eta^3).
\ee
We choose $\Gamma=4C_1$ and $\eta>0$ satisfying
\be\label{3.24}
e^{C_1\Gamma\eta}<2,~~~\Gamma^2\eta\leq\frac{1}{2},~~~\Gamma^3\eta^2\leq\frac{1}{2}.
\ee
Consequently, it follows from \eqref{3.23} and the above choices of $\Gamma$ and $\eta$ that
\be\nonumber
\|(b, {\bf u}, {\bf d}-\hat{\bf d})\|_{\mathfrak{B}_\nu^{\frac{N}{2}}({T_0})}<\Gamma\eta,
\ee
which is a contradiction with the definition of $T_0$. Hence, we can deduce that $T_0=\infty$. Global existence is thus proved.\hspace{11.5cm}$\Box$

\section{Uniqueness}
\setcounter{equation}{0}\setcounter{section}{4}\indent
In this section, we will prove the uniqueness of the solution to system \eqref{1.1}, i.e., the part (ii) of Theorem \ref{th1.1}.

{\bf Proof of the part (ii) of Theorem \ref{th1.1}}.

{\bf Case $N\geq 3$}:
  Assume that $(b_i, {\bf u}_i, {\bf d}_i)_{i=1,2}$ in $\mathfrak{B}_\nu^{\frac{N}{2}}$ solve \eqref{1.1} with the same initial data. Denote
\bess
\delta b=b_2-b_1,~~~~\delta{\bf u}={\bf u}_2-{\bf u}_1,~~~~\delta{\bf d}={\bf d}_2-{\bf d}_1.
\eess
Then $(\delta b,\delta{\bf u},\delta{\bf d})\in \mathfrak{B}_\nu^{\frac{N}{2}}(T)$ for all $T>0$. On the other hand, since $(b_i, {\bf u}_i, {\bf d}_i)_{i=1,2}$ are solutions to system \eqref{1.1} with the same initial data, $(\delta b,\delta{\bf u},\delta{\bf d})$ solves
\be\label{4.1}
\left\{\begin{array}{ll}
\partial_t\delta b+{\rm div}(\dot{T}_{{\bf u}_2}\cdot\delta b)+{\rm div}\delta{\bf u}=\delta F,\\[2mm]
\partial_t\delta{\bf u}+\dot{T}_{{\bf u}_2}\cdot\nabla\delta{\bf u}-\mathcal{A}\delta{\bf u}+\nabla\delta b=\delta G,\\[2mm]
\partial_t\delta{\bf d}+{\bf u}_2\cdot\nabla\delta{\bf d}-\theta\Delta\delta{\bf d}=\delta\mathcal{M},
\end{array}
\right.
\ee
with
\bess
\begin{split}
\delta F&=-\delta{\bf u}\cdot\nabla b_1-b_1{\rm div}\delta{\bf u}-{\rm div}(\dot{T}_{\delta b}^\prime{\bf u}_2),\\[2mm]
\delta G&=-\dot{T}_{\partial_i\delta{\bf u}}^\prime u_2^i-\delta{\bf u}\cdot\nabla{\bf u}_1-(I(b_2)-I(b_1))\mathcal{A}{\bf u}_2-I(b_1)\mathcal{A}\delta{\bf u}\\[2mm]
&\quad +\xi(I(b_2)-I(b_1)){\rm div}{\Big(}\nabla({\bf d}_2-\hat{\bf d})\odot\nabla({\bf d}_2-\hat{\bf d})-\frac{1}{2}|\nabla({\bf d}_2-\hat{\bf d})|^2{\bf I}{\Big)}\\[2mm]
&\quad +\frac{\xi}{1+b_1}{\rm div}{\Big(}\nabla({\bf d}_1-\hat{\bf d})\odot\nabla\delta{\bf d}+\nabla({\bf d}_2-\hat{\bf d})\odot\nabla\delta{\bf d}\\[2mm]
&\quad -\frac{1}{2}{\bf I}(\nabla({\bf d}_1-\hat{\bf d}):\nabla\delta{\bf d}+\nabla({\bf d}_2-\hat{\bf d}):\nabla\delta{\bf d}){\Big)}+\nabla(\mathcal{K}(b_2)-\mathcal{K}(b_1)),
\end{split}
\eess
and
\bess
\delta\mathcal{M}=-\delta{\bf u}\cdot\nabla{\bf d}_1+\theta{\Big(}|\nabla({\bf d}_1-\hat{\bf d})|^2\delta{\bf d}+(\nabla({\bf d}_1-\hat{\bf d}):\nabla\delta{\bf d}+\nabla({\bf d}_2-\hat{\bf d}):\nabla\delta{\bf d}){\bf d}_2{\Big)},
\eess
where we used the notations $A:B=\sum_{i,j=1}^NA_{ij}B_{ij}$ and $\mathcal{K}(z)=\int_0^zK(y)dy$.
Applying Propositions \ref{pr3.1} and \ref{pr3.2} to the system \eqref{4.1}, we have
\be\label{4.2}
\begin{split}
&\|(\delta b, \delta{\bf u}, \delta{\bf d})\|_{\mathfrak{B}_\nu^{\frac{N}{2}-1}(T)}\\[2mm]
&\quad\leq C\exp{\Big(}C\|{\bf u}_2\|_{L_T^1(\dot{B}_{2,1}^{\frac{N}{2}+1})}{\Big)}{\Big(}\|\delta F\|_{L_T^1(\tilde{B}_\nu^{\frac{N}{2}-1,\infty})}+\|\delta G\|_{L_T^1(\dot{B}_{2,1}^{\frac{N}{2}-2})}+\|\delta \mathcal{M}\|_{L_T^1(\dot{B}_{2,1}^{\frac{N}{2}-1})}{\Big)}.
\end{split}
\ee
 As in \cite{danchin2000global}, we could get that $b_i\in C(\mathbb{R}^+;\dot{B}_{2,1}^{\frac{N}{2}-1})\,(i=1,2)$. Indeed, $b_0\in\dot{B}_{2,1}^{\frac{N}{2}-1}$, $b_i\in L^\infty(\mathbb{R}^+;\dot{B}_{2,1}^{\frac{N}{2}-1})$ and $\partial_tb_i=-{\bf u}_i\cdot\nabla b_i-{\rm div}{\bf u}_i-b_i{\rm div}{\bf u}_i\in L^2(\mathbb{R}^+;\dot{B}_{2,1}^{\frac{N}{2}-1})$. And because $b_i\in C(\dot{B}_{2,1}^{\frac{N}{2}-1})\cap L_{loc}^\infty\dot{B}_{2,1}^{\frac{N}{2}}$, we have $b_i\in C([0,+\infty)\times\mathbb{R}^N)$. On the other hand, if \eqref{1.10} is satisfied for some $\eta$ suitably small, we have
 \be\label{4.18}
 |b_i(t,x)|\leq \frac{1}{4}~~~~{\rm for\,\, all}~~t\geq 0~~{\rm and}~~x\in \mathbb{R}^N.
 \ee

In the following, we estimate the terms $\delta F, \delta G$ and $\delta\mathcal{M}$ respectively. Firstly, according to Proposition \ref{pr2.2} and
Corollary \ref{co2.1}, one gets
\be
\|\delta F\|_{L_T^1(\tilde{B}_\nu^{\frac{N}{2}-1,\infty})}\lesssim\|\delta{\bf u}\|_{L_T^1(\dot{B}_{2,1}^{\frac{N}{2}})}
\|b_1\|_{L_T^\infty(\tilde{B}_\nu^{\frac{N}{2},\infty})}
+\|\delta b\|_{L_T^\infty(\tilde{B}_\nu^{\frac{N}{2}-1,\infty})}\|{\bf u}_2\|_{L_T^1(\dot{B}_{2,1}^{\frac{N}{2}+1})}.
\ee\label{4.4}
Similarly, the terms of  $\delta G$ could be estimated as follows:
\be\label{4.5}
\|\dot{T}_{\partial_i\delta{\bf u}}^\prime u_2^i\|_{L_T^1(\dot{B}_{2,1}^{\frac{N}{2}-2})}\lesssim \|\partial_i\delta{\bf u}\|_{L_T^2(\dot{B}_{2,1}^{\frac{N}{2}-2})}\|{\bf u}_2\|_{L_T^2(\dot{B}_{2,1}^{\frac{N}{2}})}\lesssim \|\delta{\bf u}\|_{L_T^2(\dot{B}_{2,1}^{\frac{N}{2}-1})}\|{\bf u}_2\|_{L_T^2(\dot{B}_{2,1}^{\frac{N}{2}})},
\ee

\be\label{4.6}
\|\delta{\bf u}\cdot\nabla{\bf u}_1\|_{L_T^1(\dot{B}_{2,1}^{\frac{N}{2}-2})}\lesssim
\|\delta{\bf u}\|_{L_T^2(\dot{B}_{2,1}^{\frac{N}{2}-1})}\|\nabla{\bf u}_1\|_{L_T^2(\dot{B}_{2,1}^{\frac{N}{2}-1})}\lesssim \|\delta{\bf u}\|_{L_T^2(\dot{B}_{2,1}^{\frac{N}{2}-1})}\|{\bf u}_1\|_{L_T^2(\dot{B}_{2,1}^{\frac{N}{2}})},
\ee

\be\label{4.7}
\begin{split}
&{\Big\|}(I(b_2)-I(b_1))\mathcal{A}{\bf u}_2{\Big\|}_{L_T^1(\dot{B}_{2,1}^{\frac{N}{2}-2})}
\lesssim \|I(b_2)-I(b_1)\|_{L_T^\infty(\dot{B}_{2,1}^{\frac{N}{2}-1})}
\|\mathcal{A}{\bf u}_2\|_{L_T^1(\dot{B}_{2,1}^{\frac{N}{2}-1})}\\[2mm]
&\quad \lesssim {\Big(}1+\|b_1\|_{L_T^\infty(\dot{B}_{2,1}^{\frac{N}{2}})}
+\|b_2\|_{L_T^\infty(\dot{B}_{2,1}^{\frac{N}{2}})}{\Big)}\|\nabla^2{\bf u}_2\|_{L_T^1(\dot{B}_{2,1}^{\frac{N}{2}-1})}\|\delta b\|_{L_T^\infty(\dot{B}_{2,1}^{\frac{N}{2}-1})}\\[2mm]
&\quad \lesssim {\Big(}1+\|b_1\|_{L_T^\infty(\dot{B}_{2,1}^{\frac{N}{2}})}
+\|b_2\|_{L_T^\infty(\dot{B}_{2,1}^{\frac{N}{2}})}{\Big)}\|{\bf u}_2\|_{L_T^1(\dot{B}_{2,1}^{\frac{N}{2}+1})}\|\delta b\|_{L_T^\infty(\dot{B}_{2,1}^{\frac{N}{2}-1})},
\end{split}
\ee

\be\label{4.8}
\|I(b_1)\mathcal{A}\delta{\bf u}\|_{L_T^1(\dot{B}_{2,1}^{\frac{N}{2}-2})}\lesssim
\|b_1\|_{L_T^\infty(\dot{B}_{2,1}^{\frac{N}{2}})}\|\nabla^2\delta {\bf u}\|
_{L_T^1(\dot{B}_{2,1}^{\frac{N}{2}-2})}\lesssim
\|b_1\|_{L_T^\infty(\dot{B}_{2,1}^{\frac{N}{2}})}\|\delta {\bf u}\|
_{L_T^1(\dot{B}_{2,1}^{\frac{N}{2}})},
\ee

\be\label{4.9}
\begin{split}
&{\Big\|}\xi(I(b_2)-I(b_1)){\rm div}{\Big(}\nabla({\bf d}_2-\hat{\bf d})\odot\nabla({\bf d}_2-\hat{\bf d})-\frac{1}{2}|\nabla({\bf d}_2-\hat{\bf d})|^2{\bf I}{\Big)}{\Big\|}_{L_T^1(\dot{B}_{2,1}^{\frac{N}{2}-2})}\\[2mm]
&\quad\lesssim\|I(b_2)-I(b_1)\|_{L_T^\infty(\dot{B}_{2,1}^{\frac{N}{2}-1})}
\||\nabla({\bf d}_2-\hat{\bf d})|^2\|_{L_T^1(\dot{B}_{2,1}^{\frac{N}{2}})}\\[2mm]
&\quad\lesssim {\Big(}1+\|b_1\|_{L_T^\infty(\dot{B}_{2,1}^{\frac{N}{2}})}
+\|b_2\|_{L_T^\infty(\dot{B}_{2,1}^{\frac{N}{2}})}{\Big)}\|{\bf d}_2-\hat{\bf d}\|_{L_T^2(\dot{B}_{2,1}^{\frac{N}{2}+1})}^2\|\delta b\|_{L_T^\infty(\dot{B}_{2,1}^{\frac{N}{2}-1})},
\end{split}
\ee

\be\label{4.10}
\begin{split}
&{\Big\|}\frac{\xi}{1+b_1}{\rm div}{\Big(}\nabla({\bf d}_1-\hat{\bf d})\odot\nabla\delta{\bf d}+\nabla({\bf d}_2-\hat{\bf d})\odot\nabla\delta{\bf d}{\Big)}{\Big\|}_{L_T^1(\dot{B}_{2,1}^{\frac{N}{2}-2})}\\[2mm]
&\quad\lesssim {\Big(}1+\|b_1\|_{L_T^\infty(\dot{B}_{2,1}^{\frac{N}{2}})}{\Big)}
\|\nabla\delta{\bf d}\|_{L_T^2(\dot{B}_{2,1}^{\frac{N}{2}-1})}
{\Big(}\|\nabla({\bf d}_1-\hat{\bf d})\|_{L_T^2(\dot{B}_{2,1}^{\frac{N}{2}})}+\|\nabla({\bf d}_2-\hat{\bf d})\|_{L_T^2(\dot{B}_{2,1}^{\frac{N}{2}})}{\Big)}\\[2mm]
&\quad\lesssim {\Big(}1+\|b_1\|_{L_T^\infty(\dot{B}_{2,1}^{\frac{N}{2}})}{\Big)}
\|\delta{\bf d}\|_{L_T^2(\dot{B}_{2,1}^{\frac{N}{2}})}
{\Big(}\|{\bf d}_1-\hat{\bf d}\|_{L_T^2(\dot{B}_{2,1}^{\frac{N}{2}+1})}+\|{\bf d}_2-\hat{\bf d}\|_{L_T^2(\dot{B}_{2,1}^{\frac{N}{2}+1})}{\Big)}.
\end{split}
\ee
Similar to \eqref{4.10}, there holds that
\be\label{4.11}
\begin{split}
&{\Big\|}\frac{\xi}{1+b_1}{\rm div}{\Big(}\frac{1}{2}{\bf I}(\nabla({\bf d}_1-\hat{\bf d}):\nabla\delta{\bf d}+\nabla({\bf d}_2-\hat{\bf d}):\nabla\delta{\bf d}){\Big)}{\Big\|}_{L_T^1(\dot{B}_{2,1}^{\frac{N}{2}-2})}\\[2mm]
&\quad\lesssim {\Big(}1+\|b_1\|_{L_T^\infty(\dot{B}_{2,1}^{\frac{N}{2}})}{\Big)}
\|\delta{\bf d}\|_{L_T^2(\dot{B}_{2,1}^{\frac{N}{2}})}
{\Big(}\|{\bf d}_1-\hat{\bf d}\|_{L_T^2(\dot{B}_{2,1}^{\frac{N}{2}+1})}+\|{\bf d}_2-\hat{\bf d}\|_{L_T^2(\dot{B}_{2,1}^{\frac{N}{2}+1})}{\Big)}.
\end{split}
\ee
Moreover, applying Corollary \ref{co2.1} and Lemma \ref{le2.3}, one easily gets
\be\label{4.12}
\begin{split}
&\|\nabla(\mathcal{K}(b_2)-\mathcal{K}(b_1))\|_{L_T^1(\dot{B}_{2,1}^{\frac{N}{2}-2})}
\lesssim\|\mathcal{K}(b_2)-\mathcal{K}(b_1)\|_{L_T^1(\dot{B}_{2,1}^{\frac{N}{2}-1})}\\[2mm]
&\quad\lesssim T{\Big(}\|b_1\|_{L_T^\infty(\dot{B}_{2,1}^{\frac{N}{2}})}+
\|b_2\|_{L_T^\infty(\dot{B}_{2,1}^{\frac{N}{2}})}{\Big)}\|\delta b\|_{L_T^\infty(\dot{B}_{2,1}^{\frac{N}{2}-1})}.
\end{split}
\ee
Note that in the above inequalities, we have used that $N>2$.  Collecting the above estimates \eqref{4.5}-\eqref{4.12}, it follows that
\be\label{4.13}
\begin{split}
&\|\delta G\|_{L_T^1(\dot{B}_{2,1}^{\frac{N}{2}-2})}\\[2mm]
&\quad\lesssim {\Big(}\|{\bf u}_1\|_{L_T^2(\dot{B}_{2,1}^{\frac{N}{2}})}+\|{\bf u}_2\|_{L_T^2(\dot{B}_{2,1}^{\frac{N}{2}})}{\Big)}\|\delta{\bf u}\|_{L_T^2(\dot{B}_{2,1}^{\frac{N}{2}-1})}
+\|b_1\|_{L_T^\infty(\dot{B}_{2,1}^{\frac{N}{2}})}\|\delta{\bf u}\|_{L_T^1(\dot{B}_{2,1}^{\frac{N}{2}})}\\[2mm]
&\quad\quad +{\Big(}1+\|b_1\|_{L_T^\infty(\dot{B}_{2,1}^{\frac{N}{2}})}
+\|b_2\|_{L_T^\infty(\dot{B}_{2,1}^{\frac{N}{2}})}{\Big)}\|{\bf u}_2\|_{L_T^1(\dot{B}_{2,1}^{\frac{N}{2}+1})}\|\delta b\|_{L_T^\infty(\dot{B}_{2,1}^{\frac{N}{2}-1})}\\[2mm]
&\quad\quad +{\Big(}1+\|b_1\|_{L_T^\infty(\dot{B}_{2,1}^{\frac{N}{2}})}
+\|b_2\|_{L_T^\infty(\dot{B}_{2,1}^{\frac{N}{2}})}{\Big)}\|{\bf d}_2-\hat{\bf d}\|_{L_T^2(\dot{B}_{2,1}^{\frac{N}{2}+1})}^2\|\delta b\|_{L_T^\infty(\dot{B}_{2,1}^{\frac{N}{2}-1})}\\[2mm]
&\quad\quad +{\Big(}1+\|b_1\|_{L_T^\infty(\dot{B}_{2,1}^{\frac{N}{2}})}
{\Big)}{\Big(}\|{\bf d}_1-\hat{\bf d}\|_{L_T^2(\dot{B}_{2,1}^{\frac{N}{2}+1})}+\|{\bf d}_2-\hat{\bf d}\|_{L_T^2(\dot{B}_{2,1}^{\frac{N}{2}+1})}{\Big)}\|\delta {\bf d}\|_{L_T^2(\dot{B}_{2,1}^{\frac{N}{2}})}\\[2mm]
&\quad\quad+T{\Big(}\|b_1\|_{L_T^\infty(\dot{B}_{2,1}^{\frac{N}{2}})}
+\|b_2\|_{L_T^\infty(\dot{B}_{2,1}^{\frac{N}{2}})}{\Big)}\|\delta b\|_{L_T^\infty(\dot{B}_{2,1}^{\frac{N}{2}-1})}.
\end{split}
\ee
For the estimate of $\delta\mathcal{M}$, note that
\be\label{4.14}
\begin{split}
&\|\delta{\bf u}\cdot\nabla{\bf d}_1\|_{L_T^1(\dot{B}_{2,1}^{\frac{N}{2}-1})}
=\|\delta{\bf u}\cdot\nabla({\bf d}_1-\hat{\bf d})\|_{L_T^1(\dot{B}_{2,1}^{\frac{N}{2}-1})}\\[2mm]
&\quad\lesssim\|\delta{\bf u}\|_{L_T^2(\dot{B}_{2,1}^{\frac{N}{2}-1})}\|\nabla({\bf d}_1-\hat{\bf d})\|_{L_T^2(\dot{B}_{2,1}^{\frac{N}{2}})}\lesssim\|\delta{\bf u}\|_{L_T^2(\dot{B}_{2,1}^{\frac{N}{2}-1})}\|{\bf d}_1-\hat{\bf d}\|_{L_T^2(\dot{B}_{2,1}^{\frac{N}{2}+1})},
\end{split}
\ee

\be\label{4.15}
\begin{split}
\|\theta|\nabla({\bf d}_1-\hat{\bf d})|^2\delta{\bf d}\|_{L_T^1(\dot{B}_{2,1}^{\frac{N}{2}-1})}&\lesssim \|\delta{\bf d}\|_{L_T^\infty(\dot{B}_{2,1}^{\frac{N}{2}-1})}\||\nabla({\bf d}_1-\hat{\bf d})|^2\|_{L_T^1(\dot{B}_{2,1}^{\frac{N}{2}})}\\[2mm]
&\lesssim\|\delta{\bf d}\|_{L_T^\infty(\dot{B}_{2,1}^{\frac{N}{2}-1})}\|{\bf d}_1-\hat{\bf d}\|_{L_T^2(\dot{B}_{2,1}^{\frac{N}{2}+1})}^2,
\end{split}
\ee

\be\label{4.16}
\begin{split}
&\|\theta(\nabla({\bf d}_1-\hat{\bf d}):\nabla\delta{\bf d}+\nabla({\bf d}_2-\hat{\bf d}):\nabla\delta{\bf d}){\bf d}_2\|_{L_T^1(\dot{B}_{2,1}^{\frac{N}{2}-1})}\\[2mm]
&\quad\lesssim \|(\nabla({\bf d}_1-\hat{\bf d}):\nabla\delta{\bf d}+\nabla({\bf d}_2-\hat{\bf d}):\nabla\delta{\bf d})({\bf d}_2-\hat{\bf d})\|_{L_T^1(\dot{B}_{2,1}^{\frac{N}{2}-1})}\\[2mm]
&\quad\quad+\|(\nabla({\bf d}_1-\hat{\bf d}):\nabla\delta{\bf d}+\nabla({\bf d}_2-\hat{\bf d}):\nabla\delta{\bf d})\hat{\bf d}\|_{L_T^1(\dot{B}_{2,1}^{\frac{N}{2}-1})}\\[2mm]
&\quad\lesssim \|{\bf d}_2-\hat{\bf d}\|_{L_T^\infty(\dot{B}_{2,1}^{\frac{N}{2}})}{\Big(}\|{\bf d}_1-\hat{\bf d}\|_{L_T^2(\dot{B}_{2,1}^{\frac{N}{2}+1})}+\|{\bf d}_2-\hat{\bf d}\|_{L_T^2(\dot{B}_{2,1}^{\frac{N}{2}+1})}{\Big)}
\|\delta{\bf d}\|_{L_T^2(\dot{B}_{2,1}^{\frac{N}{2}})}\\[2mm]
&\quad\quad+{\Big(}\|{\bf d}_1-\hat{\bf d}\|_{L_T^2(\dot{B}_{2,1}^{\frac{N}{2}+1})}+\|{\bf d}_2-\hat{\bf d}\|_{L_T^2(\dot{B}_{2,1}^{\frac{N}{2}+1})}{\Big)}
\|\delta{\bf d}\|_{L_T^2(\dot{B}_{2,1}^{\frac{N}{2}})}\\[2mm]
&\quad\lesssim {\Big(}1+\|{\bf d}_2-\hat{\bf d}\|_{L_T^\infty(\dot{B}_{2,1}^{\frac{N}{2}})}{\Big)}{\Big(}\|{\bf d}_1-\hat{\bf d}\|_{L_T^2(\dot{B}_{2,1}^{\frac{N}{2}+1})}+\|{\bf d}_2-\hat{\bf d}\|_{L_T^2(\dot{B}_{2,1}^{\frac{N}{2}+1})}{\Big)}
\|\delta{\bf d}\|_{L_T^2(\dot{B}_{2,1}^{\frac{N}{2}})}.
\end{split}
\ee
Therefore,
\be\label{4.17}
\begin{split}
&\|\delta\mathcal{M}\|_{L_T^1(\dot{B}_{2,1}^{\frac{N}{2}-1})}\\[2mm]
&\quad\lesssim\|\delta{\bf u}\|_{L_T^2(\dot{B}_{2,1}^{\frac{N}{2}-1})}\|{\bf d}_1-\hat{\bf d}\|_{L_T^2(\dot{B}_{2,1}^{\frac{N}{2}+1})}+\|\delta{\bf d}\|_{L_T^\infty(\dot{B}_{2,1}^{\frac{N}{2}-1})}\|{\bf d}_1-\hat{\bf d}\|_{L_T^2(\dot{B}_{2,1}^{\frac{N}{2}+1})}^2\\[2mm]
&\quad\quad+{\Big(}1+\|{\bf d}_2-\hat{\bf d}\|_{L_T^\infty(\dot{B}_{2,1}^{\frac{N}{2}})}{\Big)}{\Big(}\|{\bf d}_1-\hat{\bf d}\|_{L_T^2(\dot{B}_{2,1}^{\frac{N}{2}+1})}+\|{\bf d}_2-\hat{\bf d}\|_{L_T^2(\dot{B}_{2,1}^{\frac{N}{2}+1})}{\Big)}
\|\delta{\bf d}\|_{L_T^2(\dot{B}_{2,1}^{\frac{N}{2}})}.
\end{split}
\ee
Consequently, it follows that
\bess
\|(\delta b, \delta{\bf u}, \delta{\bf d})\|_{\mathfrak{B}_\nu^{\frac{N}{2}-1}(T)}\leq Z(T)\|(\delta b, \delta{\bf u}, \delta{\bf d})\|_{\mathfrak{B}_\nu^{\frac{N}{2}-1}(T)},
\eess
with
\bess
\begin{split}
Z(T)&=C\exp{\Big(}C\|{\bf u}_2\|_{L_T^1(\dot{B}_{2,1}^{\frac{N}{2}+1})}{\Big)}
{\Big\{}\|b_1\|_{L_T^\infty(\tilde{B}_\nu^{\frac{N}{2}, \infty})}+\|{\bf u}_2\|_{L_T^1(\dot{B}_{2,1}^{\frac{N}{2}+1})}+\|{\bf u}_1\|_{L_T^2(\dot{B}_{2,1}^{\frac{N}{2}})}\\[2mm]
&\quad+\|{\bf u}_2\|_{L_T^2(\dot{B}_{2,1}^{\frac{N}{2}})}+{\Big(}1+\|b_1\|_{L_T^\infty(\dot{B}_{2,1}^{\frac{N}{2}})}
+\|b_2\|_{L_T^\infty(\dot{B}_{2,1}^{\frac{N}{2}})}{\Big)}\\[2mm]
&\quad \times{\Big(}\|{\bf u}_2\|_{L_T^1(\dot{B}_{2,1}^{\frac{N}{2}+1})}
+\|{\bf d}_1-\hat{\bf d}\|_{L_T^2(\dot{B}_{2,1}^{\frac{N}{2}+1})}^2+\|{\bf d}_2-\hat{\bf d}\|_{L_T^2(\dot{B}_{2,1}^{\frac{N}{2}+1})}^2{\Big)}\\[2mm]
&\quad +T{\Big(}\|b_1\|_{L_T^\infty(\dot{B}_{2,1}^{\frac{N}{2}})}
+\|b_2\|_{L_T^\infty(\dot{B}_{2,1}^{\frac{N}{2}})}{\Big)}+
{\Big(}1+\|{\bf d}_2-\hat{\bf d}\|_{L_T^\infty(\dot{B}_{2,1}^{\frac{N}{2}})}+\|b_1\|_{L_T^\infty(\dot{B}_{2,1}^{\frac{N}{2}})}{\Big)}\\[2mm]
&\quad\times{\Big(}\|{\bf d}_1-\hat{\bf d}\|_{L_T^2(\dot{B}_{2,1}^{\frac{N}{2}+1})}+\|{\bf d}_2-\hat{\bf d}\|_{L_T^2(\dot{B}_{2,1}^{\frac{N}{2}+1})}{\Big)}{\Big\}}.
\end{split}
\eess
It is now clear that
\bess
\limsup_{T\rightarrow 0^+}Z(T)\leq C\|b_1\|_{L_T^\infty(\tilde{B}_\nu^{\frac{N}{2}, \infty})}.
\eess
Thus, if $\eta>0$ is sufficiently small, we have
\bess
\|(\delta b, \delta{\bf u}, \delta{\bf d})\|_{\mathfrak{B}_\nu^{\frac{N}{2}-1}(T)}=0
\eess
for certain $T>0$ small enough. So $(b_1, {\bf u}_1, {\bf d}_1)\equiv (b_2, {\bf u}_2, {\bf d}_2)$ on $[0, T]$.

Then we could argue as in \cite{danchin2000global} for the isentropic compressible Navier-Stokes equations. Let $T_m$ be the largest time such that the two solutions coincide on $[0, T_m]$. Taking $T_m$ as the initial time, we define
\bess
(\tilde{b}_i(t), \tilde{{\bf u}}_i(t), \tilde{{\bf d}}_i(t)):=(b_i(t+T_m), {\bf u}_i(t+T_m), {\bf d}_i(t+T_m)).
\eess
Repeating the above arguments and using the fact that $\|\tilde{b}_i(0)\|_{L^\infty}\leq\frac{1}{4}$ (see \eqref{4.18}), we can prove that
\bess
(\tilde{b}_1(t), \tilde{{\bf u}}_1(t), \tilde{{\bf d}}_1(t))=(\tilde{b}_2(t), \tilde{{\bf u}}_2(t), \tilde{{\bf d}}_2(t))
\eess
on a suitably small interval $[0,\epsilon]\,(\epsilon>0)$. This contradicts the assumption that $T_m$ is the largest time such that the two solutions coincide. As a result, $T_m=\infty$, which means that the uniqueness result holds in $\mathbb{R}^+$.

{\bf Case $N=2$}: Using again  Propositions \ref{pr3.1} and \ref{pr3.2}, it follows that
\be\label{4.3}
\begin{split}
&\|(\delta b, \delta{\bf u}, \delta{\bf d})\|_{\mathfrak{B}_\nu^{s-1}(T)}\\[2mm]
&\quad\leq C\exp{\Big(}C\|{\bf u}_2\|_{L_T^1(\dot{B}_{2,1}^{2})}{\Big)}{\Big(}\|\delta F\|_{L_T^1(\tilde{B}_\nu^{s-1,\infty})}+\|\delta G\|_{L_T^1(\dot{B}_{2,1}^{s-2})}+\|\delta \mathcal{M}\|_{L_T^1(\dot{B}_{2,1}^{s-1})}{\Big)}.
\end{split}
\ee
Let us estimate $\delta F, \delta G$ and $\delta\mathcal{M}$, respectively.
\bess
\begin{split}
\|\delta F\|_{L_T^1(\tilde{B}_\nu^{s-1,\infty})}&\lesssim T^{\frac{s-1}{2}}\|\delta{\bf u}\|_{L_T^{\frac{2}{3-s}}(\dot{B}_{2,1}^{1})}
\|b_1\|_{L_T^\infty(\tilde{B}_\nu^{s,\infty})}+\|\delta{\bf u}\|_{L_T^1(\dot{B}_{2,1}^{s})}
\|b_1\|_{L_T^\infty(\tilde{B}_\nu^{1,\infty})}\\[2mm]
&\quad+\|\delta b\|_{L_T^\infty(\tilde{B}_\nu^{s-1,\infty})}\|{\bf u}_2\|_{L_T^1(\dot{B}_{2,1}^{2})}\\[2mm]
&\lesssim T^{\frac{s-1}{2}}\|\delta{\bf u}\|_{L_T^{1}(\dot{B}_{2,1}^{s})}^{\frac{3-s}{2}}\|\delta{\bf u}\|_{L_T^{\infty}(\dot{B}_{2,1}^{s-2})}^{\frac{s-1}{2}}
\|b_1\|_{L_T^\infty(\tilde{B}_\nu^{s,\infty})}+\|\delta{\bf u}\|_{L_T^1(\dot{B}_{2,1}^{s})}
\|b_1\|_{L_T^\infty(\tilde{B}_\nu^{1,\infty})}\\[2mm]
&\quad+\|\delta b\|_{L_T^\infty(\tilde{B}_\nu^{s-1,\infty})}\|{\bf u}_2\|_{L_T^1(\dot{B}_{2,1}^{2})}\\[2mm]
&\lesssim T^{\frac{s-1}{2}}{\Big(}\|\delta{\bf u}\|_{L_T^{1}(\dot{B}_{2,1}^{s})}+\|\delta{\bf u}\|_{L_T^{\infty}(\dot{B}_{2,1}^{s-2})}{\Big)}
\|b_1\|_{L_T^\infty(\tilde{B}_\nu^{s,\infty})}\\[2mm]
&\quad+\|\delta{\bf u}\|_{L_T^1(\dot{B}_{2,1}^{s})}
\|b_1\|_{L_T^\infty(\tilde{B}_\nu^{1,\infty})}
+\|\delta b\|_{L_T^\infty(\tilde{B}_\nu^{s-1,\infty})}\|{\bf u}_2\|_{L_T^1(\dot{B}_{2,1}^{2})},
\end{split}
\eess
where we have used the following inequality:
$$\|\delta{\bf u}\|_{L_T^{\frac{2}{3-s}}(\dot{B}_{2,1}^{1})}\leq\|\delta{\bf u}\|_{L_T^{1}(\dot{B}_{2,1}^{s})}^{\frac{3-s}{2}}\|\delta{\bf u}\|_{L_T^{\infty}(\dot{B}_{2,1}^{s-2})}^{\frac{s-1}{2}}\leq
\|\delta{\bf u}\|_{L_T^{1}(\dot{B}_{2,1}^{s})}+\|\delta{\bf u}\|_{L_T^{\infty}(\dot{B}_{2,1}^{s-2})}. $$
Moreover,
\bess
\begin{split}
&\|\delta G\|_{L_T^1(\dot{B}_{2,1}^{s-2})}\\[2mm]
&\quad\lesssim {\Big(}\|{\bf u}_1\|_{L_T^2(\dot{B}_{2,1}^{1})}+\|{\bf u}_2\|_{L_T^2(\dot{B}_{2,1}^{1})}{\Big)}\|\delta{\bf u}\|_{L_T^2(\dot{B}_{2,1}^{s-1})}
+\|b_1\|_{L_T^\infty(\dot{B}_{2,1}^{1})}\|\delta{\bf u}\|_{L_T^1(\dot{B}_{2,1}^{s})}\\[2mm]
&\quad\quad +{\Big(}1+\|b_1\|_{L_T^\infty(\dot{B}_{2,1}^{1})}
+\|b_2\|_{L_T^\infty(\dot{B}_{2,1}^{1})}{\Big)}\|{\bf u}_2\|_{L_T^1(\dot{B}_{2,1}^{2})}\|\delta b\|_{L_T^\infty(\dot{B}_{2,1}^{s-1})}\\[2mm]
&\quad\quad +{\Big(}1+\|b_1\|_{L_T^\infty(\dot{B}_{2,1}^{1})}
+\|b_2\|_{L_T^\infty(\dot{B}_{2,1}^{1})}{\Big)}\|{\bf d}_2-\hat{\bf d}\|_{L_T^2(\dot{B}_{2,1}^{2})}^2\|\delta b\|_{L_T^\infty(\dot{B}_{2,1}^{s-1})}\\[2mm]
&\quad\quad +{\Big(}1+\|b_1\|_{L_T^\infty(\dot{B}_{2,1}^{1})}
{\Big)}{\Big(}\|{\bf d}_1-\hat{\bf d}\|_{L_T^2(\dot{B}_{2,1}^{2})}+\|{\bf d}_2-\hat{\bf d}\|_{L_T^2(\dot{B}_{2,1}^{2})}{\Big)}\|\delta {\bf d}\|_{L_T^2(\dot{B}_{2,1}^{s})}\\[2mm]
&\quad\quad+T{\Big(}\|b_1\|_{L_T^\infty(\dot{B}_{2,1}^{1})}
+\|b_2\|_{L_T^\infty(\dot{B}_{2,1}^{1})}{\Big)}\|\delta b\|_{L_T^\infty(\dot{B}_{2,1}^{s-1})},
\end{split}
\eess

\bess
\begin{split}
&\|\delta\mathcal{M}\|_{L_T^1(\dot{B}_{2,1}^{s-1})}\\[2mm]
&\quad\lesssim\|\delta{\bf u}\|_{L_T^2(\dot{B}_{2,1}^{s-1})}\|{\bf d}_1-\hat{\bf d}\|_{L_T^2(\dot{B}_{2,1}^{2})}+\|\delta{\bf d}\|_{L_T^\infty(\dot{B}_{2,1}^{s-1})}\|{\bf d}_1-\hat{\bf d}\|_{L_T^2(\dot{B}_{2,1}^{2})}^2\\[2mm]
&\quad\quad+{\Big(}1+\|{\bf d}_2-\hat{\bf d}\|_{L_T^\infty(\dot{B}_{2,1}^{1})}{\Big)}{\Big(}\|{\bf d}_1-\hat{\bf d}\|_{L_T^2(\dot{B}_{2,1}^{2})}+\|{\bf d}_2-\hat{\bf d}\|_{L_T^2(\dot{B}_{2,1}^{2})}{\Big)}
\|\delta{\bf d}\|_{L_T^2(\dot{B}_{2,1}^{s})}.
\end{split}
\eess
Consequently, it follows that
\bess
\|(\delta b, \delta{\bf u}, \delta{\bf d})\|_{\mathfrak{B}_\nu^{s-1}(T)}\leq Z(T)\|(\delta b, \delta{\bf u}, \delta{\bf d})\|_{\mathfrak{B}_\nu^{s-1}(T)}
\eess
with
\bess
\begin{split}
Z(T)&=C\exp{\Big(}C\|{\bf u}_2\|_{L_T^1(\dot{B}_{2,1}^{2})}{\Big)}
{\Big\{}\|b_1\|_{L_T^\infty(\tilde{B}_\nu^{1, \infty})}+\|{\bf u}_2\|_{L_T^1(\dot{B}_{2,1}^{2})}+\|{\bf u}_1\|_{L_T^2(\dot{B}_{2,1}^{1})}\\[2mm]
&\quad+\|{\bf u}_2\|_{L_T^2(\dot{B}_{2,1}^{1})}+{\Big(}1+\|b_1\|_{L_T^\infty(\dot{B}_{2,1}^{1})}
+\|b_2\|_{L_T^\infty(\dot{B}_{2,1}^{1})}{\Big)}\\[2mm]
&\quad \times{\Big(}\|{\bf u}_2\|_{L_T^1(\dot{B}_{2,1}^{2})}
+\|{\bf d}_1-\hat{\bf d}\|_{L_T^2(\dot{B}_{2,1}^{2})}^2+\|{\bf d}_2-\hat{\bf d}\|_{L_T^2(\dot{B}_{2,1}^{2})}^2{\Big)}\\[2mm]
&\quad +T{\Big(}\|b_1\|_{L_T^\infty(\dot{B}_{2,1}^{1})}
+\|b_2\|_{L_T^\infty(\dot{B}_{2,1}^{1})}{\Big)}+
{\Big(}1+\|{\bf d}_2-\hat{\bf d}\|_{L_T^\infty(\dot{B}_{2,1}^{1})}+\|b_1\|_{L_T^\infty(\dot{B}_{2,1}^{1})}{\Big)}\\[2mm]
&\quad\times{\Big(}\|{\bf d}_1-\hat{\bf d}\|_{L_T^2(\dot{B}_{2,1}^{2})}+\|{\bf d}_2-\hat{\bf d}\|_{L_T^2(\dot{B}_{2,1}^{2})}{\Big)}+T^{1-\frac{1}{r}}
\|b_1\|_{L_T^\infty(\tilde{B}_\nu^{s,\infty})}{\Big\}}.
\end{split}
\eess
It is clear that
\bess
\limsup_{T\rightarrow 0^+}Z(T)\leq C\|b_1\|_{L_T^\infty(\tilde{B}_\nu^{1, \infty})}.
\eess
Thus, if $\eta>0$ is sufficiently small, we have
\bess
\|(\delta b, \delta{\bf u}, \delta{\bf d})\|_{\mathfrak{B}_\nu^{s-1}(T)}=0
\eess
for certain $T>0$ small enough. So $(b_1, {\bf u}_1, {\bf d}_1)\equiv (b_2, {\bf u}_2, {\bf d}_2)$ on $[0, T]$. We can now achieve the proof as in the case $N\geq 3$.\hspace{11cm}$\Box$

\section{Incompressible limit}
\setcounter{equation}{0}\setcounter{section}{5}\indent
This part is devoted to the proof of Theorem \ref{th1.2}. Firstly we list the global well-posedness of incompressible model \eqref{1.3}.
\begin{prop}\label{pr5.2}{\rm(}see {\rm \cite{Xu2014well-posedness}}{\rm )}
Let $N\geq 2$. Suppose that the initial data $({\bf u}_0, {\bf d}_0)$ belong to $\dot{B}_{2,1}^{\frac{N}{2}-1}(\mathbb{R}^N)\times\dot{B}_{2,1}^{\frac{N}{2}}
(\mathbb{S}^{N-1})$ with ${\rm div}{\bf u}_0=0$. Then there exist two positive constants $\bar{c}$ and $\bar{C}$, depending on $N, \mu, \xi$ and $\theta$, such that if
$$
\|{\bf u}_0\|_{\dot{B}_{2,1}^{\frac{N}{2}-1}}+\|{\bf d}_0-\hat{{\bf d}}\|_{\dot{B}_{2,1}^{\frac{N}{2}}}\leq \bar{c},
$$
then there exists a unique global solution $({\bf u},{\bf d})$ to \eqref{1.3} in the class
$$
L^\infty([0,\infty); \dot{B}_{2,1}^{\frac{N}{2}-1})\cap L^1([0,\infty); \dot{B}_{2,1}^{\frac{N}{2}+1})\times L^\infty([0,\infty); \dot{B}_{2,1}^{\frac{N}{2}})\cap L^1([0,\infty); \dot{B}_{2,1}^{\frac{N}{2}+2}).
$$
Moreover, the following estimate holds true
\be\label{5.1}
\begin{split}
&\|{\bf u}\|_{\tilde{L}_T^\infty(\dot{B}_{2,1}^{\frac{N}{2}-1})}
+\|{\bf u}\|_{{L}_T^1(\dot{B}_{2,1}^{\frac{N}{2}+1})}+\|{\bf d}-\hat{{\bf d}}\|_{\tilde{L}_T^\infty(\dot{B}_{2,1}^{\frac{N}{2}})}
+\|{\bf d}-\hat{{\bf d}}\|_{{L}_T^1(\dot{B}_{2,1}^{\frac{N}{2}+2})}\\[2mm]
&\quad\leq \bar{C}{\Big(}\|{\bf u}_0\|_{\dot{B}_{2,1}^{\frac{N}{2}-1}}+\|{\bf d}_0-\hat{{\bf d}}\|_{\dot{B}_{2,1}^{\frac{N}{2}}}{\Big)}.
\end{split}
\ee
\end{prop}

The study of incompressible limit problem of \eqref{1.2} relies on Strichartz estimates for the following system of acoustics:
\begin{prop}\label{pr5.1}
{\rm(}see {\rm \cite{bahouri2011fourier}}{\rm )} Let $(b, v)$ be a solution of the following system of acoustics:
\begin{equation}\label{5.2}
 \left\{\begin{array}{ll}\medskip\displaystyle\partial_t b+\epsilon^{-1}\Lambda v=F,\\[2mm]
 \medskip\displaystyle\partial_t v-\epsilon^{-1}\Lambda b=G,\\
 \displaystyle (b, v)_{t=0}=(b_0, v_0),
 \end{array}
 \right.
 \end{equation}
where $\Lambda=\sqrt{-\Delta}$. Then, for any $s\in\mathbb{R}$ and $T\in (0, \infty]$, the following estimate holds:
\be\label{5.3}
\|(b,v)\|_{\tilde{L}_T^r(\dot{B}_{p,1}^{s+N(\frac{1}{p}-\frac{1}{2})+\frac{1}{r}})}
\leq C\epsilon^{\frac{1}{r}}\|(b_0, v_0)\|_{\dot{B}_{2,1}^s}+\epsilon^{1+\frac{1}{r}-\frac{1}{\bar{r}^\prime}}
\|(F,G)\|_{\tilde{L}_T^{\bar{r}^\prime}
(\dot{B}_{\bar{p}^\prime,1}^{s+N(\frac{1}{\bar{p}^\prime}-\frac{1}{2})+\frac{1}{\bar{r}^\prime}-1})},
\ee
with
\be\nonumber
\left.
\begin{array}{ll}\medskip\D
p\geq 2,~~~\frac{2}{r}\leq \min(1, \gamma(p)), ~~~~(r, p, N)\neq (2,\infty,3),\\[2mm]
\medskip\D \bar{p}\geq 2,~~~\frac{2}{\bar{r}}\leq \min(1, \gamma(\bar{p})), ~~~~(\bar{r}, \bar{p}, N)\neq (2,\infty,3),
\end{array}
\right.
\ee
where $\gamma(q):=(N-1)(\frac{1}{2}-\frac{1}{q}),~ \frac{1}{\bar{p}}+\frac{1}{\bar{p}^\prime}=1,$ and $\frac{1}{\bar{r}}+\frac{1}{\bar{r}^\prime}=1$.
\end{prop}

Let us make the following change of functions:
$$
c(t,x):=\epsilon b^\epsilon(\epsilon^2t,\epsilon x),~~~{\bf v}(t,x):=\epsilon {\bf u}^\epsilon(\epsilon^2t,\epsilon x),~~~{\bf h}(t,x):={\bf d}^\epsilon(\epsilon^2t,\epsilon x).
$$
Then $(b^\epsilon, {\bf u}^\epsilon, {\bf d}^\epsilon)$ solves \eqref{1.4} if and only if $(c, {\bf v}, {\bf h})$ solves
\begin{equation}\label{5.4}
 \left\{\begin{array}{ll}\medskip\displaystyle\partial_t c+{\rm div}{\bf v}=-{\rm div}(c{\bf v}),\\[2mm]
 \medskip\displaystyle\partial_t{\bf v}+{\bf v}\cdot\nabla{\bf v}-\frac{\mu\Delta{\bf v}+(\mu+\lambda)\nabla{\rm div}{\bf v}}{1+c}+\frac{P^\prime(1+c)}{1+c}\nabla c\\[2mm]
 \medskip\quad\D=\frac{-\xi}{1+c}{\rm div}(\nabla{\bf h}\odot\nabla{\bf h}-\frac{1}{2}|\nabla{\bf h}|^2{\bf I}),\\[2mm]
 \medskip\displaystyle\partial_t{\bf h}+{\bf v}\cdot\nabla{\bf h}=\theta(\Delta{\bf h}+|\nabla{\bf h}|^2{\bf h}),\\[2mm]
 \D (c, {\bf v}, {\bf h})|_{t=0}=(c_0, {\bf v}_0, {\bf h}_0),
 \end{array}
 \right.
 \end{equation}
with $(c_0, {\bf v}_0, {\bf h}_0):=(\epsilon b_0^\epsilon({\epsilon x}), \epsilon {\bf u}_0^\epsilon(\epsilon x), {\bf d}_0^\epsilon(\epsilon x))$.

According to Theorem \ref{th1.1}, there exist two positive constants $\eta=\eta (N, \lambda, \mu, \xi, \theta, P)$ and $M=M(N, \lambda, \mu, \xi, \theta, P)$ such that \eqref{5.4} has a solution $(c, {\bf v}, {\bf h})$ in $\mathfrak{B}_\nu^{\frac{N}{2}}$ as soon as
$$
\|c_0\|_{\tilde{B}_\nu^{\frac{N}{2},\infty}}+\|{\bf v}_0\|_{\dot{B}_{2,1}^{\frac{N}{2}-1}}+\|{\bf h}_0-\hat{{\bf d}}\|_{\dot{B}_{2,1}^{\frac{N}{2}}}\leq \eta.
$$
In addition, we have the estimate
$$
\|(c, {\bf v}, {\bf h}-\hat{{\bf d}})\|_{\mathfrak{B}_\nu^{\frac{N}{2}}}\leq M {\Big(}\|c_0\|_{\tilde{B}_\nu^{\frac{N}{2},\infty}}+\|{\bf v}_0\|_{\dot{B}_{2,1}^{\frac{N}{2}-1}}+\|{\bf h}_0-\hat{{\bf d}}\|_{\dot{B}_{2,1}^{\frac{N}{2}}}{\Big)}.
$$
According to the scaling properties of Besov space in Proposition \ref{pr2.1}, it is easy to verify that
$$
\|c_0\|_{\tilde{B}_\nu^{\frac{N}{2},\infty}}+\|{\bf v}_0\|_{\dot{B}_{2,1}^{\frac{N}{2}-1}}+\|{\bf h}_0-\hat{{\bf d}}\|_{\dot{B}_{2,1}^{\frac{N}{2}}}\approx\|b_0^\epsilon\|_{\tilde{B}_\nu^{\frac{N}{2},\infty}}
+\|{\bf u}_0^\epsilon\|_{\dot{B}_{2,1}^{\frac{N}{2}-1}}+\|{\bf d}_0^\epsilon-\hat{{\bf d}}\|_{\dot{B}_{2,1}^{\frac{N}{2}}},
$$
and
$$
\|(c, {\bf v}, {\bf h})\|_{\mathfrak{B}_\nu^{\frac{N}{2}}}\approx\|(b, {\bf u}, {\bf d})\|_{\mathfrak{B}_{\epsilon\nu}^{\frac{N}{2}}},
$$
which combined with Proposition \ref{pr5.2} conclude the  part (i) of Theorem \ref{th1.2}.

Now we turn to the proof of part (ii) of Theorem \ref{th1.2}.

{\bf Case 1}: $N\geq 4$. Let us first focus on the convergence of $(b^\epsilon, {\mathcal{Q}{\bf u}^\epsilon})$. Applying $\mathcal{Q}:=\nabla\Delta^{-1}{\rm div}$ to the second equation of \eqref{1.4}, we conclude that $(b^\epsilon, {\mathcal{Q}{\bf u}^\epsilon})$ satisfies
\begin{equation}\label{5.5}
 \left\{\begin{array}{ll}\medskip\displaystyle\partial_t b^\epsilon+\frac{{\rm div}\mathcal{Q}{\bf u}^\epsilon}{\epsilon}=-{\rm div}(b^\epsilon{\bf u}^\epsilon),\\[2mm]
 \medskip\displaystyle\partial_t\mathcal{Q}{\bf u}^\epsilon-\nu\Delta\mathcal{Q}{\bf u}^\epsilon+\frac{\nabla b^\epsilon}{\epsilon}=\mathcal{Q}{\Big(}
 -{\bf u}^\epsilon\cdot\nabla{\bf u}^\epsilon-\frac{\epsilon b^\epsilon}{1+\epsilon b^\epsilon}\mathcal{A}{\bf u}^\epsilon+K(\epsilon b^\epsilon)\frac{\nabla b^\epsilon}{\epsilon}\\[2mm]
 \quad\D-\frac{\xi}{1+\epsilon b^\epsilon}{\rm div}(\nabla{\bf d}^\epsilon\odot\nabla{\bf d}^\epsilon-\frac{1}{2}|\nabla{\bf d}^\epsilon|^2{\bf I}){\Big)}.
 \end{array}
 \right.
 \end{equation}
Setting $l^\epsilon:=\Lambda^{-1}{\rm div}{\mathcal{Q}}{\bf u}^\epsilon$, then system \eqref{5.5} becomes
 \begin{equation}\label{5.6}
 \left\{\begin{array}{ll}\medskip\displaystyle\partial_t b^\epsilon+\epsilon^{-1}\Lambda l^\epsilon=F^\epsilon,\\[2mm]
 \displaystyle\partial_t l^\epsilon-\epsilon^{-1}\Lambda b^\epsilon=G^\epsilon,
 \end{array}
 \right.
 \end{equation}
with
\be\nonumber
\left.
\begin{array}{ll}
F^\epsilon:=-{\rm div}(b^\epsilon{\bf u}^\epsilon),\\[2mm]
G^\epsilon:=\D\nu\Delta l^\epsilon-\Lambda^{-1}{\rm div}{\Big(}
 {\bf u}^\epsilon\cdot\nabla{\bf u}^\epsilon+\frac{\epsilon b^\epsilon}{1+\epsilon b^\epsilon}\mathcal{A}{\bf u}^\epsilon-K(\epsilon b^\epsilon)\frac{\nabla b^\epsilon}{\epsilon}\\[3mm]
 \quad\quad\quad\D+\frac{\xi}{1+\epsilon b^\epsilon}{\rm div}(\nabla{\bf d}^\epsilon\odot\nabla{\bf d}^\epsilon-\frac{1}{2}|\nabla{\bf d}^\epsilon|^2{\bf I}){\Big)}.
\end{array}
\right.
\ee

Remark that $\mathcal{Q}{\bf u}^\epsilon=-\nabla \Lambda^{-1}l^\epsilon$ so that estimating $\mathcal{Q}{\bf u}^\epsilon$ or $l^\epsilon$ is equivalent (up to an irrelevant constant). Taking $\bar{p}=2,~\bar{r}=\infty,~s=\frac{N}{2}-1$ and $r=2$ in Proposition \ref{pr5.1} yields
$$
\|(b^\epsilon,\mathcal{Q}{\bf u}^\epsilon)\|_{\tilde{L}^2(\dot{B}_{p,1}^{\frac{N}{p}-\frac{1}{2}})}
\leq C\epsilon^{\frac{1}{2}}\|(b_0^\epsilon,\mathcal{Q}{\bf u}_0^\epsilon)\|_{\dot{B}_{2,1}^{\frac{N}{2}-1}}+C\epsilon^{\frac{1}{2}}
\|(F^\epsilon,G^\epsilon)\|_{L^1(\dot{B}_{2,1}^{\frac{N}{2}-1})}
$$
for all $p\geq p_N:=\frac{2(N-1)}{N-3}$. Note that
\be\nonumber
\|{\rm div}(b^\epsilon{\bf u}^\epsilon)\|_{L^1(\dot{B}_{2,1}^{\frac{N}{2}-1})}\lesssim \|b^\epsilon\|_{L^2(\dot{B}_{2,1}^{\frac{N}{2}})}\|{\bf u}^\epsilon\|_{L^2(\dot{B}_{2,1}^{\frac{N}{2}})}\lesssim C_0^{\epsilon\nu},
\ee
\medskip
\be\nonumber
\begin{split}
\|{\bf u}^\epsilon\cdot\nabla{\bf u}^\epsilon\|_{L^1(\dot{B}_{2,1}^{\frac{N}{2}-1})}&\lesssim \|{\bf u}^\epsilon\|_{L^2(\dot{B}_{2,1}^{\frac{N}{2}})}^2\lesssim \|{\bf u}^\epsilon\|_{L^\infty(\dot{B}_{2,1}^{\frac{N}{2}-1})}\|{\bf u}^\epsilon\|_{L^1(\dot{B}_{2,1}^{\frac{N}{2}+1})}
\lesssim C_0^{\epsilon\nu},
\end{split}
\ee
\medskip
\be\nonumber
\begin{split}
{\Big\|}\frac{\epsilon b^\epsilon}{1+\epsilon b^\epsilon}\mathcal{A}{\bf u}^\epsilon{\Big\|}_{L^1(\dot{B}_{2,1}^{\frac{N}{2}-1})}&\lesssim \epsilon\|b^\epsilon\|_{L^\infty(\dot{B}_{2,1}^{\frac{N}{2}})}
\|\mathcal{A}{\bf u}^\epsilon\|_{L^1(\dot{B}_{2,1}^{\frac{N}{2}-1})}\\[2mm]
&\lesssim \nu^{-1}\|b^\epsilon\|_{L^\infty(\tilde{B}_{\epsilon\nu}^{\frac{N}{2},\infty})}
\|{\bf u}^\epsilon\|_{L^1(\dot{B}_{2,1}^{\frac{N}{2}+1})}
\lesssim C_0^{\epsilon\nu},
\end{split}
\ee
\medskip
\be\nonumber
{\Big\|}K(\epsilon b^\epsilon)\frac{\nabla b^\epsilon}{\epsilon}{\Big\|}_{L^1(\dot{B}_{2,1}^{\frac{N}{2}-1})}\lesssim
\|b^\epsilon\|_{L^2(\dot{B}_{2,1}^{\frac{N}{2}})}^2\lesssim C_0^{\epsilon\nu}.
\ee
Moreover,
\be\nonumber
\begin{split}
&{\Big\|}\frac{\xi}{1+\epsilon b^\epsilon}{\rm div}(\nabla{\bf d}^\epsilon\odot\nabla{\bf d}^\epsilon-\frac{1}{2}|\nabla{\bf d}^\epsilon|^2{\bf I}){\Big\|}_{L^1(\dot{B}_{2,1}^{\frac{N}{2}-1})}\\[2mm]
&\quad={\Big\|}\frac{\xi}{1+\epsilon b^\epsilon}{\rm div}{\Big(}\nabla({\bf d}^\epsilon-\hat{{\bf d}})\odot\nabla({\bf d}^\epsilon-\hat{{\bf d}})-\frac{1}{2}|\nabla({\bf d}^\epsilon-\hat{{\bf d}})|^2{\bf I}{\Big)}{\Big\|}_{L^1(\dot{B}_{2,1}^{\frac{N}{2}-1})}\\[2mm]
&\quad\lesssim{\Big(}1+\epsilon\|b^\epsilon\|_{L^\infty(\dot{B}_{2,1}^{\frac{N}{2}})}{\Big)}
\|{\bf d}^\epsilon-\hat{{\bf d}}\|_{L^2(\dot{B}_{2,1}^{\frac{N}{2}+1})}^2
\lesssim C_0^{\epsilon\nu}.
\end{split}
\ee
Collecting all the estimates above, we conclude that
\be\label{5.7}
\|(b^\epsilon,\mathcal{Q}{\bf u}^\epsilon)\|_{\tilde{L}^2(\dot{B}_{p,1}^{\frac{N}{p}-\frac{1}{2}})}
\lesssim C_0^{\epsilon\nu}\epsilon^{\frac{1}{2}},~~~~~p\geq\frac{2(N-1)}{N-3}.
\ee
On the other hand, define $\mathcal{P}:=Id-\mathcal{Q}$, then the rest part $(\mathcal{P}{\bf u}^\epsilon, {\bf d}^\epsilon)$ of system \eqref{1.4} satisfies
\begin{equation}\label{5.71}
 \left\{\begin{array}{ll}
 \medskip\displaystyle\partial_t\mathcal{P}{\bf u}^\epsilon-\mu\Delta\mathcal{P}{\bf u}^\epsilon=\mathcal{P}{\Big(}
 -{\bf u}^\epsilon\cdot\nabla{\bf u}^\epsilon
-\frac{\epsilon b^\epsilon}{1+\epsilon b^\epsilon}\mathcal{A}{\bf u}^\epsilon
-\frac{\xi}{1+\epsilon b^\epsilon}{\rm div}(\nabla{\bf d}^\epsilon\odot\nabla{\bf d}^\epsilon){\Big)},\\[2mm]
\displaystyle\partial_t{\bf d}^\epsilon+{\bf u}^\epsilon\cdot\nabla{\bf d}^\epsilon=\theta(\Delta{\bf d}^\epsilon+|\nabla{\bf d}^\epsilon|^2{\bf d}^\epsilon).
 \end{array}
 \right.
 \end{equation}
 Letting ${\bf w}^\epsilon=\mathcal{P}{\bf u}^\epsilon-{\bf u}$ and $\bar{\bf d}^\epsilon={\bf d}^\epsilon-{\bf d}$, then it follows from \eqref{1.3} and \eqref{5.71} that $({\bf w}^\epsilon, \bar{{\bf d}}^\epsilon)$ solves
 \be\label{5.8}
 \left\{
 \begin{array}{ll}
 \medskip\D\partial_t{\bf w}^\epsilon-\mu\Delta{\bf w}^\epsilon=-\mathcal{P}({\bf u}^\epsilon\cdot\nabla{\bf u}^\epsilon-{\bf u}\cdot\nabla{\bf u})-\mathcal{P}{\Big(}\frac{\epsilon b^\epsilon}{1+\epsilon b^\epsilon}\mathcal{A}{\bf u}^\epsilon{\Big)}\\
 \quad\quad\quad\quad\quad\quad\quad\medskip\D -\mathcal{P}{\Big(}\frac{\xi}{1+\epsilon b^\epsilon}{\rm div}(\nabla{\bf d}^\epsilon\odot\nabla{\bf d}^\epsilon)-\xi{\rm div}(\nabla{\bf d}\odot\nabla{\bf d}){\Big)},\\[3mm]
 \D \partial_t\bar{{\bf d}}^\epsilon-\theta\Delta\bar{{\bf d}}^\epsilon=
 -{\bf u}^\epsilon\cdot\nabla\bar{{\bf d}}^\epsilon-{\bf w}^\epsilon\cdot\nabla{\bf d}-\mathcal{Q}{\bf u}^\epsilon\cdot\nabla{\bf d}+\theta|\nabla {\bf d}|^2
 \bar{\bf d}^\epsilon+\theta|\nabla{\bf d}^\epsilon+\nabla{\bf d}||\nabla\bar{\bf d}^\epsilon|{\bf d}^\epsilon.
 \end{array}
 \right.
 \ee
We infer from the estimates for heat equation (see \cite{bahouri2011fourier, danchin2002zero}) that
\be\label{5.9}
\begin{split}
&\|{\bf w}^\epsilon(t)\|_{\dot{B}_{p,1}^{\frac{N}{p}-\frac{3}{2}}}+\mu\|{\bf w}^\epsilon\|_{L^1(\dot{B}_{p,1}^{\frac{N}{p}+\frac{1}{2}})}\\[2mm]
&\quad\lesssim\|{\bf w}_0^\epsilon\|_{\dot{B}_{p,1}^{\frac{N}{p}-\frac{3}{2}}}+\|\mathcal{P}({\bf u}^\epsilon\cdot\nabla{\bf u}^\epsilon-{\bf u}\cdot\nabla{\bf u})\|_{L^1(\dot{B}_{p,1}^{\frac{N}{p}-\frac{3}{2}})}
+{\Big\|}\mathcal{P}{\Big(}\frac{\epsilon b^\epsilon}{1+\epsilon b^\epsilon}\mathcal{A}{\bf u}^\epsilon{\Big)}{\Big\|}_{L^1(\dot{B}_{p,1}^{\frac{N}{p}-\frac{3}{2}})}\\[2mm]
&\quad\quad+{\Big\|}\mathcal{P}{\Big(}\frac{\xi}{1+\epsilon b^\epsilon}{\rm div}(\nabla{\bf d}^\epsilon\odot\nabla{\bf d}^\epsilon)-\xi{\rm div}(\nabla{\bf d}\odot\nabla{\bf d}){\Big)}{\Big\|}_{L^1(\dot{B}_{p,1}^{\frac{N}{p}-\frac{3}{2}})}.
\end{split}
\ee
By virtue of Corollary \ref{co2.1} and interpolation, there holds
\be\label{5.10}
\begin{split}
{\Big\|}\mathcal{P}{\Big(}\frac{\epsilon b^\epsilon}{1+\epsilon b^\epsilon}\mathcal{A}{\bf u}^\epsilon{\Big)}{\Big\|}_{L^1(\dot{B}_{p,1}^{\frac{N}{p}-\frac{3}{2}})}&\lesssim
{\Big\|}\frac{\epsilon b^\epsilon}{1+\epsilon b^\epsilon}\mathcal{A}{\bf u}^\epsilon{\Big\|}_{L^1(\dot{B}_{2,1}^{\frac{N}{2}-\frac{3}{2}})}
\lesssim\|\epsilon b^\epsilon\|_{L^4(\dot{B}_{2,1}^{\frac{N}{2}})}\|\mathcal{A}{\bf u}^\epsilon\|_{L^{\frac{4}{3}}(\dot{B}_{2,1}^{\frac{N}{2}-\frac{3}{2}})}\\[2mm]
&\lesssim\epsilon^{\frac{1}{2}}
\|b^\epsilon\|_{L^4(\tilde{B}_{\epsilon\nu}^{\frac{N}{2},4})}\|{\bf u}^\epsilon\|_{L^{\frac{4}{3}}(\dot{B}_{2,1}^{\frac{N}{2}+\frac{1}{2}})}
\lesssim (C_0^{\epsilon\nu})^2\epsilon^{\frac{1}{2}},
\end{split}
\ee
where we have used that
$$
b\in L^1(\tilde{B}_{\epsilon\nu}^{\frac{N}{2},1})\cap L^\infty(\tilde{B}_{\epsilon\nu}^{\frac{N}{2},\infty})\Rightarrow b\in L^m(\tilde{B}_{\epsilon\nu}^{\frac{N}{2},m}), ~~~{\rm for}~~~ 1\leq m \leq \infty,
$$
$$
{\bf u}\in L^\infty(\dot{B}_{2,1}^{\frac{N}{2}-1})\cap L^1(\dot{B}_{2,1}^{\frac{N}{2}+1})\Rightarrow {\bf u}\in L^m(\dot{B}_{2,1}^{\frac{N}{2}-1+\frac{2}{m}}),~~~{\rm for}~~~ 1\leq m \leq \infty,
$$
and
$$
\|b^\epsilon\|_{\tilde{B}_{\epsilon\nu}^{\frac{N}{2},4}}
\approx \|b^\epsilon\|_{\dot{B}_{2,1}^{\frac{N}{2}-\frac{1}{2}}}+(\epsilon\nu)^{\frac{1}{2}}
\|b^\epsilon\|_{\dot{B}_{2,1}^{\frac{N}{2}}}\Rightarrow
\|b^\epsilon\|_{\dot{B}_{2,1}^{\frac{N}{2}}}\lesssim
\epsilon^{-\frac{1}{2}}\|b^\epsilon\|_{\tilde{B}_{\epsilon\nu}^{\frac{N}{2},4}}.
$$
Noting that
$$
{\bf u}^\epsilon\cdot\nabla{\bf u}^\epsilon-{\bf u}\cdot\nabla{\bf u}
={\bf u}^\epsilon\cdot\nabla\mathcal{Q}{\bf u}^\epsilon+{\bf u}^\epsilon\cdot\nabla
{\bf w}^\epsilon+\mathcal{Q}{\bf u}^\epsilon\cdot\nabla{\bf u}+{\bf w}^\epsilon\cdot\nabla{\bf u},
$$
which, together with Corollary \ref{co2.1}, \eqref{5.7}, interpolation and Young inequality yields, for any $\delta>0$
\be\label{5.11}
\begin{split}
&\|\mathcal{P}({\bf u}^\epsilon\cdot\nabla{\bf u}^\epsilon-{\bf u}\cdot\nabla{\bf u})\|_{L^1(\dot{B}_{p,1}^{\frac{N}{p}-\frac{3}{2}})}\\[2mm]
&\quad\leq C\|{\bf u}^\epsilon\|_{L^2(\dot{B}_{2,1}^{\frac{N}{2}})}
\|\nabla\mathcal{Q}{\bf u}^\epsilon\|_{L^2(\dot{B}_{p,1}^{\frac{N}{p}-\frac{3}{2}})}
+C\|\nabla{\bf u}\|_{L^2(\dot{B}_{2,1}^{\frac{N}{2}-1})}
\|\mathcal{Q}{\bf u}^\epsilon\|_{L^2(\dot{B}_{p,1}^{\frac{N}{p}-\frac{1}{2}})}\\[2mm]
&\quad\quad\D +C\int_0^t{\Big(}\|{\bf u}^\epsilon\|_{\dot{B}_{2,1}^{\frac{N}{2}}}\|\nabla{\bf w}^\epsilon\|_{\dot{B}_{p,1}^{\frac{N}{p}-\frac{3}{2}}}
+\|\nabla{\bf u}\|_{\dot{B}_{2,1}^{\frac{N}{2}-1}}\|{\bf w}^\epsilon\|_{\dot{B}_{p,1}^{\frac{N}{p}-\frac{1}{2}}}{\Big)}dt^\prime\\[2mm]
&\quad\leq C(C_0^{\epsilon\nu})^2\epsilon^{\frac{1}{2}}
+\delta\|{\bf w}^\epsilon\|_{L^1(\dot{B}_{p,1}^{\frac{N}{p}+\frac{1}{2}})}
+C_{\delta}\int_0^t{\Big(}\|{\bf u}^\epsilon\|_{\dot{B}_{2,1}^{\frac{N}{2}}}^2+\|{\bf u}\|_{\dot{B}_{2,1}^{\frac{N}{2}}}^2{\Big)}\|{\bf w}^\epsilon\|_{\dot{B}_{p,1}^{\frac{N}{p}-\frac{3}{2}}}dt^\prime.
\end{split}
\ee
Moreover,
\be\label{5.12}
\begin{split}
&{\Big\|}\mathcal{P}{\Big(}\frac{\xi}{1+\epsilon b^\epsilon}{\rm div}(\nabla{\bf d}^\epsilon\odot\nabla{\bf d}^\epsilon)-\xi{\rm div}(\nabla{\bf d}\odot\nabla{\bf d}){\Big)}{\Big\|}_{L^1(\dot{B}_{p,1}^{\frac{N}{p}-\frac{3}{2}})}\\[2mm]
&\quad\lesssim{\Big\|}\xi{\rm div}(\nabla{\bf d}^\epsilon\odot\nabla{\bf d}^\epsilon
-\nabla{\bf d}\odot\nabla{\bf d}){\Big\|}_{L^1(\dot{B}_{p,1}^{\frac{N}{p}-\frac{3}{2}})}\\[2mm]
&\quad\quad+{\Big\|}\frac{\epsilon b^\epsilon}{1+\epsilon b^\epsilon}\xi{\rm div(\nabla{\bf d}^\epsilon\odot\nabla{\bf d}^\epsilon}){\Big\|}_{L^1(\dot{B}_{p,1}^{\frac{N}{p}-\frac{3}{2}})}\\[2mm]
&\quad:=I_1+I_2.
\end{split}
\ee
For the estimate of $I_1$, we have
\be\label{5.13}
\begin{split}
&I_1\lesssim\|\nabla{\bf d}^\epsilon+\nabla{\bf d}\|_{L^2(\dot{B}_{2,1}^{\frac{N}{2}})}\|\nabla\bar{\bf d}^\epsilon\|_{L^2(\dot{B}_{p,1}^{\frac{N}{p}-\frac{1}{2}})}\\[2mm]
&\quad \lesssim{\Big(}\|{\bf d}^\epsilon\|_{L^2(\dot{B}_{2,1}^{\frac{N}{2}+1})}+\|{\bf d}\|_{L^2(\dot{B}_{2,1}^{\frac{N}{2}+1})}{\Big)}\|\bar{\bf d}^\epsilon\|_{L^2(\dot{B}_{p,1}^{\frac{N}{p}+\frac{1}{2}})}.
\end{split}
\ee
For $I_2$, it follows that
\be\label{5.14}
\begin{split}
&I_2\lesssim\epsilon\|b^\epsilon\|_{L^\infty(\dot{B}_{2,1}^{\frac{N}{2}-\frac{1}{2}})}{\Big\|}
{\rm div(\nabla({\bf d}^\epsilon-\hat{\bf d})\odot\nabla({\bf d}^\epsilon-\hat{\bf d}))}{\Big\|}_{L^1(\dot{B}_{p,1}^{\frac{N}{p}-1})}\\[2mm]
&\quad\lesssim\epsilon^{\frac{1}{2}}
\|b^\epsilon\|_{L^\infty(\dot{B}_{2,1}^{\frac{N}{2}-1})}^{\frac{1}{2}}
{\Big(}\epsilon\nu
\|b^\epsilon\|_{L^\infty(\dot{B}_{2,1}^{\frac{N}{2}})}{\Big)}^{\frac{1}{2}}
\|{\bf d}^\epsilon-\hat{\bf d}\|_{L^2(\dot{B}_{p,1}^{\frac{N}{p}+1})}^2\\[2mm]
&\quad\lesssim\epsilon^{\frac{1}{2}}
\|b^\epsilon\|_{L^\infty(\tilde{B}_{\epsilon\nu}^{\frac{N}{2},\infty})}
\|{\bf d}^\epsilon-\hat{\bf d}\|_{L^2(\dot{B}_{2,1}^{\frac{N}{2}+1})}^2
\lesssim (C_0^{\epsilon\nu})^3\epsilon^{\frac{1}{2}}.
\end{split}
\ee
Thus, substituting \eqref{5.10}-\eqref{5.14} into \eqref{5.9}, and choosing $\delta$ in \eqref{5.11} sufficiently small, it is not difficult to obtain
\be\label{5.15}
\begin{split}
&\|{\bf w}^\epsilon\|_{L^\infty(\dot{B}_{p,1}^{\frac{N}{p}-\frac{3}{2}})}+\frac{3}{4}\mu\|{\bf w}^\epsilon\|_{L^1(\dot{B}_{p,1}^{\frac{N}{p}+\frac{1}{2}})}\\[2mm]
&\quad\lesssim\|{\bf w}_0^\epsilon\|_{\dot{B}_{p,1}^{\frac{N}{p}-\frac{3}{2}}}
+C_0^{\epsilon\nu}\epsilon^{\frac{1}{2}}
+\int_0^t{\Big(}\|{\bf u}^\epsilon\|_{\dot{B}_{2,1}^{\frac{N}{2}}}^2+\|{\bf u}\|_{\dot{B}_{2,1}^{\frac{N}{2}}}^2{\Big)}\|{\bf w}^\epsilon\|_{\dot{B}_{p,1}^{\frac{N}{p}-\frac{3}{2}}}dt^\prime\\[2mm]
&\quad\quad+{\Big(}\|{\bf d}^\epsilon\|_{L^2(\dot{B}_{2,1}^{\frac{N}{2}+1})}+\|{\bf d}\|_{L^2(\dot{B}_{2,1}^{\frac{N}{2}+1})}{\Big)}\|\bar{\bf d}^\epsilon\|_{L^2(\dot{B}_{p,1}^{\frac{N}{p}+\frac{1}{2}})}.
\end{split}
\ee
In order to close the estimates of ${\bf w}^\epsilon$, we now aim to bound  the term $\|\bar{\bf d}^\epsilon\|_{L^2(\dot{B}_{p,1}^{\frac{N}{p}+\frac{1}{2}})}$ of \eqref{5.15}. For this,  we take advantage of the estimates of heat equation $\eqref{5.8}_2$ and obtain that
\be\label{5.16}
\begin{split}
&\|\bar{\bf d}^\epsilon\|_{L^\infty(\dot{B}_{p,1}^{\frac{N}{p}-\frac{1}{2}})}+
\theta\|\bar{\bf d}^\epsilon\|_{L^1(\dot{B}_{p,1}^{\frac{N}{p}+\frac{3}{2}})}\\[2mm]
&\quad\lesssim \|\bar{\bf d}_0^\epsilon\|_{\dot{B}_{p,1}^{\frac{N}{p}-\frac{1}{2}}}
+\|{\bf u}^\epsilon\cdot\nabla\bar{{\bf d}}^\epsilon\|_{L^1(\dot{B}_{p,1}^{\frac{N}{p}-\frac{1}{2}})}
+\|{\bf w}^\epsilon\cdot\nabla{\bf d}\|_{L^1(\dot{B}_{p,1}^{\frac{N}{p}-\frac{1}{2}})}+\|\mathcal{Q}{\bf u}^\epsilon\cdot\nabla{\bf d}\|_{L^1(\dot{B}_{p,1}^{\frac{N}{p}-\frac{1}{2}})}\\[2mm]
&\quad\quad+\||\nabla {\bf d}|^2
 \bar{\bf d}^\epsilon\|_{L^1(\dot{B}_{p,1}^{\frac{N}{p}-\frac{1}{2}})}
+\||\nabla{\bf d}^\epsilon+\nabla{\bf d}||\nabla\bar{\bf d}^\epsilon|{\bf d}^\epsilon\|_{L^1(\dot{B}_{p,1}^{\frac{N}{p}-\frac{1}{2}})}.
\end{split}
\ee
Next, we estimate the terms of the right hand of \eqref{5.16} one by one.
\be\label{5.17}
\begin{split}
\|{\bf u}^\epsilon\cdot\nabla\bar{{\bf d}}^\epsilon\|_{L^1(\dot{B}_{p,1}^{\frac{N}{p}-\frac{1}{2}})}
&\lesssim \|{\bf u}^\epsilon\|_{L^2(\dot{B}_{2,1}^{\frac{N}{2}})}
\|\nabla\bar{{\bf d}}^\epsilon\|_{L^2(\dot{B}_{p,1}^{\frac{N}{p}-\frac{1}{2}})}\\[2mm]
&\lesssim \|{\bf u}^\epsilon\|_{L^2(\dot{B}_{2,1}^{\frac{N}{2}})}
\|\bar{{\bf d}}^\epsilon\|_{L^\infty(\dot{B}_{p,1}^{\frac{N}{p}-\frac{1}{2}})}^{\frac{1}{2}}
\|\bar{{\bf d}}^\epsilon\|_{L^1(\dot{B}_{p,1}^{\frac{N}{p}+\frac{3}{2}})}^{\frac{1}{2}}.
\end{split}
\ee
\medskip
\be\label{5.18}
\begin{split}
\|{\bf w}^\epsilon\cdot\nabla{\bf d}\|_{L^1(\dot{B}_{p,1}^{\frac{N}{p}-\frac{1}{2}})}&\lesssim\|{\bf w}^\epsilon\|_{L^2(\dot{B}_{p,1}^{\frac{N}{p}-\frac{1}{2}})}
\|\nabla{\bf d}\|_{L^2(\dot{B}_{2,1}^{\frac{N}{2}})}\\[2mm]
&\lesssim\|{\bf w}^\epsilon\|_{L^\infty(\dot{B}_{p,1}^{\frac{N}{p}-\frac{3}{2}})}^{\frac{1}{2}}
\|{\bf w}^\epsilon\|_{L^1(\dot{B}_{p,1}^{\frac{N}{p}+\frac{1}{2}})}^{\frac{1}{2}}
\|{\bf d}\|_{L^2(\dot{B}_{2,1}^{\frac{N}{2}+1})}.
\end{split}
\ee
\medskip
\be\label{5.19}
\begin{split}
\|\mathcal{Q}{\bf u}^\epsilon\cdot\nabla{\bf d}\|_{L^1(\dot{B}_{p,1}^{\frac{N}{p}-\frac{1}{2}})}
&\lesssim\|\mathcal{Q}{\bf u}^\epsilon\|_{L^2(\dot{B}_{p,1}^{\frac{N}{p}-\frac{1}{2}})}
\|\nabla{\bf d}\|_{L^2(\dot{B}_{2,1}^{\frac{N}{2}})}
\lesssim\|\mathcal{Q}{\bf u}^\epsilon\|_{L^2(\dot{B}_{p,1}^{\frac{N}{p}-\frac{1}{2}})}
\|{\bf d}\|_{L^2(\dot{B}_{2,1}^{\frac{N}{2}+1})}.
\end{split}
\ee
\medskip
\be\label{5.20}
\begin{split}
\||\nabla {\bf d}|^2\bar{\bf d}^\epsilon\|_{L^1(\dot{B}_{p,1}^{\frac{N}{p}-\frac{1}{2}})}
&\lesssim\||\nabla {\bf d}|^2
 \|_{L^1(\dot{B}_{2,1}^{\frac{N}{2}})}
 \|\bar{\bf d}^\epsilon\|_{L^\infty(\dot{B}_{p,1}^{\frac{N}{p}-\frac{1}{2}})}
\lesssim\|{\bf d}
 \|_{L^2(\dot{B}_{2,1}^{\frac{N}{2}+1})}^2
 \|\bar{\bf d}^\epsilon\|_{L^\infty(\dot{B}_{p,1}^{\frac{N}{p}-\frac{1}{2}})}.
\end{split}
\ee
\be\label{5.21}
\begin{split}
&\||\nabla{\bf d}^\epsilon+\nabla{\bf d}||\nabla\bar{\bf d}^\epsilon|{\bf d}^\epsilon\|_{L^1(\dot{B}_{p,1}^{\frac{N}{p}-\frac{1}{2}})}\\[2mm]
&\quad\lesssim\||\nabla{\bf d}^\epsilon||\nabla\bar{\bf d}^\epsilon|{\bf d}^\epsilon\|_{L^1(\dot{B}_{p,1}^{\frac{N}{p}-\frac{1}{2}})}+
\||\nabla{\bf d}||\nabla\bar{\bf d}^\epsilon|{\bf d}^\epsilon\|_{L^1(\dot{B}_{p,1}^{\frac{N}{p}-\frac{1}{2}})}\\[2mm]
&\quad\lesssim\||\nabla{\bf d}^\epsilon|{\bf d}^\epsilon\|_{L^2(\dot{B}_{2,1}^{\frac{N}{2}})}
\|\nabla\bar{\bf d}^\epsilon\|_{L^2(\dot{B}_{p,1}^{\frac{N}{p}-\frac{1}{2}})}+
\||\nabla{\bf d}|{\bf d}^\epsilon\|_{L^2(\dot{B}_{2,1}^{\frac{N}{2}})}
\|\nabla\bar{\bf d}^\epsilon\|_{L^2(\dot{B}_{p,1}^{\frac{N}{p}-\frac{1}{2}})}\\[2mm]
&\quad\lesssim\|\nabla{\bf d}^\epsilon\|_{L^2(\dot{B}_{2,1}^{\frac{N}{2}})}
\|{\bf d}^\epsilon\|_{L^\infty(\dot{B}_{2,1}^{\frac{N}{2}})}
\|\bar{\bf d}^\epsilon\|_{L^2(\dot{B}_{p,1}^{\frac{N}{p}+\frac{1}{2}})}
+\|\nabla{\bf d}\|_{L^2(\dot{B}_{2,1}^{\frac{N}{2}})}
\|{\bf d}^\epsilon\|_{L^\infty(\dot{B}_{2,1}^{\frac{N}{2}})}
\|\bar{\bf d}^\epsilon\|_{L^2(\dot{B}_{p,1}^{\frac{N}{p}+\frac{1}{2}})}\\[2mm]
&\quad\lesssim{\Big(}\|{\bf d}^\epsilon\|_{L^2(\dot{B}_{2,1}^{\frac{N}{2}+1})}+
\|{\bf d}\|_{L^2(\dot{B}_{2,1}^{\frac{N}{2}+1})}{\Big)}
\|{\bf d}^\epsilon\|_{L^\infty(\dot{B}_{2,1}^{\frac{N}{2}})}
\|\bar{{\bf d}}^\epsilon\|_{L^\infty(\dot{B}_{p,1}^{\frac{N}{p}-\frac{1}{2}})}^{\frac{1}{2}}
\|\bar{{\bf d}}^\epsilon\|_{L^1(\dot{B}_{p,1}^{\frac{N}{p}+\frac{3}{2}})}^{\frac{1}{2}}.
\end{split}
\ee
It follows from the above estimates \eqref{5.16}-\eqref{5.21} that
\be\label{5.22}
\begin{split}
&\|\bar{\bf d}^\epsilon\|_{L^\infty(\dot{B}_{p,1}^{\frac{N}{p}-\frac{1}{2}})}+
\theta\|\bar{\bf d}^\epsilon\|_{L^1(\dot{B}_{p,1}^{\frac{N}{p}+\frac{3}{2}})}\\[2mm]
&\quad\lesssim \|\bar{\bf d}_0^\epsilon\|_{\dot{B}_{p,1}^{\frac{N}{p}-\frac{1}{2}}}+\|{\bf u}^\epsilon\|_{L^2(\dot{B}_{2,1}^{\frac{N}{2}})}
\|\bar{{\bf d}}^\epsilon\|_{L^\infty(\dot{B}_{p,1}^{\frac{N}{p}-\frac{1}{2}})}^{\frac{1}{2}}
\|\bar{{\bf d}}^\epsilon\|_{L^1(\dot{B}_{p,1}^{\frac{N}{p}+\frac{3}{2}})}^{\frac{1}{2}}\\[2mm]
&\quad\quad+\|{\bf w}^\epsilon\|_{L^\infty(\dot{B}_{p,1}^{\frac{N}{p}-\frac{3}{2}})}^{\frac{1}{2}}
\|{\bf w}^\epsilon\|_{L^1(\dot{B}_{p,1}^{\frac{N}{p}+\frac{1}{2}})}^{\frac{1}{2}}
\|{\bf d}\|_{L^2(\dot{B}_{2,1}^{\frac{N}{2}+1})}\\[2mm]
&\quad\quad+\|\mathcal{Q}{\bf u}^\epsilon\|_{L^2(\dot{B}_{p,1}^{\frac{N}{p}-\frac{1}{2}})}
\|{\bf d}\|_{L^2(\dot{B}_{2,1}^{\frac{N}{2}+1})}+\|{\bf d}
\|_{L^2(\dot{B}_{2,1}^{\frac{N}{2}+1})}^2
\|\bar{\bf d}^\epsilon\|_{L^\infty(\dot{B}_{p,1}^{\frac{N}{p}-\frac{1}{2}})}\\[2mm]
&\quad\quad+{\Big(}\|{\bf d}^\epsilon\|_{L^2(\dot{B}_{2,1}^{\frac{N}{2}+1})}+
\|{\bf d}\|_{L^2(\dot{B}_{2,1}^{\frac{N}{2}+1})}{\Big)}
\|{\bf d}^\epsilon\|_{L^\infty(\dot{B}_{2,1}^{\frac{N}{2}})}
\|\bar{{\bf d}}^\epsilon\|_{L^\infty(\dot{B}_{p,1}^{\frac{N}{p}-\frac{1}{2}})}^{\frac{1}{2}}
\|\bar{{\bf d}}^\epsilon\|_{L^1(\dot{B}_{p,1}^{\frac{N}{p}+\frac{3}{2}})}^{\frac{1}{2}}.
\end{split}
\ee
Now combining \eqref{5.15} with \eqref{5.22} and using Young inequality, we get
\be\label{5.23}
\begin{split}
&\|{\bf w}^\epsilon\|_{L^\infty(\dot{B}_{p,1}^{\frac{N}{p}-\frac{3}{2}})}+\frac{\mu}{2}\|{\bf w}^\epsilon\|_{L^1(\dot{B}_{p,1}^{\frac{N}{p}+\frac{1}{2}})}
+\|\bar{\bf d}^\epsilon\|_{L^\infty(\dot{B}_{p,1}^{\frac{N}{p}-\frac{1}{2}})}+
\frac{\theta}{2}\|\bar{\bf d}^\epsilon\|_{L^1(\dot{B}_{p,1}^{\frac{N}{p}+\frac{3}{2}})}\\[2mm]
&\quad\lesssim{\Big(}\|{\bf w}_0^\epsilon\|_{\dot{B}_{p,1}^{\frac{N}{p}-\frac{3}{2}}}
+\|\bar{\bf d}_0^\epsilon\|_{\dot{B}_{p,1}^{\frac{N}{p}-\frac{1}{2}}}{\Big)}
+C_0^{\epsilon\nu}\epsilon^{\frac{1}{2}}\\[2mm]
&\quad\quad+\int_0^t{\Big(}\|{\bf u}^\epsilon(t^\prime)\|_{\dot{B}_{2,1}^{\frac{N}{2}}}^2+\|{\bf u}(t^\prime)\|_{\dot{B}_{2,1}^{\frac{N}{2}}}^2+\|{\bf d}(t^\prime)\|_{\dot{B}_{2,1}^{\frac{N}{2}+1}}^2{\Big)}\|{\bf w}^\epsilon(t^\prime)\|_{\dot{B}_{p,1}^{\frac{N}{p}-\frac{3}{2}}}dt^\prime\\[2mm]
&\quad\quad+\int_0^t{\Big(}\|{\bf d}^\epsilon(t^\prime)\|_{\dot{B}_{2,1}^{\frac{N}{2}+1}}^2+\|{\bf d}(t^\prime)\|_{\dot{B}_{2,1}^{\frac{N}{2}+1}}^2+\|{\bf u}^\epsilon(t^\prime)\|_{\dot{B}_{2,1}^{\frac{N}{2}}}^2{\Big)}\|\bar{\bf d}^\epsilon(t^\prime)\|_{\dot{B}_{p,1}^{\frac{N}{p}-\frac{1}{2}}}dt^\prime.
\end{split}
\ee
Thus, Gronwall's inequality guarantees that
\be\label{5.24}
\begin{split}
&\|{\bf w}^\epsilon\|_{L^\infty(\dot{B}_{p,1}^{\frac{N}{p}-\frac{3}{2}})}+\frac{\mu}{2}\|{\bf w}^\epsilon\|_{L^1(\dot{B}_{p,1}^{\frac{N}{p}+\frac{1}{2}})}
+\|\bar{\bf d}^\epsilon\|_{L^\infty(\dot{B}_{p,1}^{\frac{N}{p}-\frac{1}{2}})}+
\frac{\theta}{2}\|\bar{\bf d}^\epsilon\|_{L^1(\dot{B}_{p,1}^{\frac{N}{p}+\frac{3}{2}})}\\[2mm]
&\quad\lesssim{\Big(}\|{\bf w}_0^\epsilon\|_{\dot{B}_{p,1}^{\frac{N}{p}-\frac{3}{2}}}
+\|\bar{\bf d}_0^\epsilon\|_{\dot{B}_{p,1}^{\frac{N}{p}-\frac{1}{2}}}{\Big)}
+C_0^{\epsilon\nu}\epsilon^{\frac{1}{2}}.
\end{split}
\ee

{\bf Case 2}: $N=3$. Applying Proposition \ref{pr5.1} to \eqref{5.5} with $\bar{p}=2,~\bar{r}=\infty,~s=\frac{1}{2}$ and $r=\frac{2p}{p-2}$, similar to \eqref{5.7}, we obtain
\be\label{5.25}
\|(b^\epsilon,\mathcal{Q}{\bf u}^\epsilon)\|_{{L}^{\frac{2p}{p-2}}(\dot{B}_{p,1}^{\frac{2}{p}-\frac{1}{2}})}
\lesssim \|(b^\epsilon,\mathcal{Q}{\bf u}^\epsilon)\|_{\tilde{L}^{\frac{2p}{p-2}}(\dot{B}_{p,1}^{\frac{2}{p}-\frac{1}{2}})}
\lesssim C_0^{\epsilon\nu}\epsilon^{\frac{1}{2}-\frac{1}{p}},~~~~~p\geq 2.
\ee
Use the following interpolation for $2\leq q<+\infty$,
$$
L^2{\Big(}\mathbb{R}^+; \dot{B}_{\frac{q+2}{2}, 1}^{{(14-q)}/{(2q+4)}}{\Big)}
={\Big[}L^1(\mathbb{R}^+; \dot{B}_{2, 1}^{\frac{5}{2}});~{{L}^{\frac{2q}{q-2}}(\mathbb{R}^+; \dot{B}_{q,1}^{\frac{2}{q}-\frac{1}{2}})}{\Big]}_{\frac{2}{q+2}}.
$$
Make the change of parameter $p=\frac{q+2}{2}$. Due to \eqref{5.25}, there holds
\be\label{5.26}
\|\mathcal{Q}{\bf u}^\epsilon\|_{L^2(\dot{B}_{p,1}^{\frac{4}{p}-\frac{1}{2}})}\lesssim
C_0^{\epsilon\nu}\epsilon^{\frac{1}{2}-\frac{1}{p}},~~~~2\leq p<+\infty.
\ee

In the following, we want to  prove that
\be\label{5.27}
\begin{split}
&\|{\bf w}^\epsilon\|_{L^\infty(\dot{B}_{p,1}^{\frac{4}{p}-\frac{3}{2}})}+\|{\bf w}^\epsilon\|_{L^1(\dot{B}_{p,1}^{\frac{4}{p}+\frac{1}{2}})}
+\|\bar{\bf d}^\epsilon\|_{L^\infty(\dot{B}_{p,1}^{\frac{4}{p}-\frac{1}{2}})}+
\|\bar{\bf d}^\epsilon\|_{L^1(\dot{B}_{p,1}^{\frac{4}{p}+\frac{3}{2}})}\\[2mm]
&\quad\lesssim {\Big(}\|{\bf w}_0^\epsilon\|_{\dot{B}_{p,1}^{\frac{4}{p}-\frac{3}{2}}}
+\|\bar{\bf d}_0^\epsilon\|_{\dot{B}_{p,1}^{\frac{4}{p}-\frac{1}{2}}}{\Big)}
+C_0^{\epsilon\nu}\epsilon^{\frac{1}{2}-\frac{1}{p}}.
\end{split}
\ee
To this end, similar to the estimates of \eqref{5.10}-\eqref{5.14}, we have
\be\label{5.28}
\begin{split}
&{\Big\|}\mathcal{P}{\Big(}\frac{\epsilon b^\epsilon}{1+\epsilon b^\epsilon}\mathcal{A}{\bf u}^\epsilon{\Big)}{\Big\|}_{L^1(\dot{B}_{p,1}^{\frac{4}{p}-\frac{3}{2}})}
\lesssim
\epsilon\|b^\epsilon\|_{L^\infty(\dot{B}_{2,1}^{1+\frac{1}{p}})}\|\mathcal{A}{\bf u}^\epsilon\|_{L^1(\dot{B}_{p,1}^{\frac{3}{p}-1})}\\[2mm]
&\quad\lesssim \epsilon(\epsilon\nu)^{-\frac{1}{2}-\frac{1}{p}}\| b^\epsilon\|_{L^\infty(\dot{B}_{2,1}^{\frac{1}{2}})}^{\frac{1}{2}-\frac{1}{p}}
(\epsilon\nu\| b^\epsilon\|_{L^\infty(\dot{B}_{2,1}^{\frac{3}{2}})})^{\frac{1}{2}+\frac{1}{p}}\|{\bf u}^\epsilon\|_{L^1(\dot{B}_{2,1}^{\frac{5}{2}})}\\[2mm]
&\quad\lesssim\epsilon^{\frac{1}{2}-\frac{1}{p}}
\|b^\epsilon\|_{L^\infty(\tilde{B}_{\epsilon\nu}^{\frac{3}{2},\infty})}\|{\bf u}^\epsilon\|_{L^1(\dot{B}_{2,1}^{\frac{5}{2}})}
\lesssim (C_0^{\epsilon\nu})^2\epsilon^{\frac{1}{2}-\frac{1}{p}},
\end{split}
\ee
\be\label{5.29}
\begin{split}
&\|\mathcal{P}({\bf u}^\epsilon\cdot\nabla{\bf u}^\epsilon-{\bf u}\cdot\nabla{\bf u})\|_{L^1(\dot{B}_{p,1}^{\frac{4}{p}-\frac{3}{2}})}\\[2mm]
&\quad\leq C\|{\bf u}^\epsilon\|_{L^2(\dot{B}_{2,1}^{\frac{N}{2}})}
\|\nabla\mathcal{Q}{\bf u}^\epsilon\|_{L^2(\dot{B}_{p,1}^{\frac{4}{p}-\frac{3}{2}})}
+C\|\nabla{\bf u}\|_{L^2(\dot{B}_{2,1}^{\frac{N}{2}-1})}
\|\mathcal{Q}{\bf u}^\epsilon\|_{L^2(\dot{B}_{p,1}^{\frac{4}{p}-\frac{1}{2}})}\\[2mm]
&\quad\quad\D +C\int_0^t{\Big(}\|{\bf u}^\epsilon\|_{\dot{B}_{2,1}^{\frac{3}{2}}}\|\nabla{\bf w}^\epsilon\|_{\dot{B}_{p,1}^{\frac{4}{p}-\frac{3}{2}}}
+\|\nabla{\bf u}\|_{\dot{B}_{2,1}^{\frac{1}{2}}}\|{\bf w}^\epsilon\|_{\dot{B}_{p,1}^{\frac{4}{p}-\frac{1}{2}}}{\Big)}dt^\prime\\[2mm]
&\quad\leq C(C_0^{\epsilon\nu})^2\epsilon^{\frac{1}{2}-\frac{1}{p}}
+\delta\|{\bf w}^\epsilon\|_{L^1(\dot{B}_{p,1}^{\frac{4}{p}+\frac{1}{2}})}
+C_{\delta}\int_0^t{\Big(}\|{\bf u}^\epsilon\|_{\dot{B}_{2,1}^{\frac{3}{2}}}^2+\|{\bf u}\|_{\dot{B}_{2,1}^{\frac{3}{2}}}^2{\Big)}\|{\bf w}^\epsilon\|_{\dot{B}_{p,1}^{\frac{4}{p}-\frac{3}{2}}}dt^\prime.
\end{split}
\ee
Moreover,
\be\label{5.30}
\begin{split}
&{\Big\|}\mathcal{P}{\Big(}\frac{\xi}{1+\epsilon b^\epsilon}{\rm div}(\nabla{\bf d}^\epsilon\odot\nabla{\bf d}^\epsilon)-\xi{\rm div}(\nabla{\bf d}\odot\nabla{\bf d}){\Big)}{\Big\|}_{L^1(\dot{B}_{p,1}^{\frac{4}{p}-\frac{3}{2}})}\\[2mm]
&\quad\lesssim{\Big\|}\xi{\rm div}(\nabla{\bf d}^\epsilon\odot\nabla{\bf d}^\epsilon
-\nabla{\bf d}\odot\nabla{\bf d}){\Big\|}_{L^1(\dot{B}_{p,1}^{\frac{4}{p}-\frac{3}{2}})}\\[2mm]
&\quad\quad+{\Big\|}\frac{\epsilon b^\epsilon}{1+\epsilon b^\epsilon}\xi{\rm div(\nabla{\bf d}^\epsilon\odot\nabla{\bf d}^\epsilon)}{\Big\|}_{L^1(\dot{B}_{p,1}^{\frac{4}{p}-\frac{3}{2}})}:=H_1+H_2.
\end{split}
\ee
For the estimate of $H_1$, we have
\be\label{5.31}
\begin{split}
&H_1\lesssim\|\nabla{\bf d}^\epsilon+\nabla{\bf d}\|_{L^2(\dot{B}_{2,1}^{\frac{3}{2}})}\|\nabla\bar{\bf d}^\epsilon\|_{L^2(\dot{B}_{p,1}^{\frac{4}{p}-\frac{1}{2}})}\\[2mm]
&\quad \lesssim{\Big(}\|{\bf d}^\epsilon\|_{L^2(\dot{B}_{2,1}^{\frac{5}{2}})}+\|{\bf d}\|_{L^2(\dot{B}_{2,1}^{\frac{5}{2}})}{\Big)}\|\bar{\bf d}^\epsilon\|_{L^2(\dot{B}_{p,1}^{\frac{4}{p}+\frac{1}{2}})}.
\end{split}
\ee
For $H_2$, it follows that
\be\label{5.32}
\begin{split}
&H_2\lesssim\epsilon\|b^\epsilon\|_{L^\infty(\dot{B}_{2,1}^{1+\frac{1}{p}})}{\Big\|}
{\rm div(\nabla({\bf d}^\epsilon-\hat{\bf d})\odot\nabla({\bf d}^\epsilon-\hat{\bf d}))}{\Big\|}_{L^1(\dot{B}_{p,1}^{\frac{3}{p}-1})}\\[2mm]
&\quad\lesssim \epsilon(\epsilon\nu)^{-\frac{1}{2}-\frac{1}{p}}\| b^\epsilon\|_{L^\infty(\dot{B}_{2,1}^{\frac{1}{2}})}^{\frac{1}{2}-\frac{1}{p}}
(\epsilon\nu\| b^\epsilon\|_{L^\infty(\dot{B}_{2,1}^{\frac{3}{2}})})^{\frac{1}{2}+\frac{1}{p}}
\|{\bf d}^\epsilon-\hat{\bf d}\|_{L^2(\dot{B}_{p,1}^{\frac{3}{p}+1})}^2\\[2mm]
&\quad\lesssim\epsilon^{\frac{1}{2}-\frac{1}{p}}
\|b^\epsilon\|_{L^\infty(\tilde{B}_{\epsilon\nu}^{\frac{3}{2},\infty})}
\|{\bf d}^\epsilon-\hat{\bf d}\|_{L^2(\dot{B}_{2,1}^{\frac{5}{2}})}^2
\quad\lesssim (C_0^{\epsilon\nu})^3\epsilon^{\frac{1}{2}-\frac{1}{p}}.
\end{split}
\ee
Therefore, in view of the estimates of heat equation, similar to \eqref{5.15}, we have
\be\label{5.33}
\begin{split}
&\|{\bf w}^\epsilon\|_{L^\infty(\dot{B}_{p,1}^{\frac{4}{p}-\frac{3}{2}})}+\frac{3}{4}\mu\|{\bf w}^\epsilon\|_{L^1(\dot{B}_{p,1}^{\frac{4}{p}+\frac{1}{2}})}\\[2mm]
&\quad\lesssim\|{\bf w}_0^\epsilon\|_{\dot{B}_{p,1}^{\frac{4}{p}-\frac{3}{2}}}
+C_0^{\epsilon\nu}\epsilon^{\frac{1}{2}-\frac{1}{p}}
+\int_0^t{\Big(}\|{\bf u}^\epsilon\|_{\dot{B}_{2,1}^{\frac{3}{2}}}^2+\|{\bf u}\|_{\dot{B}_{2,1}^{\frac{3}{2}}}^2{\Big)}\|{\bf w}^\epsilon\|_{\dot{B}_{p,1}^{\frac{4}{p}-\frac{3}{2}}}dt^\prime\\[2mm]
&\quad\quad+{\Big(}\|{\bf d}^\epsilon\|_{L^2(\dot{B}_{2,1}^{\frac{5}{2}})}+\|{\bf d}\|_{L^2(\dot{B}_{2,1}^{\frac{5}{2}})}{\Big)}\|\bar{\bf d}^\epsilon\|_{L^2(\dot{B}_{p,1}^{\frac{4}{p}+\frac{1}{2}})}.
\end{split}
\ee
In order to close the estimates of ${\bf w}^\epsilon$, we must bound  the term $\|\bar{\bf d}^\epsilon\|_{L^2(\dot{B}_{p,1}^{\frac{4}{p}+\frac{1}{2}})}$ of \eqref{5.33}. For this purpose,  similar to \eqref{5.16}, it follows that
\be\label{5.34}
\begin{split}
&\|\bar{\bf d}^\epsilon\|_{L^\infty(\dot{B}_{p,1}^{\frac{4}{p}-\frac{1}{2}})}+
\theta\|\bar{\bf d}^\epsilon\|_{L^1(\dot{B}_{p,1}^{\frac{4}{p}+\frac{3}{2}})}\\[2mm]
&\quad\lesssim \|\bar{\bf d}_0^\epsilon\|_{\dot{B}_{p,1}^{\frac{4}{p}-\frac{1}{2}}}
+\|{\bf u}^\epsilon\cdot\nabla\bar{{\bf d}}^\epsilon\|_{L^1(\dot{B}_{p,1}^{\frac{4}{p}-\frac{1}{2}})}
+\|{\bf w}^\epsilon\cdot\nabla{\bf d}\|_{L^1(\dot{B}_{p,1}^{\frac{4}{p}-\frac{1}{2}})}+\|\mathcal{Q}{\bf u}^\epsilon\cdot\nabla{\bf d}\|_{L^1(\dot{B}_{p,1}^{\frac{4}{p}-\frac{1}{2}})}\\[2mm]
&\quad\quad+\||\nabla {\bf d}|^2
 \bar{\bf d}^\epsilon\|_{L^1(\dot{B}_{p,1}^{\frac{4}{p}-\frac{1}{2}})}
+\||\nabla{\bf d}^\epsilon+\nabla{\bf d}||\nabla\bar{\bf d}^\epsilon|{\bf d}^\epsilon\|_{L^1(\dot{B}_{p,1}^{\frac{4}{p}-\frac{1}{2}})}.
\end{split}
\ee

Next, we estimate the terms of the right hand of \eqref{5.34} as follows.
\be\label{5.35}
\begin{split}
\|{\bf u}^\epsilon\cdot\nabla\bar{{\bf d}}^\epsilon\|_{L^1(\dot{B}_{p,1}^{\frac{4}{p}-\frac{1}{2}})}
&\lesssim \|{\bf u}^\epsilon\|_{L^2(\dot{B}_{2,1}^{\frac{3}{2}})}
\|\nabla\bar{{\bf d}}^\epsilon\|_{L^2(\dot{B}_{p,1}^{\frac{4}{p}-\frac{1}{2}})}\\[2mm]
&\lesssim \|{\bf u}^\epsilon\|_{L^2(\dot{B}_{2,1}^{\frac{3}{2}})}
\|\bar{{\bf d}}^\epsilon\|_{L^\infty(\dot{B}_{p,1}^{\frac{4}{p}-\frac{1}{2}})}^{\frac{1}{2}}
\|\bar{{\bf d}}^\epsilon\|_{L^1(\dot{B}_{p,1}^{\frac{4}{p}+\frac{3}{2}})}^{\frac{1}{2}},
\end{split}
\ee
\be\label{5.36}
\begin{split}
\|{\bf w}^\epsilon\cdot\nabla{\bf d}\|_{L^1(\dot{B}_{p,1}^{\frac{4}{p}-\frac{1}{2}})}&\lesssim\|{\bf w}^\epsilon\|_{L^2(\dot{B}_{p,1}^{\frac{4}{p}-\frac{1}{2}})}
\|\nabla{\bf d}\|_{L^2(\dot{B}_{2,1}^{\frac{3}{2}})}\\[2mm]
&\lesssim\|{\bf w}^\epsilon\|_{L^\infty(\dot{B}_{p,1}^{\frac{4}{p}-\frac{3}{2}})}^{\frac{1}{2}}
\|{\bf w}^\epsilon\|_{L^1(\dot{B}_{p,1}^{\frac{4}{p}+\frac{1}{2}})}^{\frac{1}{2}}
\|{\bf d}\|_{L^2(\dot{B}_{2,1}^{\frac{5}{2}})},
\end{split}
\ee
\be\label{5.37}
\begin{split}
\|\mathcal{Q}{\bf u}^\epsilon\cdot\nabla{\bf d}\|_{L^1(\dot{B}_{p,1}^{\frac{4}{p}-\frac{1}{2}})}
&\lesssim\|\mathcal{Q}{\bf u}^\epsilon\|_{L^2(\dot{B}_{p,1}^{\frac{4}{p}-\frac{1}{2}})}
\|\nabla{\bf d}\|_{L^2(\dot{B}_{2,1}^{\frac{3}{2}})}
\lesssim\|\mathcal{Q}{\bf u}^\epsilon\|_{L^2(\dot{B}_{p,1}^{\frac{4}{p}-\frac{1}{2}})}
\|{\bf d}\|_{L^2(\dot{B}_{2,1}^{\frac{5}{2}})},
\end{split}
\ee
\be\label{5.38}
\begin{split}
\||\nabla {\bf d}|^2\bar{\bf d}^\epsilon\|_{L^1(\dot{B}_{p,1}^{\frac{4}{p}-\frac{1}{2}})}
&\lesssim\||\nabla {\bf d}|^2
 \|_{L^1(\dot{B}_{2,1}^{\frac{3}{2}})}
 \|\bar{\bf d}^\epsilon\|_{L^\infty(\dot{B}_{p,1}^{\frac{4}{p}-\frac{1}{2}})}
\lesssim\|{\bf d}
 \|_{L^2(\dot{B}_{2,1}^{\frac{5}{2}})}^2
 \|\bar{\bf d}^\epsilon\|_{L^\infty(\dot{B}_{p,1}^{\frac{4}{p}-\frac{1}{2}})},
\end{split}
\ee
\be\label{5.39}
\begin{split}
&\||\nabla{\bf d}^\epsilon+\nabla{\bf d}||\nabla\bar{\bf d}^\epsilon|{\bf d}^\epsilon\|_{L^1(\dot{B}_{p,1}^{\frac{4}{p}-\frac{1}{2}})}\\[2mm]
&\quad\lesssim\||\nabla{\bf d}^\epsilon||\nabla\bar{\bf d}^\epsilon|{\bf d}^\epsilon\|_{L^1(\dot{B}_{p,1}^{\frac{4}{p}-\frac{1}{2}})}+
\||\nabla{\bf d}||\nabla\bar{\bf d}^\epsilon|{\bf d}^\epsilon\|_{L^1(\dot{B}_{p,1}^{\frac{4}{p}-\frac{1}{2}})}\\[2mm]
&\quad\lesssim\||\nabla{\bf d}^\epsilon|{\bf d}^\epsilon\|_{L^2(\dot{B}_{2,1}^{\frac{3}{2}})}
\|\nabla\bar{\bf d}^\epsilon\|_{L^2(\dot{B}_{p,1}^{\frac{4}{p}-\frac{1}{2}})}+
\||\nabla{\bf d}|{\bf d}^\epsilon\|_{L^2(\dot{B}_{2,1}^{\frac{3}{2}})}
\|\nabla\bar{\bf d}^\epsilon\|_{L^2(\dot{B}_{p,1}^{\frac{4}{p}-\frac{1}{2}})}\\[2mm]
&\quad\lesssim\|\nabla{\bf d}^\epsilon\|_{L^2(\dot{B}_{2,1}^{\frac{3}{2}})}
\|{\bf d}^\epsilon\|_{L^\infty(\dot{B}_{2,1}^{\frac{3}{2}})}
\|\bar{\bf d}^\epsilon\|_{L^2(\dot{B}_{p,1}^{\frac{4}{p}+\frac{1}{2}})}
+\|\nabla{\bf d}\|_{L^2(\dot{B}_{2,1}^{\frac{3}{2}})}
\|{\bf d}^\epsilon\|_{L^\infty(\dot{B}_{2,1}^{\frac{3}{2}})}
\|\bar{\bf d}^\epsilon\|_{L^2(\dot{B}_{p,1}^{\frac{4}{p}+\frac{1}{2}})}\\[2mm]
&\quad\lesssim{\Big(}\|{\bf d}^\epsilon\|_{L^2(\dot{B}_{2,1}^{\frac{5}{2}})}+
\|{\bf d}\|_{L^2(\dot{B}_{2,1}^{\frac{5}{2}})}{\Big)}
\|{\bf d}^\epsilon\|_{L^\infty(\dot{B}_{2,1}^{\frac{3}{2}})}
\|\bar{{\bf d}}^\epsilon\|_{L^\infty(\dot{B}_{p,1}^{\frac{4}{p}-\frac{1}{2}})}^{\frac{1}{2}}
\|\bar{{\bf d}}^\epsilon\|_{L^1(\dot{B}_{p,1}^{\frac{4}{p}+\frac{3}{2}})}^{\frac{1}{2}}.
\end{split}
\ee
It follows from the above estimates \eqref{5.34}-\eqref{5.39} that
\be\label{5.40}
\begin{split}
&\|\bar{\bf d}^\epsilon\|_{L^\infty(\dot{B}_{p,1}^{\frac{4}{p}-\frac{1}{2}})}+
\theta\|\bar{\bf d}^\epsilon\|_{L^1(\dot{B}_{p,1}^{\frac{4}{p}+\frac{3}{2}})}\\[2mm]
&\quad\lesssim \|\bar{\bf d}_0^\epsilon\|_{\dot{B}_{p,1}^{\frac{4}{p}-\frac{1}{2}}}+\|{\bf u}^\epsilon\|_{L^2(\dot{B}_{2,1}^{\frac{3}{2}})}
\|\bar{{\bf d}}^\epsilon\|_{L^\infty(\dot{B}_{p,1}^{\frac{4}{p}-\frac{1}{2}})}^{\frac{1}{2}}
\|\bar{{\bf d}}^\epsilon\|_{L^1(\dot{B}_{p,1}^{\frac{4}{p}+\frac{3}{2}})}^{\frac{1}{2}}\\[2mm]
&\quad\quad+\|{\bf w}^\epsilon\|_{L^\infty(\dot{B}_{p,1}^{\frac{4}{p}-\frac{3}{2}})}^{\frac{1}{2}}
\|{\bf w}^\epsilon\|_{L^1(\dot{B}_{p,1}^{\frac{4}{p}+\frac{1}{2}})}^{\frac{1}{2}}
\|{\bf d}\|_{L^2(\dot{B}_{2,1}^{\frac{5}{2}})}\\[2mm]
&\quad\quad+\|\mathcal{Q}{\bf u}^\epsilon\|_{L^2(\dot{B}_{p,1}^{\frac{4}{p}-\frac{1}{2}})}
\|{\bf d}\|_{L^2(\dot{B}_{2,1}^{\frac{5}{2}})}+\|{\bf d}
\|_{L^2(\dot{B}_{2,1}^{\frac{5}{2}})}^2
\|\bar{\bf d}^\epsilon\|_{L^\infty(\dot{B}_{p,1}^{\frac{4}{p}-\frac{1}{2}})}\\[2mm]
&\quad\quad+{\Big(}\|{\bf d}^\epsilon\|_{L^2(\dot{B}_{2,1}^{\frac{5}{2}})}+
\|{\bf d}\|_{L^2(\dot{B}_{2,1}^{\frac{5}{2}})}{\Big)}
\|{\bf d}^\epsilon\|_{L^\infty(\dot{B}_{2,1}^{\frac{3}{2}})}
\|\bar{{\bf d}}^\epsilon\|_{L^\infty(\dot{B}_{p,1}^{\frac{4}{p}-\frac{1}{2}})}^{\frac{1}{2}}
\|\bar{{\bf d}}^\epsilon\|_{L^1(\dot{B}_{p,1}^{\frac{4}{p}+\frac{3}{2}})}^{\frac{1}{2}}.
\end{split}
\ee
Now combining \eqref{5.33} with \eqref{5.40} and using Young inequality, we get
\be\label{5.41}
\begin{split}
&\|{\bf w}^\epsilon\|_{L^\infty(\dot{B}_{p,1}^{\frac{4}{p}-\frac{3}{2}})}+\frac{\mu}{2}\|{\bf w}^\epsilon\|_{L^1(\dot{B}_{p,1}^{\frac{4}{p}+\frac{1}{2}})}
+\|\bar{\bf d}^\epsilon\|_{L^\infty(\dot{B}_{p,1}^{\frac{4}{p}-\frac{1}{2}})}+
\frac{\theta}{2}\|\bar{\bf d}^\epsilon\|_{L^1(\dot{B}_{p,1}^{\frac{4}{p}+\frac{3}{2}})}\\[2mm]
&\quad\lesssim{\Big(}\|{\bf w}_0^\epsilon\|_{\dot{B}_{p,1}^{\frac{4}{p}-\frac{3}{2}}}
+\|\bar{\bf d}_0^\epsilon\|_{\dot{B}_{p,1}^{\frac{4}{p}-\frac{1}{2}}}{\Big)}
+C_0^{\epsilon\nu}\epsilon^{\frac{1}{2}-\frac{1}{p}}\\[2mm]
&\quad\quad+\int_0^t{\Big(}\|{\bf u}^\epsilon(t^\prime)\|_{\dot{B}_{2,1}^{\frac{3}{2}}}^2+\|{\bf u}(t^\prime)\|_{\dot{B}_{2,1}^{\frac{3}{2}}}^2+\|{\bf d}(t^\prime)\|_{\dot{B}_{2,1}^{\frac{5}{2}}}^2{\Big)}\|{\bf w}^\epsilon(t^\prime)\|_{\dot{B}_{p,1}^{\frac{4}{p}-\frac{3}{2}}}dt^\prime\\[2mm]
&\quad\quad+\int_0^t{\Big(}\|{\bf d}^\epsilon(t^\prime)\|_{\dot{B}_{2,1}^{\frac{5}{2}}}^2+\|{\bf d}(t^\prime)\|_{\dot{B}_{2,1}^{\frac{5}{2}}}^2+\|{\bf u}^\epsilon(t^\prime)\|_{\dot{B}_{2,1}^{\frac{3}{2}}}^2{\Big)}\|\bar{{\bf d}}^\epsilon(t^\prime)\|_{\dot{B}_{p,1}^{\frac{4}{p}-\frac{1}{2}}}dt^\prime.
\end{split}
\ee
Gronwall's inequality then yields that
\be\label{5.42}
\begin{split}
&\|{\bf w}^\epsilon\|_{L^\infty(\dot{B}_{p,1}^{\frac{4}{p}-\frac{3}{2}})}+\frac{\mu}{2}\|{\bf w}^\epsilon\|_{L^1(\dot{B}_{p,1}^{\frac{4}{p}+\frac{1}{2}})}
+\|\bar{\bf d}^\epsilon\|_{L^\infty(\dot{B}_{p,1}^{\frac{4}{p}-\frac{1}{2}})}+
\frac{\theta}{2}\|\bar{\bf d}^\epsilon\|_{L^1(\dot{B}_{p,1}^{\frac{4}{p}+\frac{3}{2}})}\\[2mm]
&\quad\lesssim{\Big(}\|{\bf w}_0^\epsilon\|_{\dot{B}_{p,1}^{\frac{4}{p}-\frac{3}{2}}}
+\|\bar{\bf d}_0^\epsilon\|_{\dot{B}_{p,1}^{\frac{4}{p}-\frac{1}{2}}}{\Big)}
+C_0^{\epsilon\nu}\epsilon^{\frac{1}{2}-\frac{1}{p}}.
\end{split}
\ee

{\bf Case 3}: $N=2$.  Applying Proposition \ref{pr5.1} to \eqref{5.5} with $\bar{p}=2,~\bar{r}=\infty,~s=0$ and $r=\frac{4p}{p-2}$ yields
\be\label{5.43}
\|(b^\epsilon,\mathcal{Q}{\bf u}^\epsilon)\|_{{L}^{\frac{4p}{p-2}}(\dot{B}_{p,1}^{\frac{3}{2p}-\frac{3}{4}})}
\lesssim \|(b^\epsilon,\mathcal{Q}{\bf u}^\epsilon)\|_{\tilde{L}^{\frac{4p}{p-2}}(\dot{B}_{p,1}^{\frac{3}{2p}-\frac{3}{4}})}
\lesssim C_0^{\epsilon\nu}\epsilon^{\frac{1}{4}-\frac{1}{2p}},~~~~~p\geq 2.
\ee
Use the following interpolation for $2\leq q<+\infty$,
$$
L^2{\Big(}\mathbb{R}^+; \dot{B}_{\frac{6q+4}{q+6}, 1}^{(14+q)/(6q+4)}{\Big)}
={\Big[}L^1(\mathbb{R}^+; \dot{B}_{2, 1}^{2});~{{L}^{\frac{4q}{q-2}}(\mathbb{R}^+; \dot{B}_{q,1}^{\frac{3}{2q}-\frac{3}{4}})}{\Big]}_{\frac{q+2}{3q+2}}.
$$
Make the change of parameter $p=\frac{6q+4}{q+6}$. Thanks  to estimate \eqref{5.43}, we conclude that
\be\label{5.44}
\|\mathcal{Q}{\bf u}^\epsilon\|_{L^2(\dot{B}_{p,1}^{\frac{5}{2p}-\frac{1}{4}})}\lesssim
C_0^{\epsilon\nu}\epsilon^{\frac{1}{4}-\frac{1}{2p}},~~~~2\leq p\leq 6.
\ee
Next, we are going  to prove the convergence of  ${\bf w}^\epsilon$ and $\bar{\bf d}^\epsilon$. For this purpose, it follows from Corollary \ref{co2.1} that
\be\label{5.46}
\begin{split}
&{\Big\|}\mathcal{P}{\Big(}\frac{\epsilon b^\epsilon}{1+\epsilon b^\epsilon}\mathcal{A}{\bf u}^\epsilon{\Big)}{\Big\|}_{L^1(\dot{B}_{p,1}^{\frac{5}{2p}-\frac{5}{4}})}
\lesssim
\epsilon\|b^\epsilon\|_{L^\infty(\dot{B}_{2,1}^{\frac{3}{4}+\frac{1}{2p}})}\|\mathcal{A}{\bf u}^\epsilon\|_{L^1(\dot{B}_{p,1}^{\frac{2}{p}-1})}\\[2mm]
&\quad\lesssim \epsilon(\epsilon\nu)^{-\frac{3}{4}-\frac{1}{2p}}\| b^\epsilon\|_{L^\infty(\dot{B}_{2,1}^{0})}^{\frac{1}{4}-\frac{1}{2p}}
(\epsilon\nu\| b^\epsilon\|_{L^\infty(\dot{B}_{2,1}^{1})})^{\frac{3}{4}+\frac{1}{2p}}\|{\bf u}^\epsilon\|_{L^1(\dot{B}_{2,1}^{2})}\\[2mm]
&\quad\lesssim\epsilon^{\frac{1}{4}-\frac{1}{2p}}
\|b^\epsilon\|_{L^\infty(\tilde{B}_{\epsilon\nu}^{1,\infty})}\|{\bf u}^\epsilon\|_{L^1(\dot{B}_{2,1}^{2})}
\lesssim (C_0^{\epsilon\nu})^2\epsilon^{\frac{1}{4}-\frac{1}{2p}},
\end{split}
\ee
\be\label{5.47}
\begin{split}
&\|\mathcal{P}({\bf u}^\epsilon\cdot\nabla{\bf u}^\epsilon-{\bf u}\cdot\nabla{\bf u})\|_{L^1(\dot{B}_{p,1}^{\frac{5}{2p}-\frac{5}{4}})}\\[2mm]
&\quad\leq C\|{\bf u}^\epsilon\|_{L^2(\dot{B}_{2,1}^{1})}
\|\nabla\mathcal{Q}{\bf u}^\epsilon\|_{L^2(\dot{B}_{p,1}^{\frac{5}{2p}-\frac{5}{4}})}
+C\|\nabla{\bf u}\|_{L^2(\dot{B}_{2,1}^{0})}
\|\mathcal{Q}{\bf u}^\epsilon\|_{L^2(\dot{B}_{p,1}^{\frac{5}{2p}-\frac{5}{4}})}\\[2mm]
&\quad\quad\D +C\int_0^t{\Big(}\|{\bf u}^\epsilon\|_{\dot{B}_{2,1}^{1}}\|\nabla{\bf w}^\epsilon\|_{\dot{B}_{p,1}^{\frac{5}{2p}-\frac{5}{4}}}
+\|\nabla{\bf u}\|_{\dot{B}_{2,1}^{0}}\|{\bf w}^\epsilon\|_{\dot{B}_{p,1}^{\frac{5}{2p}-\frac{1}{4}}}{\Big)}dt^\prime\\[2mm]
&\quad\leq C(C_0^{\epsilon\nu})^2\epsilon^{\frac{1}{4}-\frac{1}{2p}}
+\delta\|{\bf w}^\epsilon\|_{L^1(\dot{B}_{p,1}^{\frac{5}{2p}+\frac{3}{4}})}
+C_{\delta}\int_0^t{\Big(}\|{\bf u}^\epsilon\|_{\dot{B}_{2,1}^{1}}^2+\|{\bf u}\|_{\dot{B}_{2,1}^{1}}^2{\Big)}\|{\bf w}^\epsilon\|_{\dot{B}_{p,1}^{\frac{5}{2p}-\frac{5}{4}}}dt^\prime,
\end{split}
\ee

\be\label{5.48}
\begin{split}
&{\Big\|}\mathcal{P}{\Big(}\frac{\xi}{1+\epsilon b^\epsilon}{\rm div}(\nabla{\bf d}^\epsilon\odot\nabla{\bf d}^\epsilon)-\xi{\rm div}(\nabla{\bf d}\odot\nabla{\bf d}){\Big)}{\Big\|}_{L^1(\dot{B}_{p,1}^{\frac{5}{2p}-\frac{5}{4}})}\\[2mm]
&\quad\lesssim{\Big\|}\xi{\rm div}(\nabla{\bf d}^\epsilon\odot\nabla{\bf d}^\epsilon
-\nabla{\bf d}\odot\nabla{\bf d}){\Big\|}_{L^1(\dot{B}_{p,1}^{\frac{5}{2p}-\frac{5}{4}})}\\[2mm]
&\quad\quad+{\Big\|}\frac{\epsilon b^\epsilon}{1+\epsilon b^\epsilon}\xi{\rm div(\nabla{\bf d}^\epsilon\odot\nabla{\bf d}^\epsilon)}{\Big\|}_{L^1(\dot{B}_{p,1}^{\frac{5}{2p}-\frac{5}{4}})}:=K_1+K_2.
\end{split}
\ee
Attention is now focused on bounding $K_1$ and $K_2$. For the estimate of $K_1$, we have
\be\label{5.49}
\begin{split}
&K_1\lesssim\|\nabla{\bf d}^\epsilon+\nabla{\bf d}\|_{L^2(\dot{B}_{2,1}^{1})}\|\nabla\bar{\bf d}^\epsilon\|_{L^2(\dot{B}_{p,1}^{\frac{5}{2p}-\frac{1}{4}})}\\[2mm]
&\quad \lesssim{\Big(}\|{\bf d}^\epsilon\|_{L^2(\dot{B}_{2,1}^{2})}+\|{\bf d}\|_{L^2(\dot{B}_{2,1}^{2})}{\Big)}\|\bar{\bf d}^\epsilon\|_{L^2(\dot{B}_{p,1}^{\frac{5}{2p}+\frac{3}{4}})}.
\end{split}
\ee
For $K_2$, it follows that
\be\label{5.50}
\begin{split}
&K_2\lesssim\epsilon\|b^\epsilon\|_{L^\infty(\dot{B}_{2,1}^{\frac{3}{4}+\frac{1}{2p}})}{\Big\|}
{\rm div(\nabla({\bf d}^\epsilon-\hat{\bf d})\odot\nabla({\bf d}^\epsilon-\hat{\bf d}))}{\Big\|}_{L^1(\dot{B}_{p,1}^{\frac{2}{p}-1})}\\[2mm]
&\quad\lesssim \epsilon(\epsilon\nu)^{-\frac{3}{4}-\frac{1}{2p}}\| b^\epsilon\|_{L^\infty(\dot{B}_{2,1}^{0})}^{\frac{1}{4}-\frac{1}{2p}}
(\epsilon\nu\| b^\epsilon\|_{L^\infty(\dot{B}_{2,1}^{1})})^{\frac{3}{4}+\frac{1}{2p}}
\|{\bf d}^\epsilon-\hat{\bf d}\|_{L^2(\dot{B}_{p,1}^{\frac{2}{p}+1})}^2\\[2mm]
&\quad\lesssim\epsilon^{\frac{1}{4}-\frac{1}{2p}}
\|b^\epsilon\|_{L^\infty(\tilde{B}_{\epsilon\nu}^{1,\infty})}
\|{\bf d}^\epsilon-\hat{\bf d}\|_{L^2(\dot{B}_{2,1}^{2})}^2
\lesssim (C_0^{\epsilon\nu})^3\epsilon^{\frac{1}{4}-\frac{1}{2p}}.
\end{split}
\ee
Therefore, in view of the estimates of heat equation, similar to \eqref{5.15}, we have
\be\label{5.51}
\begin{split}
&\|{\bf w}^\epsilon\|_{L^\infty(\dot{B}_{p,1}^{\frac{5}{2p}-\frac{5}{4}})}+\frac{3}{4}\mu\|{\bf w}^\epsilon\|_{L^1(\dot{B}_{p,1}^{\frac{5}{2p}+\frac{3}{4}})}\\[2mm]
&\quad\lesssim\|{\bf w}_0^\epsilon\|_{\dot{B}_{p,1}^{\frac{5}{2p}-\frac{5}{4}}}
+C_0^{\epsilon\nu}\epsilon^{\frac{1}{4}-\frac{1}{2p}}
+\int_0^t{\Big(}\|{\bf u}^\epsilon\|_{\dot{B}_{2,1}^{1}}^2+\|{\bf u}\|_{\dot{B}_{2,1}^{1}}^2{\Big)}\|{\bf w}^\epsilon\|_{\dot{B}_{p,1}^{\frac{5}{2p}-\frac{5}{4}}}dt^\prime\\[2mm]
&\quad\quad+{\Big(}\|{\bf d}^\epsilon\|_{L^2(\dot{B}_{2,1}^{2})}+\|{\bf d}\|_{L^2(\dot{B}_{2,1}^{2})}{\Big)}\|\bar{\bf d}^\epsilon\|_{L^2(\dot{B}_{p,1}^{\frac{5}{2p}+\frac{3}{4}})}.
\end{split}
\ee
  In order to bound $\|\bar{\bf d}^\epsilon\|_{L^2(\dot{B}_{p,1}^{\frac{5}{2p}+\frac{3}{4}})}$,  similar to \eqref{5.16}, we have
\be\label{5.52}
\begin{split}
&\|\bar{\bf d}^\epsilon\|_{L^\infty(\dot{B}_{p,1}^{\frac{5}{2p}-\frac{1}{4}})}+
\theta\|\bar{\bf d}^\epsilon\|_{L^1(\dot{B}_{p,1}^{\frac{5}{2p}+\frac{7}{4}})}\\[2mm]
&\quad\lesssim \|\bar{\bf d}_0^\epsilon\|_{\dot{B}_{p,1}^{\frac{5}{2p}-\frac{1}{4}}}
+\|{\bf u}^\epsilon\cdot\nabla\bar{{\bf d}}^\epsilon\|_{L^1(\dot{B}_{p,1}^{\frac{5}{2p}-\frac{1}{4}})}
+\|{\bf w}^\epsilon\cdot\nabla{\bf d}\|_{L^1(\dot{B}_{p,1}^{\frac{5}{2p}-\frac{1}{4}})}+\|\mathcal{Q}{\bf u}^\epsilon\cdot\nabla{\bf d}\|_{L^1(\dot{B}_{p,1}^{\frac{5}{2p}-\frac{1}{4}})}\\[2mm]
&\quad\quad+\||\nabla {\bf d}|^2
 \bar{\bf d}^\epsilon\|_{L^1(\dot{B}_{p,1}^{\frac{5}{2p}-\frac{1}{4}})}
+\||\nabla{\bf d}^\epsilon+\nabla{\bf d}||\nabla\bar{\bf d}^\epsilon|{\bf d}^\epsilon\|_{L^1(\dot{B}_{p,1}^{\frac{5}{2p}-\frac{1}{4}})}.
\end{split}
\ee
Next, we estimate the terms of the right hand of \eqref{5.52}
by Corollary \ref{co2.1} and Lemma \ref{le2.2} as follows.
\be\label{5.53}
\begin{split}
\|{\bf u}^\epsilon\cdot\nabla\bar{{\bf d}}^\epsilon\|_{L^1(\dot{B}_{p,1}^{\frac{5}{2p}-\frac{1}{4}})}
&\lesssim \|{\bf u}^\epsilon\|_{L^2(\dot{B}_{2,1}^{1})}
\|\nabla\bar{{\bf d}}^\epsilon\|_{L^2(\dot{B}_{p,1}^{\frac{5}{2p}-\frac{1}{4}})}\\[2mm]
&\lesssim \|{\bf u}^\epsilon\|_{L^2(\dot{B}_{2,1}^{1})}
\|\bar{{\bf d}}^\epsilon\|_{L^\infty(\dot{B}_{p,1}^{\frac{5}{2p}-\frac{1}{4}})}^{\frac{1}{2}}
\|\bar{{\bf d}}^\epsilon\|_{L^1(\dot{B}_{p,1}^{\frac{5}{2p}+\frac{7}{4}})}^{\frac{1}{2}},
\end{split}
\ee
\be\label{5.54}
\begin{split}
\|{\bf w}^\epsilon\cdot\nabla{\bf d}\|_{L^1(\dot{B}_{p,1}^{\frac{5}{2p}-\frac{1}{4}})}&\lesssim\|{\bf w}^\epsilon\|_{L^2(\dot{B}_{p,1}^{\frac{5}{2p}-\frac{1}{4}})}
\|\nabla{\bf d}\|_{L^2(\dot{B}_{2,1}^{1})}\\[2mm]
&\lesssim\|{\bf w}^\epsilon\|_{L^\infty(\dot{B}_{p,1}^{\frac{5}{2p}-\frac{5}{4}})}^{\frac{1}{2}}
\|{\bf w}^\epsilon\|_{L^1(\dot{B}_{p,1}^{\frac{5}{2p}+\frac{3}{4}})}^{\frac{1}{2}}
\|{\bf d}\|_{L^2(\dot{B}_{2,1}^{2})},
\end{split}
\ee
\be\label{5.55}
\begin{split}
\|\mathcal{Q}{\bf u}^\epsilon\cdot\nabla{\bf d}\|_{L^1(\dot{B}_{p,1}^{\frac{5}{2p}-\frac{1}{4}})}
&\lesssim\|\mathcal{Q}{\bf u}^\epsilon\|_{L^2(\dot{B}_{p,1}^{\frac{5}{2p}-\frac{1}{4}})}
\|\nabla{\bf d}\|_{L^2(\dot{B}_{2,1}^{1})}
\lesssim C_0^{\epsilon\nu}\epsilon^{\frac{1}{4}-\frac{1}{2p}}
\|{\bf d}\|_{L^2(\dot{B}_{2,1}^{2})},
\end{split}
\ee
\be\label{5.56}
\begin{split}
\||\nabla {\bf d}|^2\bar{\bf d}^\epsilon\|_{L^1(\dot{B}_{p,1}^{\frac{5}{2p}-\frac{1}{4}})}
&\lesssim\||\nabla {\bf d}|^2
 \|_{L^1(\dot{B}_{2,1}^{1})}
 \|\bar{\bf d}^\epsilon\|_{L^\infty(\dot{B}_{p,1}^{\frac{5}{2p}-\frac{1}{4}})}
\lesssim\|{\bf d}
 \|_{L^2(\dot{B}_{2,1}^{2})}^2
 \|\bar{\bf d}^\epsilon\|_{L^\infty(\dot{B}_{p,1}^{\frac{5}{2p}-\frac{1}{4}})},
\end{split}
\ee
\be\label{5.57}
\begin{split}
&\||\nabla{\bf d}^\epsilon+\nabla{\bf d}||\nabla\bar{\bf d}^\epsilon|{\bf d}^\epsilon\|_{L^1(\dot{B}_{p,1}^{\frac{5}{2p}-\frac{1}{4}})}\\[2mm]
&\quad\lesssim\||\nabla{\bf d}^\epsilon||\nabla\bar{\bf d}^\epsilon|{\bf d}^\epsilon\|_{L^1(\dot{B}_{p,1}^{\frac{5}{2p}-\frac{1}{4}})}+
\||\nabla{\bf d}||\nabla\bar{\bf d}^\epsilon|{\bf d}^\epsilon\|_{L^1(\dot{B}_{p,1}^{\frac{5}{2p}-\frac{1}{4}})}\\[2mm]
&\quad\lesssim\||\nabla{\bf d}^\epsilon|{\bf  d}^\epsilon\|_{L^2(\dot{B}_{2,1}^{1})}
\|\nabla\bar{\bf d}^\epsilon\|_{L^2(\dot{B}_{p,1}^{\frac{5}{2p}-\frac{1}{4}})}+
\||\nabla{\bf d}|{\bf d}^\epsilon\|_{L^2(\dot{B}_{2,1}^{1})}
\|\nabla\bar{\bf d}^\epsilon\|_{L^2(\dot{B}_{p,1}^{\frac{5}{2p}-\frac{1}{4}})}\\[2mm]
&\quad\lesssim\|\nabla{\bf d}^\epsilon\|_{L^2(\dot{B}_{2,1}^{1})}
\|{\bf d}^\epsilon\|_{L^\infty(\dot{B}_{2,1}^{1})}
\|\bar{\bf d}^\epsilon\|_{L^2(\dot{B}_{p,1}^{\frac{5}{2p}+\frac{3}{4}})}+\|\nabla{\bf d}\|_{L^2(\dot{B}_{2,1}^{1})}
\|{\bf d}^\epsilon\|_{L^\infty(\dot{B}_{2,1}^{1})}
\|\bar{\bf d}^\epsilon\|_{L^2(\dot{B}_{p,1}^{\frac{5}{2p}+\frac{3}{4}})}\\[2mm]
&\quad\lesssim{\Big(}\|{\bf d}^\epsilon\|_{L^2(\dot{B}_{2,1}^{2})}+
\|{\bf d}\|_{L^2(\dot{B}_{2,1}^{2})}{\Big)}
\|{\bf d}^\epsilon\|_{L^\infty(\dot{B}_{2,1}^{1})}
\|\bar{{\bf d}}^\epsilon\|_{L^\infty(\dot{B}_{p,1}^{\frac{5}{2p}-\frac{1}{4}})}^{\frac{1}{2}}
\|\bar{{\bf d}}^\epsilon\|_{L^1(\dot{B}_{p,1}^{\frac{5}{2p}+\frac{7}{4}})}^{\frac{1}{2}}.
\end{split}
\ee
It follows from the above estimates \eqref{5.52}-\eqref{5.57} that
\be\label{5.58}
\begin{split}
&\|\bar{\bf d}^\epsilon\|_{L^\infty(\dot{B}_{p,1}^{\frac{5}{2p}-\frac{1}{4}})}+
\theta\|\bar{\bf d}^\epsilon\|_{L^1(\dot{B}_{p,1}^{\frac{5}{2p}+\frac{7}{4}})}\\[2mm]
&\quad\lesssim \|\bar{\bf d}_0^\epsilon\|_{\dot{B}_{p,1}^{\frac{5}{2p}-\frac{1}{4}}}+\|{\bf u}^\epsilon\|_{L^2(\dot{B}_{2,1}^{1})}
\|\bar{{\bf d}}^\epsilon\|_{L^\infty(\dot{B}_{p,1}^{\frac{5}{2p}-\frac{1}{4}})}^{\frac{1}{2}}
\|\bar{{\bf d}}^\epsilon\|_{L^1(\dot{B}_{p,1}^{\frac{5}{2p}+\frac{7}{4}})}^{\frac{1}{2}}\\[2mm]
&\quad\quad+\|{\bf w}^\epsilon\|_{L^\infty(\dot{B}_{p,1}^{\frac{5}{2p}-\frac{5}{4}})}^{\frac{1}{2}}
\|{\bf w}^\epsilon\|_{L^1(\dot{B}_{p,1}^{\frac{5}{2p}+\frac{3}{4}})}^{\frac{1}{2}}
\|{\bf d}\|_{L^2(\dot{B}_{2,1}^{2})}\\[2mm]
&\quad\quad+C_0^{\epsilon\nu}\epsilon^{\frac{1}{4}-\frac{1}{2p}}
\|{\bf d}\|_{L^2(\dot{B}_{2,1}^{2})}+\|{\bf d}
 \|_{L^2(\dot{B}_{2,1}^{2})}^2
 \|\bar{\bf d}^\epsilon\|_{L^\infty(\dot{B}_{p,1}^{\frac{5}{2p}-\frac{1}{4}})}\\[2mm]
&\quad\quad+{\Big(}\|{\bf d}^\epsilon\|_{L^2(\dot{B}_{2,1}^{2})}+
\|{\bf d}\|_{L^2(\dot{B}_{2,1}^{2})}{\Big)}
\|{\bf d}^\epsilon\|_{L^\infty(\dot{B}_{2,1}^{1})}
\|\bar{{\bf d}}^\epsilon\|_{L^\infty(\dot{B}_{p,1}^{\frac{5}{2p}-\frac{1}{4}})}^{\frac{1}{2}}
\|\bar{{\bf d}}^\epsilon\|_{L^1(\dot{B}_{p,1}^{\frac{5}{2p}+\frac{7}{4}})}^{\frac{1}{2}}.
\end{split}
\ee
Now combining \eqref{5.51} with \eqref{5.58} and using Young inequality, we get
\be\nonumber
\begin{split}
&\|{\bf w}^\epsilon\|_{L^\infty(\dot{B}_{p,1}^{\frac{5}{2p}-\frac{5}{4}})}+\frac{\mu}{2}\|{\bf w}^\epsilon\|_{L^1(\dot{B}_{p,1}^{\frac{5}{2p}+\frac{3}{4}})}
+\|\bar{\bf d}^\epsilon\|_{L^\infty(\dot{B}_{p,1}^{\frac{5}{2p}-\frac{1}{4}})}+
\frac{\theta}{2}\|\bar{\bf d}^\epsilon\|_{L^1(\dot{B}_{p,1}^{\frac{5}{2p}+\frac{7}{4}})}\\[2mm]
&\quad\lesssim{\Big(}\|{\bf w}_0^\epsilon\|_{\dot{B}_{p,1}^{\frac{5}{2p}-\frac{5}{4}}}
+\|\bar{\bf d}_0^\epsilon\|_{\dot{B}_{p,1}^{\frac{5}{2p}-\frac{1}{4}}}{\Big)}
+C_0^{\epsilon\nu}\epsilon^{\frac{1}{4}-\frac{1}{2p}}\\[2mm]
&\quad\quad+\int_0^t{\Big(}\|{\bf u}^\epsilon(t^\prime)\|_{\dot{B}_{2,1}^{1}}^2+\|{\bf u}(t^\prime)\|_{\dot{B}_{2,1}^{1}}^2+\|{\bf d}(t^\prime)\|_{\dot{B}_{2,1}^{2}}^2{\Big)}\|{\bf w}^\epsilon(t^\prime)\|_{\dot{B}_{p,1}^{\frac{5}{2p}-\frac{5}{4}}}dt^\prime\\[2mm]
&\quad\quad+\int_0^t{\Big(}\|{\bf d}^\epsilon(t^\prime)\|_{\dot{B}_{2,1}^{2}}^2+\|{\bf d}(t^\prime)\|_{\dot{B}_{2,1}^{2}}^2+\|{\bf u}^\epsilon(t^\prime)\|_{\dot{B}_{2,1}^{1}}^2{\Big)}\|\bar{{\bf d}}^\epsilon(t^\prime)\|_{\dot{B}_{p,1}^{\frac{5}{2p}-\frac{1}{4}}}dt^\prime,
\end{split}
\ee
which together with Gronwall's lemma yields that
\be\nonumber
\begin{split}
&\|{\bf w}^\epsilon\|_{L^\infty(\dot{B}_{p,1}^{\frac{5}{2p}-\frac{5}{4}})}+\frac{\mu}{2}\|{\bf w}^\epsilon\|_{L^1(\dot{B}_{p,1}^{\frac{5}{2p}+\frac{3}{4}})}
+\|\bar{\bf d}^\epsilon\|_{L^\infty(\dot{B}_{p,1}^{\frac{5}{2p}-\frac{1}{4}})}+
\frac{\theta}{2}\|\bar{\bf d}^\epsilon\|_{L^1(\dot{B}_{p,1}^{\frac{5}{2p}+\frac{7}{4}})}\\[2mm]
&\quad\lesssim{\Big(}\|{\bf w}_0^\epsilon\|_{\dot{B}_{p,1}^{\frac{5}{2p}-\frac{5}{4}}}
+\|\bar{\bf d}_0^\epsilon\|_{\dot{B}_{p,1}^{\frac{5}{2p}-\frac{1}{4}}}{\Big)}
+C_0^{\epsilon\nu}\epsilon^{\frac{1}{4}-\frac{1}{2p}}.
\end{split}
\ee
Thus the proof of Theorem $\ref{th1.2}$ is completed.\hspace{8cm}$\Box$
 %%%%%%%%%%%%%
% SECTION 7%
%%%%%%%%%%%%%
%\section*{Acknowledgments}
%\begin{thebibliography}{}
%\bibliographystyle{abbrv}
%\setlength{\baselineskip}{12pt}
%\bibliography{biequnyi-ref}
%\end{thebibliography}

\end{document}